\documentclass{article}
\usepackage[
backend=biber,
datamodel=mydatamodel,
style=alphabetic,
natbib=true,
url=false,
doi=true,
eprint=true,
giveninits=true,
isbn=false,
maxbibnames=10,
]{biblatex}
\DeclareFieldFormat{mrnumber}{%
	MR\space
	\ifhyperref
	{\href{http://www.ams.org/mathscinet-getitem?mr=#1}{\nolinkurl{#1}}}
	{\nolinkurl{#1}}}
\DeclareFieldFormat{zbl}{%
	Zbl\space
	\ifhyperref
	{\href{https://zbmath.org/?q=an:#1}{\nolinkurl{#1}}}
	{\nolinkurl{#1}}}
\renewbibmacro*{doi+eprint+url}{%
	\iftoggle{bbx:doi}
	{\printfield{doi}}
	{}%
	\newunit\newblock
	\printfield{mrnumber}%
	\newunit\newblock
	\printfield{zbl}%
	\newunit\newblock
	\iftoggle{bbx:eprint}
	{\usebibmacro{eprint}}
	{}%
	\newunit\newblock
	\iftoggle{bbx:url}
	{\usebibmacro{url+urldate}}
	{}}
\usepackage[a4paper, margin=3cm]{geometry}
\usepackage{hyperref}
\usepackage{authblk}
\usepackage{parskip} 
\usepackage{amsmath}
\usepackage{amssymb}
\usepackage{amsthm}
\usepackage{mathtools}
\usepackage{xurl}
\usepackage{enumitem} 
\usepackage{graphicx} 
\usepackage{tikz} 
\usetikzlibrary{calc}
\usetikzlibrary {arrows.meta}
\usepackage{tkz-euclide}
\usepackage[ruled]{algorithm2e}
\usepackage{listings}

\usepackage[capitalize,nameinlink,noabbrev]{cleveref}
\renewcommand{\subset}{\subseteq}
\newcommand{\lcb}{\left\lbrace} 
\newcommand{\rcb}{\right\rbrace} 
\newcommand{\cb}[1]{\lcb #1 \rcb} 
\newcommand{\lab}{\left[} 
\newcommand{\rab}{\right]} 
\newcommand{\ab}[1]{\lab #1 \rab} 
\newcommand{\abOf}[1]{\!\ab{#1}} 
\newcommand{\lb}{\left(} 
\newcommand{\rb}{\right)} 
\newcommand{\br}[1]{\lb #1 \rb} 
\newcommand{\brOf}[1]{\!\br{#1}} 
\newcommand{\abs}[1]{\left| #1 \right|} 
\newcommand*{\E}{\mathbb{E}} 
\newcommand{\sizedMid}[2]{#1 \, \kern-\nulldelimiterspace\mathopen{}\left| \vphantom{#1}\,#2\right.\mathclose{}\kern-\nulldelimiterspace}

\DeclarePairedDelimiterX\Set[1]{\lbrace}{\rbrace}%
{  #1 }
\newcommand{\Ex}{\E\expectarg}
\DeclarePairedDelimiterX{\expectarg}[1]{[}{]}{%
	\ifnum\currentgrouptype=16 \else\begingroup\fi
	\activatebar#1
	\ifnum\currentgrouptype=16 \else\endgroup\fi
}
\newcommand{\innermid}{\nonscript\;\delimsize\vert\nonscript\;}
\newcommand{\activatebar}{%
	\begingroup\lccode`\~=`\|
	\lowercase{\endgroup\let~}\innermid 
	\mathcode`|=\string"8000
}
\newcommand*{\mc}[1]{\mathcal{#1}}

\newcommand*{\ms}[1]{\mathsf{#1}}

\newcommand{\N}{\mathbb{N}}

\newcommand{\R}{\mathbb{R}}
\newcommand{\Rp}{[0, \infty)} 
\newcommand{\Rpp}{(0, \infty)} 

\newcommand{\pr}{^\prime}
\newcommand{\prr}{^{\prime\prime}}
\newcommand{\prrr}{^{\prime\prime\prime}}
\def\integral from #1to #2of #3by #4;{\int_{#1}^{#2} \! #3 \mathrm{d}#4} %
\def\integralMeasure in #1of #2by #3of #4;{\int_{#1} \! #2{#4} #3{\mathrm{d}#4}} %
\def\mapping #1from #2to #3;{#1 \colon #2 \rightarrow #3}
\def\mappingDef #1from #2to #3maps #4to #5;{#1 \colon #2 \rightarrow #3,\ #4 \mapsto #5}
\def\seq #1by #2;{\br{#1}_{#2\in\N}}
\def\seqInText #1by #2;{(#1)_{#2\in\N}}
\newcommand{\innerProduct}[2]{\left\langle#1\,,\, #2\right\rangle}
\newcommand{\ip}[2]{\innerProduct{#1}{#2}}
\newcommand{\lebesgue}{\mathcal{L}}
\newcommand{\lebesguePow}[1]{\lebesgue^{#1}}

\newcommand{\sgn}{\mathsf{sgn}} 
\newcommand{\dl}{\mathrm{d}}
\def\converges for #1to #2;{\xrightarrow{#1} #2}
\def\convergesAlmostSurely for #1to #2;{\xrightarrow{#1}_{\mathsf{fs}} #2}
\def\convergesInProbability for #1to #2;{\xrightarrow{#1}_{\mathsf{p}} #2}
\def\convergesInL #1for #2to #3;{\xrightarrow{#2}_{\lebesguePow{#1}} #3}

\newcommand{\ind}{\mathbf{1}}
\newcommand{\indOf}[1]{\ind_{\!#1}}%
\newcommand{\normof}[1]{\Vert #1 \Vert}
\newcommand{\normOf}[1]{\left\Vert #1 \right\Vert}
\newcommand{\equationFullstop}{\, .}
\newcommand{\eqfs}{\equationFullstop}
\newcommand{\equationComma}{\, ,}
\newcommand{\eqcm}{\equationComma}

\newcommand{\euler}{\mathrm{e}}
\DeclareMathOperator*{\argmin}{arg\,min}

\newcommand{\ol}[2]{\overline{#1,#2}}

\newcommand{\cato}{$\ms{CAT(0)}$}
\newcommand{\mtr}{\varphi}
\newcommand{\tran}{\tau}
\newcommand{\trans}{\tran_{\!\sqrt{}}}
\newcommand{\dtrans}{\trans\pr}
\newcommand{\ddtrans}{\trans\prr}
\newcommand{\dddtrans}{\trans\prrr}
\newcommand{\Tran}{\mc{T}}
\newcommand{\dtran}{\tran\pr}
\newcommand{\ddtran}{\tran\prr}
\newcommand{\dddtran}{\tran\prrr}
\newcommand{\ctransymbol}{L}
\newcommand{\ctrano}[1]{\ctransymbol_{#1}}
\newcommand{\ctranobest}[1]{\ctransymbol_{#1}^{*}}
\newcommand{\ctranbest}{\ctranobest{\tran}}
\newcommand{\ctran}{\ctrano{\tran}}
\newcommand{\ctransqr}{\ctrano{}}
\newcommand{\setcc}{\mc S}
\newcommand{\setd}[1]{\mc C^{#1}}
\newcommand{\setccd}[1]{\mc S^{#1}}
\newcommand{\setcco}{\mc S_0}
\newcommand{\setccod}[1]{\mc S_0^{#1}}

\newcommand{\lhs}{\ell_{\tran}}
\newcommand{\ifu}{F_{\tran}}
\newcommand{\incr}{\nearrow}
\newcommand{\decr}{\searrow}
\newcommand{\posi}{\geq0}
\newcommand{\nega}{\leq0}
\newcommand{\conc}{\smallfrown}
\newcommand{\conv}{\smallsmile}

\definecolor{colQuadDiag}{HTML}{D81B60}
\definecolor{colQuadSidePos}{HTML}{1E88E5}
\definecolor{colQuadSideNeg}{HTML}{004D40}

\newbox{\myorcidthanksbox}
\sbox{\myorcidthanksbox}{\large\includegraphics[height=1.8ex]{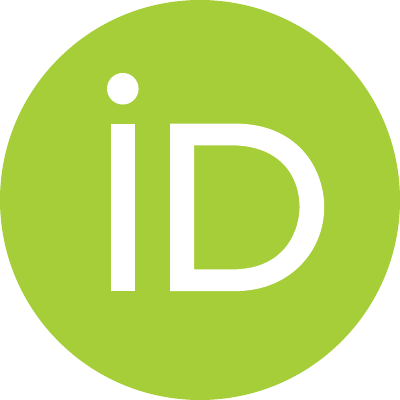}}
\newcommand{\orcidthanks}[1]{%
	\href{https://orcid.org/#1}{\raisebox{-0.5ex}{\usebox{\myorcidthanksbox}}\,#1}}
\usepackage{thmtools}
\def\postBoxSkip{1.0ex}
\def\postBoxSkipCmd{\vskip\postBoxSkip}
\def\preBoxSkip{1.0ex}
\def\preBoxSkipCmd{\vskip\preBoxSkip}
\declaretheoremstyle[
	bodyfont=\normalfont,
	postfoothook={\postBoxSkipCmd},
	preheadhook={\preBoxSkipCmd},
	mdframed={
		backgroundcolor = black!2,
		startcode={},
}]{ruledBoxStyle}
\declaretheoremstyle[
	bodyfont=\normalfont,
	postfoothook={\postBoxSkipCmd},
	preheadhook={\preBoxSkipCmd},
	mdframed={
		backgroundcolor=white,
}]{ruledBoxStyleWhite}
\declaretheoremstyle[
	bodyfont=\normalfont,
	postfoothook={\postBoxSkipCmd},
	preheadhook={\preBoxSkipCmd},
	mdframed={
		backgroundcolor=black!2,
		linecolor = black!2,
		tikzsetting = {
			draw = black,
			line width = 2pt,%
			dashed,%
			dash pattern = on 10pt off 3pt
		},
}]{dashedBoxStyle}
\declaretheoremstyle[
	bodyfont=\normalfont,
	postfoothook={\postBoxSkipCmd},
	preheadhook={\preBoxSkipCmd},
	mdframed={
		linecolor = white,
		startcode={},
		tikzsetting = {
			draw = black,
			line width = 1pt,%
			loosely dotted,
		},
	}
]{dashedStyle}
\declaretheoremstyle[
	bodyfont=\normalfont,
	postfoothook={\postBoxSkipCmd},
	preheadhook={\preBoxSkipCmd},
	mdframed={
		linecolor = black,
		innerlinewidth=1pt,outerlinewidth=1pt,
		middlelinewidth=1pt,
		linecolor=black,middlelinecolor=white,
		startcode={},
	}
]{doubleStyle}
\declaretheoremstyle[
	bodyfont=\normalfont,
	postfoothook={\postBoxSkipCmd},
	preheadhook={\preBoxSkipCmd},
	mdframed={
		backgroundcolor = black!4,
		linecolor = black!4,
		startcode={},
}]{boxStyle}
\declaretheoremstyle[
	headfont=\normalfont\itshape, 
	notefont=\normalfont\itshape, 
	notebraces={}{},
	bodyfont=\normalfont,
	qed=\qedsymbol,
	numbered=no,
	headindent=0pt,
	postheadspace=1ex,
	name={Proof},
	postheadhook={},
	mdframed={
		hidealllines = true,
		innerrightmargin = 0pt,
		innerleftmargin = 0pt,
		innertopmargin = 0pt,
		innerbottommargin = 0pt,
		leftmargin = 0pt,
		rightmargin = 0pt,
	}
]{proofStyle}
\declaretheoremstyle[
	bodyfont=\normalfont,
	postfoothook={\postBoxSkipCmd},
	preheadhook={\preBoxSkipCmd},
	mdframed={
		backgroundcolor = white,
		linecolor = black,
		startcode={},
		leftline = false,
		rightline = false,
}]{tobBottomStyle}
\declaretheoremstyle[
bodyfont=\normalfont,
]{standardStyle}
\declaretheorem[style=ruledBoxStyle,name=Definition]{definition}
\declaretheorem[style=ruledBoxStyle,name=Lemma,numberlike=definition]{lemma}
\declaretheorem[style=ruledBoxStyle,name=Proposition,numberlike=definition]{proposition}

\declaretheorem[style=ruledBoxStyle,name=Theorem,numberlike=definition]{theorem}
\declaretheorem[style=ruledBoxStyle,name=Corollary,numberlike=definition]{corollary}
\declaretheorem[style=ruledBoxStyle,name=Theorem,numbered=no]{theorem*}
\declaretheorem[style=boxStyle,name=Remark,numberlike=definition]{remark}

\makeatletter
\let\proof\@undefined
\let\endproof\@undefined
\makeatother
\declaretheorem[style=proofStyle]{proof}

\newcounter{subExample}%
\renewcommand{\thmcontinues}[1]{Teil \arabic{subExample}} 
\title{Quadruple Inequalities:\\Between Cauchy--Schwarz and Triangle}
\date{}
\author[1,2]{Christof Sch\"otz\thanks{math@christof-schoetz.de, \orcidthanks{0000-0003-3528-4544}}}
\affil[1]{Potsdam Institute for Climate Impact Research}
\affil[2]{Technical University of Munich}

\begin{document}
\maketitle
\begin{abstract}
	We prove a set of inequalities that interpolate the Cauchy--Schwarz inequality and the triangle inequality. Every nondecreasing, convex function with a concave derivative induces such an inequality. They hold in any metric space that satisfies a metric version of the Cauchy--Schwarz inequality, including all CAT(0) spaces and, in particular, all Euclidean spaces. Because these inequalities establish relations between the six distances of four points, we call them quadruple inequalities. In this context, we introduce the quadruple constant --- a real number that quantifies the distortion of the Cauchy--Schwarz inequality by a given function. Additionally, for inner product spaces, we prove an alternative, more symmetric version of the quadruple inequalities, which generalizes the parallelogram law.

\end{abstract}
\tableofcontents
\section{Introduction}
\subsection{Relating Cauchy--Schwarz and Triangle}
The Cauchy--Schwarz inequality states that in any inner product space $(V, \langle\cdot,\cdot\rangle)$, we have
\begin{equation}\label{eq:csi}
	\langle u,v \rangle \leq \normof{u} \normof{v}
\end{equation}
for all $u,v \in V$, where $ \normof{u} = \sqrt{\langle u,u \rangle}$. In any metric space $(\mc Q, d)$ the triangle inequality 
\begin{equation}\label{eq:tri}
	\ol yz \leq \ol yp + \ol pz
\end{equation}
is true for all $p,y,z\in\mc Q$, where we use the short notation $\ol yz := d(y,z)$. These two inequalities can be connected as follows: Consider the inequality
\begin{equation}\label{eq:quadtran}
	\tran(\ol yq)  - \tran(\ol yp) - \tran(\ol zq) + \tran(\ol zp) \leq \ctran\,  \ol qp \, \dtran(\ol yz)
\end{equation}
for $y,z,q,p\in\mc Q$, a differentiable function $\tran\colon[0,\infty)\to\R$ with derivative $\dtran$, and a constant $\ctran \in\Rp$. We call \eqref{eq:quadtran} a \textit{quadruple inequality} \cite{schoetz19} as it establishes a relationship between the six distances among four points, see \cref{fig:fourpoints}. If we plug the identity $\tran = \tran_1 := (x\mapsto x)$ and $\ctran = 2$ into \eqref{eq:quadtran}, we obtain
\begin{equation}\label{eq:quad1}
	\ol yq  - \ol yp - \ol zq + \ol zp \leq 2 \, \ol qp
	\eqfs
\end{equation}
The triangle inequality \eqref{eq:tri} implies \eqref{eq:quad1}. Furthermore, in a symmetric distance space $(\mc Q, d)$ \cite{deza16}, where $d$ does not necessarily fulfill the triangle inequality, \eqref{eq:quad1} also implies \eqref{eq:tri} by setting $z=q$. 
Next, let us evaluate \eqref{eq:quadtran} with $\tran = \tran_2 :=(x\mapsto x^2)$ and $\ctran = 1$. We get
\begin{equation}\label{eq:quad2}
	\ol yq^2 - \ol yp^2 - \ol zq^2 + \ol zp^2 \leq 2\, \ol qp\, \ol yz
	\eqfs
\end{equation}
If we assume that the metric space $(\mc Q, d)$ is induced by an inner product space $(V, \langle\cdot,\cdot\rangle)$, i.e., $\mc Q = V$ and $d(q,p) = \normof{q-p}$, then \eqref{eq:quad2} becomes
\begin{equation}
	2 \ip{q-p}{z-y} \leq 2 \normof{q-p}\normof{y-z}
	\eqfs
\end{equation}
Thus, in this case, \eqref{eq:quad2} is equivalent to \eqref{eq:csi}. Hence, we can consider \eqref{eq:quad2} to be a generalization of the Cauchy--Schwarz inequality to metric spaces. Equation \eqref{eq:quad2} is not true in every metric space. But it plays an important role in the study of geodesic metric spaces. A geodesic metric space is a metric space such that any two points $q,p\in\mc Q$ can be joined by a geodesic, i.e., a curve of length $\ol qp$.
Equation \eqref{eq:quad2} is well-known to hold in non-positively curved geodesic spaces, which are called \cato\ spaces or, if they are complete, Hadamard spaces or global NPC spaces. In these spaces, \eqref{eq:quad2} is also known as \textit{four point cosq condition} \cite{berg08} or \textit{Reshetnyak's Quadruple Comparison} \cite[Proposition 2.4]{sturm03}. Furthermore, a geodesic space is \cato\ if it fulfills \eqref{eq:quad2} \cite[Theorem 1]{berg08}. 
\begin{figure}
	\begin{center}
		\begin{tikzpicture}[thick]
			\coordinate (y) at (-1,0);
			\coordinate (q) at (7,0);
			\coordinate (z) at (1.5,5);
			\coordinate (p) at (2,2);
			
			\draw[colQuadSidePos] (y) -- node[below] {$\ol yq$} (q);
			\draw[colQuadSideNeg] (q) -- node[right] {$\ol zq$} (z);
			\draw[colQuadSidePos] (z) -- node[right] {$\ol zp$} (p);
			\draw[colQuadSideNeg] (p) -- node[right] {$\ol yp$} (y);
			\draw[colQuadDiag] (y) -- node[left] {$\ol yz$} (z);
			\draw[colQuadDiag] (q) -- node[above] {$\ol qp$} (p);
			
			\fill (y) circle[radius=2pt] node[below] {$y$};
			\fill (q) circle[radius=2pt] node[below] {$q$};
			\fill (z) circle[radius=2pt] node[above] {$z$};
			\fill (p) circle[radius=2pt] node[left] {$p$};
		\end{tikzpicture}
	\end{center}
	\caption{Four points and their six distances.}\label{fig:fourpoints}
\end{figure}
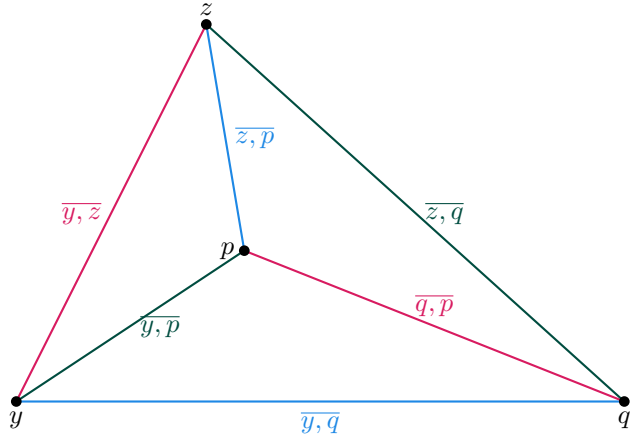
\subsection{Contributions}
The functions $\tran_1, \tran_2$ are both nondecreasing, convex, and have a concave derivative. They can be considered as edge cases of all functions with these properties: As a linear function, $\tran_1$ can be thought of as ``least convex'' of all convex functions. Similarly, $\tran_2$, which has a linear and strictly increasing derivative, is a ``most convex'' function among all functions with a concave derivative. As our main result, we show that in all metric spaces with the property \eqref{eq:quad2}, inequality \eqref{eq:quadtran} is true for all functions ``between" $\tran_1$ and $\tran_2$, i.e., for all nondecreasing, convex functions with a concave derivative. In this sense, we ``interpolate" the triangle and the Cauchy--Schwarz inequality.
\begin{theorem}\label{thm:main}
	Let $(\mc Q, d)$ be a metric space.
	Let $y,z,q,p\in\mc Q$. Assume
	\begin{equation*}
		\ol yq^2 - \ol yp^2 - \ol zq^2 + \ol zp^2 \leq 2 \, \ol qp\, \ol yz
		\eqfs
	\end{equation*}
	Let $\tran\colon[0,\infty)\to\R$ be differentiable. Assume $\tran$ is nondecreasing, convex and has a concave derivative $\dtran$. Then
	\begin{equation}\label{eq:main}
		\tran(\ol yq)  - \tran(\ol yp) - \tran(\ol zq) + \tran(\ol zp) \leq 2 \, \ol qp \, \dtran(\ol yz)\eqfs
	\end{equation}
\end{theorem}
We emphasize that \eqref{eq:quad2} as condition in \cref{thm:main} is only required for the specific four points $y,z,q,p\in\mc Q$ in the given order.
We call functions $\tran$ that satisfy \eqref{eq:quadtran} when \eqref{eq:quad2} is true quadruple transformations:
\begin{definition}[Quadruple transformation]
	Let $\tran\colon(0,\infty)\to\R$ be differentiable. We call $\tran$ a \textit{quadruple transformation} if there is a constant  $\ctran\in\Rp$ such that the following condition holds:
	For every metric space $(\mc Q, d)$ and pairwise distinct points $y,z,q,p\in\mc Q$ such that
	\begin{equation*}
		\ol yq^2 - \ol yp^2 - \ol zq^2 + \ol zp^2 \leq 2\, \ol qp\, \ol yz\eqcm
	\end{equation*}
	we also have
	\begin{equation*}
		\tran(\ol yq)  - \tran(\ol yp) - \tran(\ol zq) + \tran(\ol zp) \leq \ctran\,  \ol qp \, \dtran(\ol yz)
		\eqfs
	\end{equation*}
	Define the \textit{quadruple constant} of a quadruple transformation $\tran$ as the smallest possible choice of $\ctran$ and denote it as $\ctranbest$.
	Let $\Tran$ denote the \textit{set of all quadruple transformations}. 
\end{definition}
Note that $L^*_{\tran_2} = 1$ and $\tran = \tran_2$ in \eqref{eq:quadtran} yields the Cauchy--Schwarz inequality. Hence, we can interpret the quadruple constant $\ctranbest$ as a factor describing the distortion of the Cauchy--Schwarz inequality induced by applying $\tran$ instead of $\tran_2$. 

Let $\setcc$ be the set of all nondecreasing, convex, and differentiable functions $\tran\colon[0,\infty)\to\R$ with concave derivative.
Then, by \cref{thm:main}, $\setcc \subset \Tran$.
We show that, for any $\tran\in\setcc$, the quadruple constant $\ctranbest$ is in the interval $[1, 2]$, see \cref{lmm:constLowerBoundOne}. Further lower bounds on $\ctranbest$ are discussed in \cref{lmm:constLowerBounds}. Slightly stronger versions of \cref{thm:main} are presented in \cref{thm:main:param} and \cref{thm:strong}.

Let $\setcco$ be the set of functions in $\setcc$ with $\tran(0)=0$. For $\tran\in\setcco$, the right-hand side of \eqref{eq:main} can be bounded by $2\tran(\ol qp) + 2\tran(\ol yz)$, see \cref{cor:rhsbound}. In inner product spaces, we derive a stronger upper bound:
\begin{theorem}\label{thm:ipsymm}
	Let $(V, \ip\cdot\cdot)$ be an inner product space with induced metric $d$. Let $y,z,q,p\in V$. Let $\tran\in\setcco$. Then
	\begin{equation}\label{eq:symm:tran}
		\tran(\ol yq)  - \tran(\ol yp) - \tran(\ol zq) + \tran(\ol zp) \leq \tran(\ol qp) + \tran(\ol yz)
		\eqfs
	\end{equation}
\end{theorem}
This can be viewed as a generalization of the \emph{parallelogram law} in inner product spaces, $\normof{u+v}^2 + \normof{u-v}^2 \leq 2 \normof{u}^2 + 2\normof{v}^2$ for $u,v\in V$, as it implies
\begin{equation}
	\tran(\normof{u+v}) + \tran(\normof{u-v}) \leq  2 \tran(\normof{u}) + 2 \tran(\normof{v}) 
\end{equation}
for $\tran\in\setcco$, see \cref{cor:ip}.
\subsection{Related Literature}
For a history of the Cauchy--Schwarz inequality and many of its extensions, \cite{steele04} is highly recommended. The book \cite{deza16} is an excellent reference for metric related concepts.
\subsubsection{Convex Analysis}
\cref{thm:ipsymm} is related to \textit{Karamata's inequality} \cite{karamata32}: Let $f \colon \R \to\R$ be a convex and nondecreasing function. Let $a_1,\dots, a_n, b_1, \dots, b_n \in \R$ with 
\begin{equation}\label{eq:majorization}
	a_1 \geq \dots \geq a_n\eqcm \qquad b_1 \geq \dots \geq b_n\eqcm \qquad  \sum_{i=1}^k a_i \leq \sum_{i=1}^k b_i
\end{equation}
for $k = 1,\dots, n$. Then
\begin{equation}
	\sum_{i=1}^n f(a_i) \leq \sum_{i=1}^n f(b_i) \eqfs
\end{equation}
If we set $f = \tran$, $n = 4$, $a_1 = \ol yq$, $a_2 = \ol zp$, $a_3 = a_4 = 0$, $b_1 = \ol qp$, $b_2 = \ol yz$, $b_3 = \ol yp$, $b_4 = \ol zq$, then Karamata's inequality proves \cref{thm:ipsymm} for configurations of distances that fulfill \eqref{eq:majorization}. But this does not cover all cases.	
\subsubsection{Quadruple Inequality}
\cref{thm:main} extends \cite[Theorem 3]{schoetz19}, which states that $\tran_\alpha := (x \mapsto x^\alpha)$ with $\alpha\in[1,2]$ is a quadruple transformation and, together with \cite[Appendix E]{schoetz19}, implies $\ctranobest{\tran_\alpha} = 2^{2-\alpha}$. Moreover, in \cite[Appendix E]{schoetz19} it is shown that $\alpha\in[1,2]$ are the only positive real exponents with $\tran_\alpha\in\Tran$. Note that $\tran_\alpha\in\setcc$ for $\alpha\in[1,2]$. The proof of \cref{thm:main} requires new ideas compared to the one of \cite[Theorem 3]{schoetz19}, e.g., as we cannot take derivatives with respect to $\alpha$. Our generalization to all functions $\tran\in\setcc$ comes at the cost of a larger constant on the right-hand side: \cref{thm:main} applied to $\tran = \tran_\alpha$ yields a constant factor $\ctrano{\tran_\alpha} = 2$, which is strictly greater than $\ctranobest{\tran_\alpha}$ for $\alpha>1$. 
\subsubsection{Metric Geometry}
Aside from \cato\ spaces briefly discussed above, some further ideas in metric geometry seem relevant in the context of quadruple inequalities.

A function $\mtr\colon\Rp\to\Rp$ is called \textit{metric preserving}, if $\mtr \circ d$ is a metric for any metric space $(\mc Q, d)$. See \cite{corazza99} for an overview. As the quadruple inequality \eqref{eq:quadtran} with $\tran = \tran_2$ is a condition for \cref{thm:main}, we may think of the main result as stating that a Cauchy--Schwarz-like inequality is preserved under transformation with $\tran$. But note that the right-hand side of \eqref{eq:quadtran} is not written in terms of $\tran \circ d$.

A metric space $(\mc Q, d)$ has the \textit{Euclidean $k$-point property} \cite[Definition 50.1]{blumenthal70} if any $k$-tuple of points in $\mc Q$ has an isometric embedding in the Euclidean space $\R^{k-1}$. If $(\mc Q, d)$ has the Euclidean $4$-point property, then \eqref{eq:quad2} is fulfilled. For $\gamma\in\Rpp$, let $\mtr_\gamma(x) := x^\gamma$. This function is metric preserving for $\gamma \leq 1$. According to \cite[Theorem 52.1]{blumenthal70}, $(\mc Q, \mtr_\gamma \circ d)$ has the Euclidean $4$-point property for all $\gamma \leq 1/2$. Furthermore, $\gamma = 1/2$ is the largest exponent with this property. Thus, $(\mc Q, \mtr_\gamma \circ d)$ fulfills \eqref{eq:quad2} for $\gamma\in(0,1/2]$. In particular,
\begin{equation}\label{eq:symm:blumi}
	\ol yq^{2\gamma} - \ol yp^{2\gamma} - \ol zq^{2\gamma} + \ol zp^{2\gamma} \leq 2 \,\ol qp^\gamma \,\ol yz^\gamma
	\eqfs
\end{equation}
As $\tilde d = d^{2\gamma}$ is a metric --- $x\mapsto x^{2\gamma}$ is metric preserving for $\gamma\in(0,1/2]$ --- we obtain from \eqref{eq:quad1},
\begin{equation}
	\ol yq^{2\gamma} - \ol yp^{2\gamma} - \ol zq^{2\gamma} + \ol zp^{2\gamma} \leq 2 \min\brOf{\ol qp^{2\gamma}, \ol yz^{2\gamma}}
	\eqcm
\end{equation}
which also implies \eqref{eq:symm:blumi}.

The Euclidean $4$-point property can be weakened for \cato\ spaces. A metric space $(\mc Q, d)$ fulfills the \cato\ $4$-point condition \cite[Definition II.1.10]{bridson99} if, for all $y,z,q,p\in\mc Q$, there are $\bar y,\bar z,\bar q,\bar p\in\R^2$ such that 
\begin{align*}
	\ol yq	&= \normof{\bar y - \bar q}\eqcm &
	\ol yp	&= \normof{\bar y - \bar p}\eqcm &
	\ol zq	&= \normof{\bar z - \bar q}\eqcm &
	\ol zp	&= \normof{\bar z - \bar p}\eqcm &
	\\
	&&
	\ol qp	&\leq \normof{\bar q - \bar p}\eqcm &
	\ol yz	&\leq \normof{\bar y - \bar z}\eqfs &
	&&
\end{align*}
Every \cato\ space fulfills the \cato\ four-point condition, see \cite{reshetnyak68} or \cite[Proposition II.1.11]{bridson99}. 

Another famous $4$-point property is \textit{Ptolemy's inequality}: A metric space $(\mc Q, d)$ is called \textit{Ptolemaic} if, for all $y,z, q,p\in\mc Q$, we have
\begin{equation}\label{eq:ptolemy}
	\ol yq \, \ol zp + \ol yp\, \ol zq \leq \ol qp \, \ol yz 
	\eqfs
\end{equation}
Every inner product space is Ptolemaic. If a normed vector space is Ptolemaic, then it is an inner product space. All \cato\ spaces are Ptolemaic. A complete Riemannian manifold is Ptolemaic if and only if it is \cato\ \cite[Theorem 1.1]{buckley09}. Each geodesic metric space satisfying the $\tran_2$-quadruple inequality is Ptolemaic, but a geodesic Ptolemaic metric space is not necessarily \cato\ \cite{foertsch07,berg08}.

Strongly related to \cref{thm:ipsymm} is the concept of roundness of a metric space: A value $\alpha\in\Rpp$ is called \textit{roundness exponent} of a metric space $(\mc Q, d)$ if, for all $y,z,q,p\in\mc Q$,
\begin{equation}\label{eq:symm:pow}
	\ol yq^\alpha  - \ol yp^\alpha - \ol zq^\alpha + \ol zp^\alpha \leq \ol qp^\alpha + \ol yz^\alpha
	\eqfs
\end{equation}
Let $R = R(\mc Q, d)$ be the set of all roundness exponents of $(\mc Q, d)$. The \textit{roundness} $r = r(\mc Q, d)$ of $(\mc Q, d)$ is the supremum of the roundness exponents $r := \sup R$. By the triangle inequality and the metric preserving property of $(x \mapsto x^\alpha)$ for $\alpha \in(0,1]$, we have $(0, 1] \subset R$ for all metric spaces. The functions spaces $L_p(0,1)$ have roundness $p$ for $p\in[1,2]$ \cite{enflo70lp}. For a geodesic metric space, roundness $r=2$ is equivalent to being \cato, see \cite[Remark 7]{berg08}. A metric space is called \textit{ultrametric} if the triangle inequality can be strengthened to $\ol yz \leq \max(\ol yp, \ol zp)$ for all points $y,z,p$. Every ultrametric space can be isometrically embedded in a Hilbert space, see, e.g., \cite[Corollary 5.4]{faver14}. A metric space is ultrametric if and only if $r = \infty$, \cite[Theorem 5.1]{faver14}. Then $R = (0, \infty)$, \cite[Proposition 2.7]{faver14}. In general, $R$ is not necessarily an interval \cite[Remark p.\ 254]{enflo70}. But if $(\mc Q, d)$ is a (subset of a) Banach space with the metric $d$ induced by its norm, then $R = (0, r]$ with $r \in [1,2]$, \cite[Proposition 4.1.2]{enflo70}. In particular, \eqref{eq:symm:pow} holds for $\alpha\in(0,2]$ in all inner product spaces. A metric space is called \textit{additive} if 
\begin{equation}
	\ol yq + \ol zp \leq \max(\ol yp + \ol zq, \ol qp + \ol yz)
\end{equation}
for all points $y,z,q,p$. Every ultrametric space is additive. Every additive metric space is Ptolemaic. Additive metric spaces have roundness $r \geq 2$ \cite[Proposition 4.1]{faver14}.
\subsubsection{Martingale Theory}
Nondecreasing, convex functions with concave derivative play an important role in the Topchii--Vatutin inequality of martingales, see \cite[Theorem 2]{topchii98} and \cite{alsmeyer03}: For a suitably integrable martingale $(M_n)_{n\in\N_0}$, we have
\begin{equation}
	\Ex{\tran(\abs{M_n}) - \tran(\abs{M_0})} \leq 2 \sum_{k=1}^n \Ex{\tran(\abs{M_k-M_{k-1}})}
\end{equation}
for all $\tran\in\setcco$, where $\Ex{\cdot}$ denotes the expectation. In this context, the functions $\tran\in\setcco$ are named \textit{weakly convex}.
Moreover, \cite[Lemma 6]{topchii98} gives a weaker version of \cref{thm:ipsymm}:
Let $\tran \in \setcco$. For $a,b\in\Rp$ with $a\geq b$, it was shown that $\tran(a+b) + \tran(a-b) \leq 2 \tran(a) + 2 \tran(b)$.
\subsubsection{Statistics}
\cref{thm:main} can be applied to prove rates of convergence for certain kinds of means \cite{schoetz19}: We may want to calculate a mean value of some sample points in a metric spaces. One candidate for this is the \textit{Fréchet mean} \cite{frechet48}, also called \textit{barycenter}. It is the (set of) minimizer(s) of the squared distance to the sample points. If $Y$ is a random variable with values in a metric space $(\mc Q, d)$, the Fréchet mean is $\argmin_{q\in\mc Q}\Ex{\ol Yq^2}$, where we assume $\Ex{\ol Yq^2}<\infty$ for all $q\in\mc Q$. Similarly, one can define the Fréchet median \cite{fletcher09} as $\argmin_{q\in\mc Q}\Ex{\ol Yq}$, or a more general $\tran$-Fréchet mean \cite{schoetz22} as $\argmin_{q\in\mc Q}\Ex{\tran(\ol Yq)}$ for functions $\tran\colon\Rp\to\R$. Given a sequence of independent random variables $Y_1, Y_2, \dots$ with the same distribution as $Y$, a standard task in statistics is to bound the distance between the sample statistics and its corresponding population version. In our case, assume the $\tran$-Fréchet mean is unique and define
\begin{align*}
	m &:= \argmin_{q\in\mc Q}\Ex{\tran(\ol Yq)}\eqcm
	&
	\hat m_n &:= \argmin_{q\in\mc Q} \frac 1n \sum_{i=1}^n \tran(\ol {Y_i}q)\eqfs
\end{align*}
We want to bound $\ol {\hat m_n}{m}$ depending on $n$. One can employ quadruple inequalities such as \eqref{eq:quadtran} to obtain a suitable upper bound \cite[Theorem 1]{schoetz19}. This approach is particularly useful, if we do not want to make the assumption that the diameter of the metric space $\sup_{q,p\in\mc Q} \ol qp$ is finite. With \cref{thm:main}, one can obtain such a bound for $\tran$-Fréchet means with $\tran\in\setcc$ (under some conditions). We emphasize that this is only possible with \eqref{eq:quadtran} and not with \eqref{eq:symm:tran}. Noteworthy examples of $\tran\in\setcc$ in this context, aside from $\tran = \tran_\alpha$, are the Huber loss $\tran_{\ms H,\delta}$ \cite{huber64} and the Pseudo-Huber loss $\tran_{\ms{pH},\delta}$ \cite{charbonnier94} for $\delta\in\Rpp$,
 \begin{align*}
	\tran_{\ms H,\delta}(x) &:= 
		\begin{cases}
			\frac12 x^2 & \text{ for } x \leq \delta\eqcm\\
			\delta(x - \frac12 \delta) & \text{ for } x > \delta\eqcm\\
		\end{cases}
	&
	\tran_{\ms{pH},\delta}(x) &:= \delta^2 \br{\sqrt{1 + \frac{x^2}{\delta^2}} - 1}
	\eqcm
\end{align*}
as well as $x\mapsto \ln(\cosh(x))$ \cite{green90}.
These functions are of great interest in robust statistics and image processing as their respective minimizers combine properties of the classical mean ($\tran_2$-Fréchet mean) and the median ($\tran_1$-Fréchet mean).
\subsection{Outline}
In the remaining sections, we first discuss the set $\Tran$, i.e., the set of quadruple transformations, see section \ref{sec:quad}. We continue with a discussion of the set $\setcc$, i.e., nondecreasing, convex functions with concave derivative, in section \ref{sec:cc}. Thereafter, we prove our main result, i.e., $\setcc \subset \Tran$. The basic ideas of the proof and variations of the main result are presented in section \ref{sec:proofoutline}. The technical details can be found in appendix \ref{app:mainproof} and \ref{app:aux}. The proof of \cref{thm:ipsymm} can be found in appendix \ref{app:proofsymm}. In section \ref{sec:impl} we discuss implications of the main results and open questions.

\section{Quadruple Transformations}\label{sec:quad}
We explore some properties of quadruple functions $\tran \in\Tran$ and their quadruple constant $\ctranbest$.
\subsection{Properties}
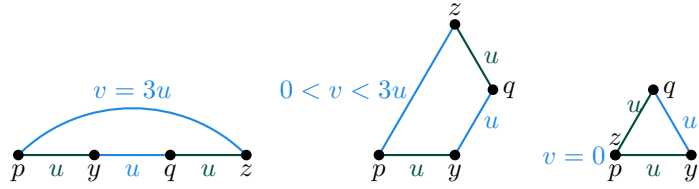
\begin{figure}
	\begin{center}
		\begin{tikzpicture}[thick]
			\coordinate (y) at (1,0);
			\coordinate (q) at (2,0);
			\coordinate (z) at (3,0);
			\coordinate (p) at (0,0);
			
			\draw[colQuadSidePos] (y) -- node[below] {$u$} (q);
			\draw[colQuadSideNeg] (q) -- node[below] {$u$} (z);
			\draw[colQuadSidePos] (z) to[out=135,in=45] node[above] {$v = 3u$} (p);
			\draw[colQuadSideNeg] (p) -- node[below] {$u$} (y);
			
			\fill (y) circle[radius=2pt] node[below] {$y$};
			\fill (q) circle[radius=2pt] node[below] {$q$};
			\fill (z) circle[radius=2pt] node[below] {$z$};
			\fill (p) circle[radius=2pt] node[below] {$p$};
		\end{tikzpicture}
		\begin{tikzpicture}[thick]
			\coordinate (p) at (0,0);
			\coordinate (y) at (1,0);
			\path (y) +(60:1) coordinate (q);
			\path (q) +(120:1) coordinate (z);
			
			\draw[colQuadSideNeg] (p) -- node[below] {$u$} (y);
			\draw[colQuadSidePos] (y) -- node[right] {$u$} (q);
			\draw[colQuadSideNeg] (q) -- node[right] {$u$} (z);
			\draw[colQuadSidePos] (z) -- node[left] {$0 < v < 3u$} (p);
			
			\fill (p) circle[radius=2pt] node[below] {$p$};
			\fill (y) circle[radius=2pt] node[below] {$y$};
			\fill (q) circle[radius=2pt] node[right] {$q$};
			\fill (z) circle[radius=2pt] node[above] {$z$};
		\end{tikzpicture}
		\begin{tikzpicture}[thick]
			\coordinate (p) at (0,0);
			\coordinate (z) at (0,0);
			\coordinate (y) at (1,0);
			\path (y) +(120:1) coordinate (q);
			
			\draw[colQuadSideNeg] (p) -- node[below] {$u$} (y);
			\draw[colQuadSidePos] (y) -- node[right] {$u$} (q);
			\draw[colQuadSideNeg] (q) -- node[above] {$u$} (z);
			\draw[colQuadSidePos] (z) -- node[left] {$v=0$} (p);
			
			\fill (p) circle[radius=2pt] node[below] {$p$};
			\fill (y) circle[radius=2pt] node[below] {$y$};
			\fill (q) circle[radius=2pt] node[right] {$q$};
			\fill (z) circle[radius=2pt] node[above] {$z$};
		\end{tikzpicture}
	\end{center}
	\caption{Construction for the proof of \cref{lmm:quadconst} \ref{lmm:quadconst:quadzeroisconst}.}\label{fig:quadriProofConst}
\end{figure}
\begin{lemma}[Constant functions]\label{lmm:quadconst}\mbox{ }
	\begin{enumerate}[label=(\roman*)]
		\item\label{lmm:quadconst:constisquad} For $c\in\R$, let $\tran_c := (x\mapsto c)$. Then $\tran_c\in\Tran$ with $\ctranbest = 0$.
		\item\label{lmm:quadconst:quadzeroisconst} If $\tran\in\Tran$ with $\ctran = 0$, then there is $c\in\R$ such that $\tran = \tran_c$.
	\end{enumerate}
\end{lemma}
\begin{proof}
	The first part is trivial. For the second part, let $u\in\Rpp$ and $v\in(0,3u]$. Take $y,z,q,p\in\R^2$ such that $\ol yq = \ol yp = \ol zp = u$ and  $\ol zq = v$, or $\ol yq = \ol yp = \ol zq = u$ and $\ol zp = v$, see \cref{fig:quadriProofConst}, to obtain
	\begin{align}
		\abs{\tran(u) - \tran(v)} \leq 0
	\end{align}
	from \eqref{eq:quadtran}. This can only be fulfilled for constant functions.
\end{proof}
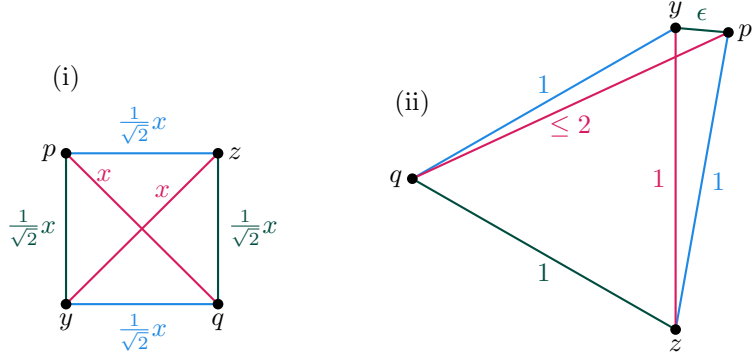
\begin{figure}
	\begin{center}
		\begin{tikzpicture}[thick]
			\coordinate (y) at (0,0);
			\coordinate (q) at (2,0);
			\coordinate (z) at (2,2);
			\coordinate (p) at (0,2);
			
			\draw[colQuadSidePos] (y) -- node[below] {$\frac1{\sqrt2} x$} (q);
			\draw[colQuadSideNeg] (q) -- node[right] {$\frac1{\sqrt2} x$} (z);
			\draw[colQuadSidePos] (z) -- node[above] {$\frac1{\sqrt2} x$} (p);
			\draw[colQuadSideNeg] (p) -- node[left] {$\frac1{\sqrt2} x$} (y);
			\draw[colQuadDiag] (y) -- node[near end,left] {$x$} (z);
			\draw[colQuadDiag] (q) -- node[near end,above] {$x$} (p);
			
			\fill (y) circle[radius=2pt] node[below] {$y$};
			\fill (q) circle[radius=2pt] node[below] {$q$};
			\fill (z) circle[radius=2pt] node[right] {$z$};
			\fill (p) circle[radius=2pt] node[left] {$p$};
			
			\node at (0, 3) {(i)};
		\end{tikzpicture}
		\hspace*{1cm}
		\begin{tikzpicture}[thick]
			\coordinate (y) at (0,4);
			\coordinate (q) at ({-sqrt(3)*2},2);
			\coordinate (z) at (0,0);
			\path (z) +(80:4) coordinate (p);
			
			\draw[colQuadSidePos] (y) -- node[above] {$1$} (q);
			\draw[colQuadSideNeg] (q) -- node[below] {$1$} (z);
			\draw[colQuadSidePos] (z) -- node[right] {$1$} (p);
			\draw[colQuadSideNeg] (p) -- node[above] {$\epsilon$} (y);
			\draw[colQuadDiag] (y) -- node[left] {$1$} (z);
			\draw[colQuadDiag] (q) -- node[below] {$\leq 2$} (p);
			
			\fill (y) circle[radius=2pt] node[above] {$y$};
			\fill (q) circle[radius=2pt] node[left] {$q$};
			\fill (z) circle[radius=2pt] node[below] {$z$};
			\fill (p) circle[radius=2pt] node[right] {$p$};
			
			\node at ({-sqrt(3)*2}, 3) {(ii)};
		\end{tikzpicture}
	\end{center}
	\caption{Constructions for the proof of \cref{lmm:nondecreasing}.}\label{fig:quadriProofNondec}
\end{figure}
\begin{lemma}\label{lmm:nondecreasing}
	Let $\tran\in\Tran$. Then 
	\begin{enumerate}[label=(\roman*)]
		\item \label{lmm:nondecreasing:incr} $\dtran(x)\geq0$ for all $x\in\Rpp$ and
		\item \label{lmm:nondecreasing:finite} $\inf_{x\in\Rpp} \tran(x) > -\infty$.
	\end{enumerate}  
\end{lemma}
\begin{proof}
	\begin{enumerate}[label=(\roman*)]
		\item 
			Let $y,z,q,p\in\R^2$ form a square with diagonal $x\in\Rpp$, see \cref{fig:quadriProofNondec} (i). Then, by \eqref{eq:quadtran},
			\begin{equation}
				0 \leq \ctran x \dtran(x)\eqfs
			\end{equation}
			As $x>0$ and $\ctran \geq 0$ with equality only if $\tran$ is constant, we have $0 \leq \dtran(x)$.
		\item 
			By \ref{lmm:nondecreasing:incr}, we only have to check that $\tran(\varepsilon)$ is bounded below for $\varepsilon\in(0,1]$. Let $\varepsilon>0$. In the Euclidean plane $\R^2$, set $y = (0,1), z = (0,0), q = (-\sqrt{3}/2, 1/2), p = (\cos(\alpha), \sin(\alpha))$, where $\alpha\in[0,\pi/2]$ is chosen such that $\ol yp = \varepsilon$, see \cref{fig:quadriProofNondec} (ii). Then $\ol yq = \ol zq = \ol yz = \ol zp = 1$ and $\ol qp \leq 2$. Then, using \eqref{eq:quadtran} and \ref{lmm:nondecreasing:incr}, we obtain
			\begin{equation}
				\tran(1) - \tran(\varepsilon) \leq 2 \ctran \dtran(1)\eqfs
			\end{equation}
			This yields a constant lower bound of $\tran(1) - 2 \ctran \dtran(1) > -\infty$ on $\tran(\varepsilon)$ for all $\varepsilon\in(0,1]$.
	\end{enumerate}
\end{proof}
Next, we extend the domain of $\tran\in\Tran$ in a consistent way to include $0$.
\begin{proposition}
	Let $\tran\in\Tran$. Define $\tran(0) := \lim_{x\searrow0} \tran(x)$ and $\dtran(0) := \liminf_{x\searrow0} \dtran(x)$. Then, given \eqref{eq:quad2}, inequality \eqref{eq:quadtran} holds for any quadruple of points (not necessarily pairwise distinct).
\end{proposition} 
\begin{proof}
	By \cref{lmm:nondecreasing}, $\tran(0)$ is well-defined. We have to show that \eqref{eq:quad2} implies \eqref{eq:quadtran} in the cases where at least one distance is zero. In the case $y=z$, the left-hand side of \eqref{eq:quadtran} vanishes and the right-hand side is nonnegative by \cref{lmm:nondecreasing} \ref{lmm:nondecreasing:incr}. Next, consider the case $y,z,q,p\in\mc Q$ with $y\neq z$ but at least two points being identical. As any triplet of points in a metric space can be isometrically embedded in the Euclidean plane $\R^2$, this can also be done for $(y,z,q,p)$. Furthermore, we can find a sequence of quadruples of points $(y_n, z_n, q_n, p_n)_{n\in\N} \subset (\R^2)^4$ with $\normof{y_n - z_n} = \ol yz$ and all other distances strictly positive and convergent to the respective distance of the points $y,z,q,p$. By continuity of $\tran$, the definition of $\tran(0)$, and the constant value of $\normof{y_n - z_n}$, \eqref{eq:quadtran} holds in the limit $n\to\infty$.
\end{proof}
\begin{proposition}[Power functions]\label{prp:powerinequ}
	Let $\tran_\alpha = (x\mapsto x^\alpha)$ for $\alpha\in\Rpp$. We have $\tran_\alpha\in\Tran$ if and only if $\alpha\in[1,2]$. In this case, the quadruple constant is $\ctranobest{\tran_\alpha} = 2^{2-\alpha}$.
\end{proposition}
\begin{proof}
	\cite[Theorem 3 and Appendix E]{schoetz19}.
\end{proof}
\subsection{Lower Bounds on the Quadruple Constant}
\begin{lemma}
	The quadruple constant $\ctranbest$ is well-defined in the sense that if $\tran\in\Tran$ there is a smallest value $\ctranbest\in\Rp$ such that \eqref{eq:quadtran} is true for all metric spaces and quadruples of points therein.
\end{lemma}
\begin{proof}
	Let $\tran\in\Tran$. Let $(\ctransymbol_n)_{n\in\N}\subset\Rp$ be a decreasing sequence with $\ctransymbol_\infty := \lim_{n\to\infty} \ctransymbol_n$. Assume \eqref{eq:quadtran} is true for all $\ctran = \ctransymbol_n$. We need to show that \eqref{eq:quadtran} is also true for $\ctran = \ctransymbol_\infty$. 
	If it were false, there would be a metric space $(\mc Q, d)$ with points $y,z,q,p\in\mc Q$ such that
	\begin{equation}
		\tran(\ol yq)  - \tran(\ol yp) - \tran(\ol zq) + \tran(\ol zp) > \ctransymbol_\infty\,  \ol qp \, \dtran(\ol yz)
		\eqfs
	\end{equation}
 	As the inequality is strict, it would also hold when $\ctransymbol_\infty$ is replaced by a $L_n$ for a sufficiently large $n$, which is a contradiction.
\end{proof}
\begin{proposition}\label{lmm:constLowerBounds}
	Let $\tran\in\Tran$. Let $u,v\in\Rpp$. Then, assuming the denominator is not 0, 
	\begin{enumerate}[label=(\roman*)]
		\item \label{lmm:constLowerBounds:2tau}
		\begin{equation}
			\ctranbest \geq 2\frac{\tran(u) -\tran(0)}{u \dtran(u)}\eqcm
		\end{equation}
		\item \label{lmm:constLowerBounds:2u}
		\begin{equation}
			\ctranbest \geq \frac{\tran(2u) -\tran(0)}{2u \dtran(u)}\eqcm
		\end{equation}
		\item \label{lmm:constLowerBounds:subadd}
		\begin{equation}
			\ctranbest \geq \frac{\dtran(u) - \dtran(v)}{\dtran(|u-v|)}\eqcm
		\end{equation}
		\item \label{lmm:constLowerBounds:supadd}
		\begin{equation}
			\ctranbest \geq \frac{\dtran(u) + \dtran(v)}{\dtran(u+v)}\eqfs
		\end{equation}
	\end{enumerate}
\end{proposition}
\begin{figure}
	\begin{center}
		\begin{tikzpicture}[thick]
			\coordinate (y) at (4,0);
			\coordinate (q) at (0,0);
			\coordinate (z) at (0,0);
			\coordinate (p) at (2,0);
			
			\draw[colQuadSidePos] (y) to[in=45,out=135] node[above] {$u$} (q);
			\draw[colQuadSideNeg] (p) -- node[below] {$\frac u2$} (y);
			\draw[colQuadDiag] (q) -- node[below] {$\frac u2$} (p);
			
			\fill (y) circle[radius=2pt] node[below] {$y$};
			\fill (q) circle[radius=2pt] node[below] {$q = z$};
			\fill (p) circle[radius=2pt] node[below] {$p$};
			
			\node at (0, 1) {(i)};
		\end{tikzpicture}
		\hspace*{1cm}
		\begin{tikzpicture}[thick]
			\coordinate (y) at (2,0);
			\coordinate (q) at (0,0);
			\coordinate (z) at (0,0);
			\coordinate (p) at (4,0);
			
			\draw[colQuadSidePos] (y) --  node[below] {$u$} (q);
			\draw[colQuadSideNeg] (p) -- node[below] {$u$} (y);
			\draw[colQuadDiag] (q) to[in=135,out=45] node[above] {$2u$} (p);
			
			\fill (y) circle[radius=2pt] node[below] {$y$};
			\fill (q) circle[radius=2pt] node[below] {$q = z$};
			\fill (p) circle[radius=2pt] node[below] {$p$};
			
			\node at (0, 1) {(ii)};
		\end{tikzpicture}
		
		\vspace*{1cm}
		
		\begin{tikzpicture}[thick]
			\coordinate (y) at (3,0);
			\coordinate (q) at (0,0);
			\coordinate (z) at (5,0);
			\coordinate (p) at (1,0);
			
			\draw[colQuadSidePos] (y) to[in=45,out=135] node[above] {$u$} (q);
			\draw[colQuadSideNeg] (q) to[in=-135,out=-45] node[above] {$v$} (z);
			\draw[colQuadSidePos] (z) to[in=45,out=135] node[above] {$v-\epsilon$} (p);
			\draw[colQuadSideNeg] (p) -- node[below] {$u-\epsilon$} (y);
			\draw[colQuadDiag] (y) -- node[below] {$|u-v|$} (z);
			\draw[colQuadDiag] (q) -- node[below] {$\epsilon$} (p);
			
			\fill (y) circle[radius=2pt] node[below] {$y$};
			\fill (q) circle[radius=2pt] node[below] {$q$};
			\fill (z) circle[radius=2pt] node[below] {$z$};
			\fill (p) circle[radius=2pt] node[below] {$p$};
			
			\node at (0, 1) {(iii)};
		\end{tikzpicture}
		\hspace*{1cm}
		\begin{tikzpicture}[thick]
			\coordinate (y) at (3,0);
			\coordinate (q) at (0,0);
			\coordinate (z) at (-2,0);
			\coordinate (p) at (1,0);
			
			\draw[colQuadSidePos] (y) to[in=45,out=135] node[above] {$u$} (q);
			\draw[colQuadSideNeg] (q) -- node[below] {$v$} (z);
			\draw[colQuadSidePos] (z) to[in=135,out=45] node[above] {$v+\epsilon$} (p);
			\draw[colQuadSideNeg] (p) -- node[below] {$u-\epsilon$} (y);
			\draw[colQuadDiag] (y) to[in=-45,out=-135] node[below] {$u+v$} (z);
			\draw[colQuadDiag] (q) -- node[below] {$\epsilon$} (p);
			
			\fill (y) circle[radius=2pt] node[below] {$y$};
			\fill (q) circle[radius=2pt] node[below] {$q$};
			\fill (z) circle[radius=2pt] node[below] {$z$};
			\fill (p) circle[radius=2pt] node[below] {$p$};
			
			\node at (-2, 1) {(iv)};
		\end{tikzpicture}
	\end{center}
	\caption{Constructions for the proof of \cref{lmm:constLowerBounds}.}\label{fig:quadriProofLower}
\end{figure}
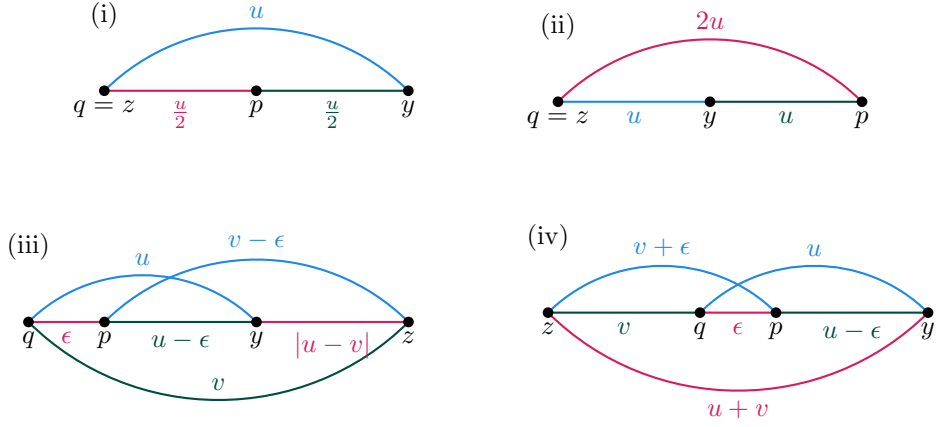
\begin{proof}
	For all parts, we apply \eqref{eq:quadtran} in the metric space of the real line with Euclidean metric, see \cref{fig:quadriProofLower}.
	\begin{enumerate}[label=(\roman*)]
		\item Set $z=q=0$, $y=u$, $p = u/2$. Then \eqref{eq:quadtran} becomes
		\begin{equation*}
			\tran(u) - \tran(0) \leq \ctran \frac u2 \dtran(u)
			\eqfs
		\end{equation*}
		\item 
		Set $z=q=0$, $y=u$, $p = 2u$.
		Then \eqref{eq:quadtran} becomes
		\begin{equation*}
			\tran(2u) - \tran(0) \leq 2 \ctran u \dtran(u)
			\eqfs
		\end{equation*}
		\item 
		Let $\varepsilon\in(0, \min(u,v))$. Set $q = 0$, $p = \varepsilon$, $y = u$, $z = v$.
		Then \eqref{eq:quadtran} becomes
		\begin{align*}
			\tran(u)  - \tran(u-\varepsilon) - \tran(v) + \tran(v-\varepsilon) \leq \ctran\,  \varepsilon \, \dtran(\abs{u-v})
			\eqfs
		\end{align*}
		Thus,
		\begin{align*}
			\ctran \dtran(\abs{u-v}) \geq 
			\frac{\tran(u)  - \tran(u-\varepsilon)}{\varepsilon} -\frac{\tran(v) - \tran(v-\varepsilon)}{\varepsilon}
			\xrightarrow{\varepsilon \to 0}
			\dtran(u) - \dtran(v)
			\eqfs
		\end{align*}
		\item 
		Let $\varepsilon\in(0, \min(u,v))$. Let $q = 0$, $p = \varepsilon$, $y = u$, $z = -v$.
		Then \eqref{eq:quadtran} becomes
		\begin{align*}
			\tran(u)  - \tran(u-\varepsilon) - \tran(v) + \tran(v+\varepsilon) \leq \ctran\,  \varepsilon \, \dtran(u+v)
			\eqfs
		\end{align*}
		Thus,
		\begin{align*}
			\ctran \dtran(u+v) \geq 
			\frac{\tran(u)  - \tran(u-\varepsilon)}{\varepsilon}  + \frac{\tran(v+\varepsilon) - \tran(v)}{\varepsilon}
			\xrightarrow{\varepsilon \to 0}
			\dtran(u) + \dtran(v)
			\eqfs
		\end{align*}
	\end{enumerate}
\end{proof}
\begin{proposition}\label{lmm:constLowerBoundOne}
	Let $\tran\in\Tran$. Assume $\tran$ is not constant. Then $\ctranbest \geq 1$.
\end{proposition}
\begin{figure}
	\begin{center}
		\begin{tikzpicture}[thick]
			\coordinate (y) at (4,0);
			\coordinate (q) at (0,0);
			\coordinate (z) at (0,0);
			\coordinate (p) at (1,0);
			
			\draw[colQuadSidePos] (y) to[in=45,out=135] node[above] {$u$} (q);
			\draw[colQuadSideNeg] (p) -- node[below] {$u-\epsilon$} (y);
			\draw[colQuadDiag] (q) -- node[below] {$\epsilon$} (p);
			
			\fill (y) circle[radius=2pt] node[below] {$y$};
			\fill (q) circle[radius=2pt] node[left] {$q = z$};
			\fill (p) circle[radius=2pt] node[below] {$p$};
		\end{tikzpicture}
	\end{center}
	\caption{Constructions for the proof of \cref{lmm:constLowerBoundOne}.}\label{fig:quadriProofOne}
\end{figure}
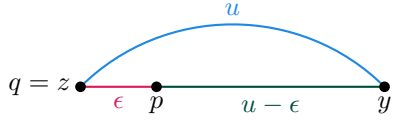
\begin{proof}
	Let $u\in\Rpp$ be such that $\dtran(u) \neq 0$. Then \cref{lmm:nondecreasing} \ref{lmm:nondecreasing:incr} implies $\dtran(u) > 0$.
	Let $\varepsilon\in(0, u)$. In the metric space of the real line with Euclidean metric, we choose $z = q = 0$, $p = \varepsilon$, $y = u$, see \cref{fig:quadriProofOne}.
	Then \eqref{eq:quadtran} becomes
	\begin{align*}
		\tran(u)  - \tran(u-\varepsilon) - \tran(0) + \tran(\varepsilon) \leq \ctran\,  \varepsilon \, \dtran(u)
		\eqfs
	\end{align*}
	As $\tran$ is nondecreasing by \cref{lmm:nondecreasing} \ref{lmm:nondecreasing:incr}, we have $\tran(0) = \lim_{x\searrow0} \tran(x) \leq \tran(\varepsilon)$.
	Thus,
	\begin{align*}
		\ctran \dtran(u) \geq 
		\frac{\tran(u)  - \tran(u-\varepsilon)}{\varepsilon}
		\xrightarrow{\varepsilon \to 0}
		\dtran(u)
		\eqfs
	\end{align*}
\end{proof}
\subsection{Stability}\label{ssec:quad:stabi}
We can construct potentially new elements of $\Tran$ from given ones by taking limits or certain linear combinations, as the next two propositions show.
\begin{proposition}\label{lmm:quad:limit}
	Let $(\tran_n)_{n\in\N}\subset\Tran$ with constants $(\ctrano{\tran_n})_{n\in\N}\subset\Rp$. Let $\tran\colon\Rpp\to\R$ be a differentiable function. Assume $\lim_{n\to\infty}\tran_n(x) = \tran(x)$ and $\limsup_{n\to\infty} \dtran_n(x) \leq \dtran(x)$ for all $x\in\Rpp$. Set $\ctran := \limsup_{n\to\infty} \ctrano{\tran_n}$. Assume $\ctran < \infty$. Then $\tran\in\Tran$ with constant $\ctran$.
\end{proposition}
\begin{proof}
	This is a direct consequence of the assumed limit properties of $\tran$, $\dtran$, and $\ctran$, and of the quadruple inequality for $\tran_n$ with constant $\ctrano{\tran_n}$:
	\begin{align*}
		\tran(\ol yq) - \tran(\ol yp) - \tran(\ol zq) + \tran(\ol zp)
		&=
		\lim_{n\to\infty}
		\br{\tran_n(\ol yq) - \tran_n(\ol yp) - \tran_n(\ol zq) + \tran_n(\ol zp)}
		\\&\leq
		\limsup_{n\to\infty}
		\ctrano{\tran_n}\,\ol qp\, \dtran_n(\ol yz)
		\\&\leq
		\ctran \,\ol qp\, \dtran(\ol yz)
		\eqfs
	\end{align*}
\end{proof}
\begin{proposition}\label{lmm:quadstablin}\mbox{ }
	\begin{enumerate}[label=(\roman*)]
		\item\label{lmm:quadstablin:factor} Let $\tran\in\Tran$ with constant $\ctran$. Let $a\in\Rp$. Then $(x\mapsto a\tran(x)) \in \Tran$ with the same quadruple constant $\ctran$.
		\item\label{lmm:quadstablin:add} Let $\tran, \tilde \tran\in \Tran$ with constant $\ctrano{\tran}$ and $\ctrano{\tilde\tran}$, respectively. Then $(x\mapsto \tran(x)+\tilde\tran(x)) \in \Tran$ with quadruple constant $\max(\ctrano{\tran}, \ctrano{\tilde\tran})$.
	\end{enumerate}
\end{proposition}
\begin{proof}
	Both parts follow directly from the definition of $\Tran$ and linearity of the derivative operator.
\end{proof}

\section{Nondecreasing, Convex Functions with Concave Derivative}\label{sec:cc}
As the set of nondecreasing, convex functions with concave derivative is central to this article, we start by discussing some basic properties of these functions in this section.

We define $\setcc$ to be the set of nondecreasing convex functions $\tran\colon\Rp\to\R$ that are differentiable on $\Rpp$ with concave derivative $\dtran$. We extend the domain of $\dtran$ to $\Rp$ by setting  $\dtran(0) := \lim_{x\searrow0} \dtran(x)$, which exists as $\dtran$ is nonnegative and nondecreasing.
Let $I\subset\R$ be convex. Let $t_0\in I$ with $t_0<\sup I$. For a function $f\colon I \to \R$ the \textit{right derivative} at $t_0$ is defined as
\begin{equation}
	\partial_+ f(t_0) = \lim_{t\searrow t_0}\frac{f(t)-f(t_0)}{t-t_0}
\end{equation}
if the limit exists.
Let $\setd{k}(\Rp)$ denote the space of $k$-times continuously differentiable functions $\Rp\to\R$, where derivatives at $0$ are taken as right derivatives. 
Denote $\setcco := \cb{\tran\in\setcc \colon \tran(0) = 0}$. 
Denote $\setccd{k} := \setcc\cap \setd{k}(\Rp)$ and $\setccod{k} := \setcco\cap \setd{k}(\Rp)$ for $k\in\N\cup\{\infty\}$.
\subsection{Properties}
In this section, we establish some simple properties of functions $\tran\in\setccd{k}$ that will be useful in the proof of the main results. The proofs can be found in the appendix \ref{app:aux:cc}.
\begin{lemma}\label{lmm:ccSimpleProps}
	Let $\tran\in\setcc$. 
	\begin{enumerate}[label=(\roman*)]
		\item Then $\tran$ is continuous; its right derivative at $0$ exists and $\partial_+ \tran(0) = \dtran(0)$; $\dtran$ is nonnegative, nondecreasing, and continuous.
		\item Assume $\tran\in\setccd{2}$. Then $\ddtran$ is nonnegative and nonincreasing.
		\item Assume $\tran\in\setccd{3}$. Then $\dddtran$ is nonpositive.
	\end{enumerate}
\end{lemma}
The next lemma shows that all functions $\tran\in\setccd3$ are between a nondecreasing linear function and a parabola that opens upward.
\begin{lemma}[Polynomial bounds]\label{lmm:ccpoly}\mbox{ }
	Let $\tran\in\setccd3$. Let $x, y \in \Rp$. Then
	\begin{enumerate}[label=(\roman*)]
		\item  \label{lmm:ccpoly:taylor} 
		\begin{equation*}
			\tran(x) + y \dtran(x) \leq \tran(x + y) \leq \tran(x) + y \dtran(x) + \frac12 y^2 \ddtran(x)
			\eqcm
		\end{equation*}
		\item \label{lmm:ccpoly:dtaylor} 
		\begin{equation*}
			\dtran(x) \leq \dtran(x + y) \leq \dtran(x) + y \ddtran(x)
			\eqfs
		\end{equation*}
	\end{enumerate}
\end{lemma}
In the following lemma, we provide useful bounds for the proof of the main results. 
\begin{lemma}[Difference bound]\label{lmm:ccdiff}
	Let $\tran\in\setcc$.
	\begin{enumerate}[label=(\roman*)]
		\item \label{lmm:ccdiff:tight} Let $x,y\in\Rp$. Assume $x \geq y$. Then
		\begin{equation*}
			\frac{x-y}2\br{\dtran(x)+\dtran(y)} \leq \tran(x) - \tran(y) \leq (x-y)\dtran\brOf{\frac{x+y}{2}}
			\eqfs
		\end{equation*}
		\item \label{lmm:ccdiff:tightabs} Let $x,y\in\Rp$. Then
		\begin{equation*}
			\tran(x+y) - \tran(\abs{x-y}) \leq 2\min(x, y)\dtran(\max(x, y))
			\eqfs
		\end{equation*}
	\end{enumerate}
\end{lemma}
\subsection{Approximation}
In the proof of \cref{thm:main}, we will first show the result for $\tran\in\setccd{3}$ and then approximate the remaining functions in $\setcc$ via smooth functions. The following lemma shows that this is possible.
\begin{lemma}[Smooth approximation]\label{lmm:cc:molli}
	Let $\tran\in\setcc$. Then there is a sequence $(\tran_n)_{n\in\N}\subset \setccd{\infty}$ such that $\tran(x) =\lim_{n\to\infty}\tran_n(x)$ and $\dtran(x) =\lim_{n\to\infty}\dtran_n(x)$.
\end{lemma}
\begin{proof}
	We will smooth $\dtran$ by convolution with a mollifier. The convolution is executed in the group of positive real numbers under multiplication endowed with its Haar measure $\mu(A) = \int_A \frac1x \dl x$ for $A\subset\Rpp$ measurable.
	
	For $n\in\N$, let $\varphi_n\in\setd\infty(\Rpp)$ be a sequence of nonnegative functions with support in $[\exp(-1/n), \exp(1/n)]$ and 
	\begin{equation}
		\int_0^\infty \frac{\varphi_n(x)} x \dl x = 1\eqfs
	\end{equation}
	Let $\tran\in\setcc$ with derivative $\dtran$. 
	For $n\in\N$, $x\in\Rp$, we define
	\begin{equation}
		\tran_n(x) := \tran(0) + \int_{0}^x \int_{0}^\infty \frac{\varphi_n(z)}{z} \dtran\brOf{\frac yz}  \dl z \dl y
		\eqfs
	\end{equation}
	Then, for $y\in\Rp$,
	\begin{equation}
		\dtran_n(y) = \int_{0}^\infty \frac{\varphi_n(z)}{z} \dtran\brOf{\frac yz}  \dl z = \int_{\R} \varphi_n(\euler^t) \dtran\brOf{\euler^{\log(y) - t}}  \dl t
		\eqfs
	\end{equation}
	Thus, $s\mapsto \dtran_n(\euler^s)$ is the convolution of $t\mapsto \varphi_n(\euler^t)$ with $t\mapsto\dtran\brOf{\euler^t}$. Using standard results on convolutions, the mollified function has the following properties: 
	\begin{enumerate}[label=(\roman*)]
		\item $\dtran_n$ is infinitely differentiable on $\Rpp$, because $\varphi_n$ is,
		\item $\dtran_n$ is nonnegative, nondecreasing, and concave, because $\dtran$ is and $\varphi_n$ is nonnegative,
		\item $\lim_{n\to\infty}\dtran_n(x) = \dtran(x)$ because $\dtran$ is continuous.
	\end{enumerate}
	Furthermore, $\tran_n$ is convex, as $\dtran_n$ is nondecreasing and $\lim_{n\to\infty}\tran_n(x) = \tran(x)$ by dominated convergence.
	Thus, $(\tran_n)_{n\in\N} \subset \setccd{\infty}$, and the sequence has the desired point-wise limits.
\end{proof}
\subsection{Stability}
The set $\setcc$ enjoys similar stability properties as $\Tran$. Thus, after having shown $\setcc\subset\Tran$, we cannot easily find further functions in $\Tran$ from the constructions presented in section \ref{ssec:quad:stabi}.
\begin{proposition}
	Let $(\tran_n)_{n\in\N}\subset\setcc$. Let $\tran\colon\Rp\to\R$ be differentiable. Assume that $\tran(x) = \lim_{n\to\infty}\tran_n(x)$ for all $x\in\Rp$. Then $\tran\in\setcc$.
\end{proposition}
\begin{proof}
	As $\tran_n$ is nondecreasing and convex, so is $\tran$.
	As $\dtran_n$ is concave, by \cite[Theorem 25.7]{rockafellar70}, for all $x\in\Rpp$,
	\begin{equation}
		\lim_{n\to\infty} \dtran_n(x) = \dtran(x)
		\eqfs
	\end{equation}
	Thus, as all $\dtran_n$ are concave, $\dtran$ is concave on $\Rpp$.
	As $\tran$ is convex, $\frac{\tran(x+h)-\tran(x)}{h}$ is increasing in both $x$ and $h$. Thus,
	\begin{align*}
		\liminf_{x\searrow0} \dtran(x) 
		&= 
		\inf_{x\in\Rpp} \dtran(x)
		\\&=
		\inf_{x,h\in\Rpp} \frac{\tran(x+h)-\tran(x)}{h} 
		\\&=
		\liminf_{h\searrow0} \frac{\tran(h)-\tran(0)}{h}
		\\&=
		\dtran(0)
		\eqfs
	\end{align*}
	Therefore, $\dtran$ is continuous at $0$ and concavity extends to $\Rp$.
\end{proof}
\begin{proposition}\mbox{ }
	\begin{enumerate}[label=(\roman*)]
		\item Let $\tran\in\setcc$. Let $a\in\Rp$. Then $(x\mapsto a\tran(x)) \in \setcc$.
		\item Let $\tran, \tilde \tran\in \setcc$. Then $(x\mapsto \tran(x)+\tilde\tran(x)) \in \setcc$.
	\end{enumerate}
\end{proposition}
\begin{proof}
	Both parts follow directly from the definition of $\setcc$ and linearity of the derivative operator.
\end{proof}

\section{Outline of the Proof of \cref{thm:main}}\label{sec:proofoutline}
In the first step of the proof of \cref{thm:main}, we represent general $4$-point metric spaces with $6$ real-valued parameters. We refer to this representation as a \textit{parametrization}. It converts our problem from the domain of metric geometry to the domain of real analysis. The rest of the proof consists of a complex sequence of elementary calculus arguments. This sequence may seem difficult to discover. To aid this process, an extensive application of computer-assisted numerical assessments was employed. The inequality \eqref{eq:quadtran} and transformations of it were evaluated on a grid of the parameter space and for a finite set of functions $\tran$. This computational tool played a crucial role in guiding the proof. It helped to identify steps that would not be useful and indicated steps with potential merit. 
\subsection{Parametrization}
There are different ways to parametrize a general $k$-tuple of points of a metric space. A trivial way is to take the $k$ distances as the $k$ real parameters. Then the parameter space is a subset of $\Rp^k$ with certain constraints that ensure the triangle inequality. In this parametrization, the representation of distances is simple and the parameter space is more complex.
We will base our parametrization on the Euclidean cosine-formula. We start with a parametrization of $3$ points, see \cref{fig:param:threepoints}. The parameter space for this construction is $\R^2 \times [-1, 1]$ without further constraints. In contrast to the purely distance based parametrization described above, this parametrization yields a more complex representation of all distances, but a very simple parameter space.
\begin{proposition}[3-point parametrization]\label{lmm:threepoint}\mbox{ }
	\begin{enumerate}[label=(\roman*)]
		\item 	Let $(\mc Q, d)$ be a metric space and $y,q,p\in\mc Q$. Set $a := \ol yp$ and $b:= \ol qp$. Isometrically embed $y,q,p$ in the Euclidean plane $\R^2$ and set $s := \cos(\measuredangle ypq)$ with the angle $\measuredangle ypq$ measured in the Euclidean plane. Then,  $\ol yq^2 = a^2 + b^2 - 2 s ab$	and $a,b\in\Rp, s \in[-1, 1]$.
		\item
		For all $a,b\in\Rp, s \in[-1, 1]$ there is a metric space with a triplet of points $y,q,p$ such that $\ol yp = a$, $\ol qp = b$, and $\ol yq^2 = a^2 + b^2 - 2 s ab$.
	\end{enumerate}
\end{proposition}
\begin{proof}
	The first part is true by construction. For given parameters $a,b\in\Rp, s \in[-1, 1]$, we can easily construct a triangle in $\R^2$ with sides $a, b, \sqrt{a^2 + b^2 - 2 s ab}$ and an angle $\arccos(s)$ between the sides with length $a$ and $b$. The corners of this triangle are the points $y,q,p$.
\end{proof}
\begin{figure}
	\begin{center}
		\begin{tikzpicture}[thick]
			\coordinate (y) at (-1,0);
			\coordinate (q) at (5,0);
			\coordinate (p) at (1,2);
			
			\draw (y) -- node[below] {$\sqrt{a^2 + b^2 - 2sab}$} (q);
			\draw (p) -- node[left] {$a$} (y);
			\draw (q) -- node[above] {$b$} (p);
			
			\fill (y) circle[radius=2pt] node[below] {$y$};
			\fill (q) circle[radius=2pt] node[below] {$q$};
			\fill (p) circle[radius=2pt] node[left] {$p$};
			
			\tkzMarkAngle[draw=black](y,p,q)
			\tkzLabelAngle[pos = 0.5](y,p,q){$\lessdot s$}
		\end{tikzpicture}
	\end{center}
	\caption{A 3-point parametrization. We denote $\lessdot s := \arccos(s)$.}\label{fig:param:threepoints}
\end{figure}
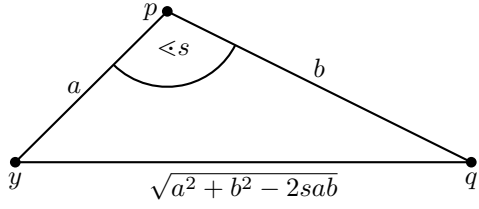
For the proof of \cref{thm:main}, we use a $4$-point parametrization, see \cref{fig:param:fourpoints}. It is based on repeated application of the Euclidean cosine formula. Its parameter space is a subset of $\Rp^3\times[-1,1]^3$ with rather complex constraints. We later relax this parametrization to a construction with a parameter space $\Rp^3\times[-1,1]^2$ without further constraints. That construction is not a parametrization, but results shown in its parameter space are stronger than those shown in the parametrization.
\begin{proposition}[4-point parametrization]\label{lmm:fourpoint}\mbox{ }
	\begin{enumerate}[label=(\roman*)]
		\item Let $(\mc Q, d)$ be a metric space and $y,z,q,p\in\mc Q$. Set $a:=\ol zp$, $c:=\ol yp$, $b:=\ol qp$, $s := \cos(\measuredangle ypq)$, $r := \cos(\measuredangle zpq)$, $t := \cos(\measuredangle ypz)$, with all angles measured as in an isometric $3$-point embedding in the Euclidean plane.
		Then,
		\begin{align*}
			\ol yz^2 &= a^2 + c^2 - 2tac\eqcm
			&
			\ol yq^2 &= c^2 + b^2 - 2scb\eqcm
			&
			\ol zq^2 &= a^2 + b^2 - 2rab
			\eqfs
		\end{align*}
		Furthermore, $a,b,c\in\Rp$, $r,s,t\in[-1, 1]$, and
		\begin{equation}\label{lmm:fourpoint:trianlge}
			\begin{aligned}
				-tac &\leq b^2 - rab - scb + \sqrt{a^2 - 2rab + b^2}\sqrt{c^2 - 2scb + b^2}\eqcm\\
				-rab &\leq c^2 - tac - scb + \sqrt{a^2 - 2tac + c^2}\sqrt{c^2 - 2scb + b^2}\eqcm\\
				-scb &\leq a^2 - rab - tac + \sqrt{a^2 - 2rab + b^2}\sqrt{a^2 - 2tac + c^2}\eqfs
			\end{aligned} 
		\end{equation}
		\item For all $a,b,c\in\Rp$, $r,s,t\in[-1, 1]$ that fulfill \eqref{lmm:fourpoint:trianlge}, there is a metric space $(\mc Q, d)$ with a quadruple of points $y,z,q,p\in\mc Q$ such that
		\begin{align*}
			a &= \ol zp\eqcm
			&
			c &= \ol yp\eqcm
			&
			b &= \ol qp\eqcm
			\\
			\ol yz^2 &= a^2 + c^2 - 2tac\eqcm
			&
			\ol yq^2 &= c^2 + b^2 - 2scb\eqcm
			&
			\ol zq^2 &= a^2 + b^2 - 2rab
			\eqfs
		\end{align*}
	\end{enumerate}
\end{proposition}
\begin{proof}
	\begin{enumerate}[label=(\roman*)]
		\item 
		The inequalities \eqref{lmm:fourpoint:trianlge} are due to the triangle inequality,
		\begin{equation*}
			\ol yz \leq \ol yq + \ol zq 
			\eqcm\qquad
			\ol zq \leq \ol yz + \ol yq
			\eqcm\qquad
			\ol yq \leq \ol yz + \ol zq
			\eqfs
		\end{equation*}
		\item
		Define a four point set $\mc Q$ with elements named $y,z,q,p$. Define $d\colon\mc Q\times \mc Q\to\Rp$ with the equations given in the lemma, extended by symmetry and $d(x,x) = 0$ for all $x\in\mc Q$. By construction, $d$ is a semimetric \cite{deza16} (vanishing distance for non-identical points allowed) so that identifying points with distance 0 yields a metric space.
	\end{enumerate}
\end{proof}
\begin{figure}
	\begin{center}
		\begin{tikzpicture}[thick]
			\coordinate (y) at (-1,0);
			\coordinate (q) at (7,0);
			\coordinate (z) at (1.5,5);
			\coordinate (p) at (2,2);
			
			\draw[colQuadSidePos] (y) -- node[below] {$\sqrt{c^2 + b^2 - 2scb}$} (q);
			\draw[colQuadSideNeg] (q) -- node[right] {$\sqrt{a^2 + b^2 - 2rab}$} (z);
			\draw[colQuadSidePos] (z) -- node[right] {$a$} (p);
			\draw[colQuadSideNeg] (p) -- node[left] {$c$} (y);
			\draw[colQuadDiag] (y) -- node[left] {$\sqrt{a^2 + c^2 - 2tac}$} (z);
			\draw[colQuadDiag] (q) -- node[below] {$b$} (p);
			
			\fill (y) circle[radius=2pt] node[below] {$y$};
			\fill (q) circle[radius=2pt] node[below] {$q$};
			\fill (z) circle[radius=2pt] node[right] {$z$};
			\fill (p) circle[radius=2pt] node[below] {$p$};
			
			\tkzMarkAngle[draw=black](y,p,q)
			\tkzLabelAngle[pos = 0.65](y,p,q){$\lessdot s$}
			\tkzMarkAngle[draw=black](q,p,z)
			\tkzLabelAngle[pos = 0.5](q,p,z){$\lessdot r$}
			\tkzMarkAngle[draw=black](z,p,y)
			\tkzLabelAngle[pos = 0.5](z,p,y){$\lessdot t$}
		\end{tikzpicture}
	\end{center}
	\caption{A 4-point parametrization. We denote $\lessdot x := \arccos(x)$.}\label{fig:param:fourpoints}
\end{figure}
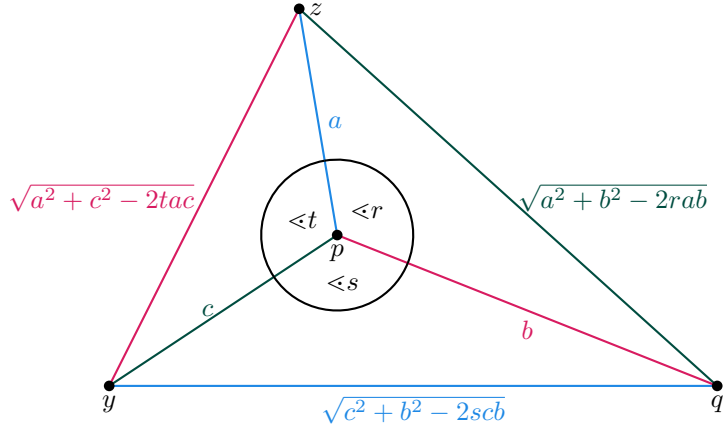
With this parametrization and 
\begin{equation}\label{eq:quad2lhs:simplification}
	a^2 - c^2 - (a^2-2rab+b^2) + (c^2-2scb+b^2)
	= 2b \br{ra-sc}
	\eqcm
\end{equation}
\eqref{eq:quad2} can be expressed as
\begin{equation}\label{eq:quad2:param}
	b \br{ra-sc} \leq b \sqrt{a^2 + c^2 - 2tac}\eqcm
\end{equation}
 and \eqref{eq:quadtran} becomes
\begin{equation}\label{eq:quadtran:param}
	\tran(a) - \tran(c) - \trans\brOf{a^2-2rab+b^2} + \trans\brOf{c^2-2scb+b^2} \leq \ctran \, b \, \dtran(\sqrt{a^2 + c^2 - 2tac})\eqcm
\end{equation}
where we use the shorthand $\trans(x) := \tran(\sqrt x)$.
Thus, \cref{thm:main} is equivalent to showing that \eqref{eq:quad2:param} implies \eqref{eq:quadtran:param} for all $a,b,c\in\Rp$, $r,s,t\in[-1, 1]$ that fulfill \eqref{lmm:fourpoint:trianlge}. We will prove a stronger but simpler looking result in section \ref{app:mainproof} of the appendix:
\begin{theorem}\label{thm:main:param}
	Let $a,b,c\geq0$, $r,s\in[-1,1]$, and $\tran\in\setccod{3}$.
	Then
	\begin{equation}\label{eq:thm:main:param}
			\tran(a) - \tran(c) - \trans\brOf{a^2-2rab+b^2} + \trans\brOf{c^2-2scb+b^2}
			\leq 
			2 b \dtran\brOf{\max(ra - sc, |a-c|)}
		\eqfs
	\end{equation}
\end{theorem}
\subsection{Remaining Proof Steps}\label{ssec:remaining}
From \cref{thm:main:param}, we obtain a slightly stronger result than \cref{thm:main} by relaxing \eqref{eq:quad2}:
\begin{theorem}\label{thm:strong}
	Let $\tran\in\setcc$.  Let $(\mc Q, d)$ be a metric space.
	Let $y,z,q,p\in\mc Q$. Assume, there is $\ctransqr\in[2,\infty)$ such that
	\begin{equation}\label{eq:strong:assu}
		\ol yq^2 - \ol yp^2 - \ol zq^2 + \ol zp^2 \leq \ctransqr \, \ol qp\, \ol yz
		\eqfs
	\end{equation}
	Then
	\begin{equation}\label{eq:strong}
		\tran(\ol yq)  - \tran(\ol yp) - \tran(\ol zq) + \tran(\ol zp) \leq \ctransqr \, \ol qp \, \dtran\brOf{\ol yz}\eqfs
	\end{equation}
\end{theorem}
\begin{proof}[that \cref{thm:main:param} implies \cref{thm:strong}]\label{proof:paramImpliesStrong}
	Let $\tran \in \setccod3$.
	Using \eqref{eq:quad2lhs:simplification}, the metric version of \eqref{eq:thm:main:param} is 
	\begin{equation}\label{eq:strong:proof:1}
		\begin{aligned}
			&\tran(\ol yq)  - \tran(\ol yp) - \tran(\ol zq) + \tran(\ol zp) 
			\\&\leq 
			2 \,\ol qp\, \dtran\brOf{\max\brOf{\frac{\ol yq^2 - \ol yp^2 - \ol zq^2 + \ol zp^2}{2 \,\ol qp}, |\ol zp - \ol yp|}}
			\eqfs
		\end{aligned}
	\end{equation}
	Using \eqref{eq:strong:assu} and the triangle inequality, we bound the right-hand side of \eqref{eq:strong:proof:1}, by
	\begin{equation}
		2 \,\ol qp\, \dtran\brOf{\max\brOf{\frac{\ctransqr \, \ol yz}{2}, \ol yz}}\eqfs
	\end{equation}
	With \cref{lmm:tranconcave} \ref{lmm:tranconcave:factor} and $\ctransqr/2\geq 1$, we obtain \eqref{eq:strong}.
	
	To extend the result shown for $\tran\in\setccod3$ to $\tran\in\setcc$, we use \cref{lmm:cc:molli} and \cref{lmm:quad:limit} to remove the smoothness requirement, and \cref{lmm:quadconst} \ref{lmm:quadconst:constisquad} and \cref{lmm:quadstablin} \ref{lmm:quadstablin:add} to remove the requirement $\tran(0) = 0$.
\end{proof}
\cref{thm:main} follows from \cref{thm:strong} by fixing $\ctransqr = 2$.
The remaining part of the main proof, i.e., the proof of \cref{thm:main:param}, is given in the appendix section \ref{app:mainproof}. \cref{fig:proof:overview} gives an overview of how the different intermediate results presented above and below relate to each other. 
\begin{figure}
	\begin{center}
		\begin{tikzpicture}[node distance=2cm,>=stealth,thick]
			\tikzset{
				box/.style={
					draw,
					rectangle,
					minimum width=2.5cm,
					minimum height=0.7cm,
					text centered
				}
			}
			
			\node[box] (main) {\cref{thm:main}};
			\node[box, below of=main] (strong) {\cref{thm:strong}};
			\node[box, below of=strong] (mainParam) {\cref{thm:main:param}};
			\node[box, below of=mainParam] (firstreduction) {\cref{lmm:param:firstreduction}: Elimination of $r$};
			
			\node[box, anchor=east] (reduii) at (12,0) {\cref{lmm:param:reduii}: (ii)};
			\node[box, anchor=east] (redui1) at (12,-1)  {\cref{lmm:param:reduicGaGbGsc}: (i) $c\geq a\geq b \geq sc$};
			\node[box, anchor=east] (redui2) at (12, -2) {\cref{lmm:param:reduicGaGbLsc}: (i) $c\geq a\geq b \leq sc$};
			\node[box, anchor=east] (redui3) at (12, -3) {\cref{lmm:param:reduicGaLb}: (i) $c\geq a\leq b$};
			\node[box, anchor=east] (redui4) at (12, -4) {\cref{lmm:param:reduicLaLbscbLsc}: (i) $a\geq c$, $b\leq 2sc$, $sc\geq a-b$};
			\node[box, anchor=east] (redui5) at (12, -5) {\cref{lmm:param:reduicLabGsc}: (i) $a\geq c$, $b\geq 2sc$};
			\node[box, anchor=east] (redui6) at (12, -6) {\cref{lmm:param:reduicLaGbscbLsc}: (i) $a\geq c$, $b\leq 2sc$, $sc \leq a-b$};
			
			\draw[->] (mainParam) --node[anchor=west]{section \ref{ssec:remaining}} (strong);
			\draw[->] (strong) --node[anchor=west]{$L=2$} (main);
			\draw[->] (firstreduction) --node[anchor=west]{section \ref{sssec:firstReduImpliesMainParam}} (mainParam);
			
			\draw (4,0) -- (4,-6);
			\draw[] (reduii) -- (4, 0);
			\draw[] (redui1) -- (4, -1);
			\draw[] (redui2) -- (4, -2);
			\draw[] (redui3) -- (4, -3);
			\draw[] (redui4) -- (4, -4);
			\draw[] (redui5) -- (4, -5);
			\draw[->] (redui6.west) -- (firstreduction.east);
		\end{tikzpicture}
	\end{center}
	\caption{Overview of theorems and lemmas in the main proof.}\label{fig:proof:overview}
\end{figure}
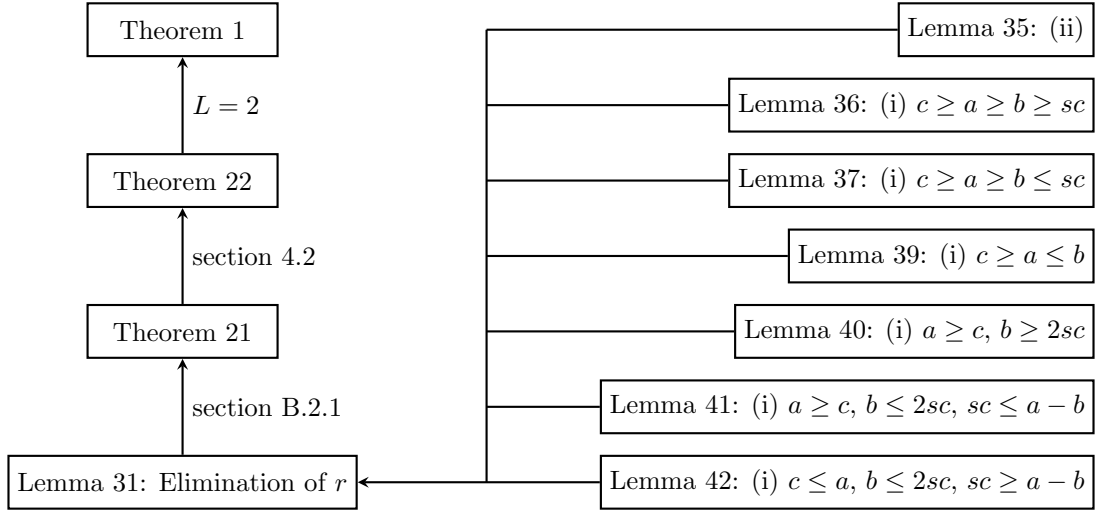

\section{Corollaries and Discussion}\label{sec:impl}
With \cref{thm:main} and \cref{thm:ipsymm}, we have shown new fundamental inequalities in metric spaces and inner product spaces that are related to the Cauchy--Schwarz and the Triangle inequalities. In this section, we discuss some corollaries of these results.
\subsection{Symmetries}
\begin{figure}
	\begin{center}
		\begin{tikzpicture}[thick]
			\coordinate (y) at (-1,0);
			\coordinate (q) at (7,0);
			\coordinate (z) at (4.5,3);
			\coordinate (p) at (1,4);
			
			\draw[colQuadSidePos] (y) -- node[below] {$\ol yq$} (q);
			\draw[colQuadSideNeg] (q) -- node[right] {$\ol zq$} (z);
			\draw[colQuadSidePos] (z) -- node[above] {$\ol zp$} (p);
			\draw[colQuadSideNeg] (p) -- node[left] {$\ol yp$} (y);
			\draw[colQuadDiag, dashed] (y) -- node[left] {$\ol yz$} (z);
			\draw[colQuadDiag, dotted] (q) -- node[below] {$\ol qp$} (p);
			
			\fill (y) circle[radius=2pt] node[below] {$y$};
			\fill (q) circle[radius=2pt] node[below] {$q$};
			\fill (z) circle[radius=2pt] node[above] {$z$};
			\fill (p) circle[radius=2pt] node[left] {$p$};
		\end{tikzpicture}
	\end{center}
	\caption{Four points as a quadrilateral. The sides of the quadrilateral show up on the left of the quadruple inequalities (the terms of opposite sides have the same sign); the diagonals form the right-hand side.}\label{fig:quadrilateral:symm}
\end{figure}
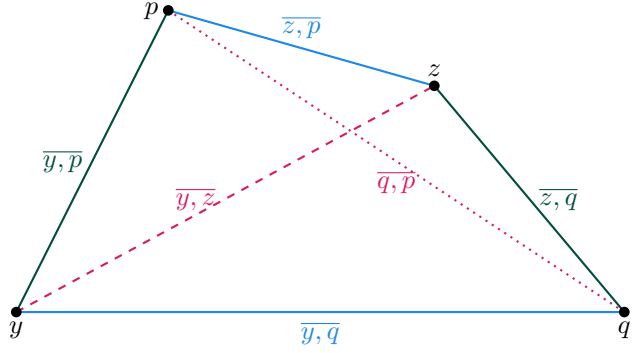
\cref{fig:quadrilateral:symm} illustrates the symmetries in quadruple inequalities. Sides of the same color contribute essentially in the same way in the inequality. In \eqref{eq:quadtran}, the diagonals of \cref{fig:quadrilateral:symm} come up in non-exchangeable terms. But they can be swapped in the assumption \eqref{eq:quad2}. 
Thus, if the conditions of \cref{thm:main} are fulfilled, we have, for $\tran\in\setcc$,
\begin{equation}
	\tran(\ol yq) - \tran(\ol yp) - \tran(\ol zq) + \tran(\ol zp) \leq 2 \min(\ol qp\,\dtran(\ol yz), \ol yz\, \dtran(\ol qp))
	\eqfs
\end{equation}
Furthermore, swapping $y$ and $z$ or $q$ and $p$ does not influence the right-hand side but changes the sign on the left-hand side. Thus, assuming
\begin{equation}
	\abs{\ol yq^2 - \ol yp^2 - \ol zq^2 + \ol zp^2} \leq 2\, \ol qp \, \ol yz
\end{equation}
we get
\begin{equation}
	\abs{\tran(\ol yq) - \tran(\ol yp) - \tran(\ol zq) + \tran(\ol zp)} \leq 2 \min(\ol qp\,\dtran(\ol yz), \ol yz\, \dtran(\ol qp))
	\eqfs
\end{equation}
Further bounds of the right-hand side are shown in the next subsection.
\subsection{Bounds for the Right-Hand Side}
\begin{corollary}\label{cor:rhsbound}
	Let $\tran\in\setcco$. Let $(\mc Q, d)$ be a metric space. Let $y,z,q,p\in\mc Q$. Assume
	\begin{equation}
		\ol yq^2 - \ol yp^2 - \ol zq^2 + \ol zp^2 \leq 2\, \ol qp \, \ol yz
		\eqfs
	\end{equation}
	Then the value 
	\begin{equation}
		\tran(\ol yq) - \tran(\ol yp) - \tran(\ol zq) + \tran(\ol zp)
	\end{equation}
	is bounded from above by all of the following values:
	\begin{enumerate}[label=(\roman*)]
		\item\label{cor:rhsbound:min} $2 \min(\ol qp,\, \ol yz) \dtran\brOf{\max(\ol qp,\, \ol yz)}$,
		\item\label{cor:rhsbound:geo} $2 \,\ol qp^\beta \ol yz^{1-\beta} \dtran\brOf{\ol qp^{1-\beta} \ol yz^\beta}$ for all $\beta\in[0,1]$,
		\item\label{cor:rhsbound:ari} $2 \br{\beta\,\ol qp + (1-\beta)\ol yz} \dtran\brOf{(1-\beta)\ol qp + \beta\ol yz}$ for all $\beta\in[0,1]$,
		\item\label{cor:rhsbound:sqrt} $2 \sqrt{\ol qp\, \ol yz} \dtran\brOf{\sqrt{\ol qp\, \ol yz}}$,
		\item\label{cor:rhsbound:avg} $\br{\ol qp + \ol yz} \dtran\brOf{\frac{\ol qp + \ol yz}2}$,
		\item\label{cor:rhsbound:tran} $4\tran\brOf{\sqrt{\ol qp\, \ol yz}}$,
		\item\label{cor:rhsbound:sum} $4\tran\brOf{\frac{\ol qp + \ol yz}2}$,
		\item\label{cor:rhsbound:transum} $2\tran(\ol qp) + 2\tran(\ol yz)$.
	\end{enumerate}
\end{corollary}
\begin{proof}
	We first apply \cref{thm:main} twice, to $(y,z,q,p)$ and  to $(q,p,y,z)$, to obtain
	\begin{equation}\label{eq:cor:rhsbound:min}
		\tran(\ol yq) - \tran(\ol yp) - \tran(\ol zq) + \tran(\ol zp) \leq 2 \min(\ol qp,\, \ol yz) \dtran\brOf{\max(\ol qp,\, \ol yz)}
		\eqfs
	\end{equation}
	This shows \ref{cor:rhsbound:min}.
	Let $a,b\in\Rp$ and $\beta\in [0,1]$. Next we use \cref{lmm:tranconcave} \ref{lmm:tranconcave:factor} and the weighted arithmetic--geometric mean inequality, 
	\begin{align*}
		\min(a,b) \dtran(\max(a,b)) 
		&\leq 
		a^\beta b^{1-\beta} \dtran(a^{1-\beta} b^\beta)
		\\&\leq 
		\br{\beta a + (1-\beta)b} \dtran((1-\beta)a + \beta b)
		\eqfs
	\end{align*}
	Applying these inequalities to \eqref{eq:cor:rhsbound:min} shows \ref{cor:rhsbound:geo} and \ref{cor:rhsbound:ari}, and their special cases \ref{cor:rhsbound:sqrt} and \ref{cor:rhsbound:avg}, where $\beta=1/2$.
	By \cref{lmm:ccdiff} \ref{lmm:ccdiff:tight} with $y=0$, we have 
	\begin{equation}\label{eq:xdtrantran}
		x\dtran(x) \leq 2 \tran(x)
	\end{equation}
	for all $x\in\Rp$. Thus,
	\begin{align*}
		\min(a,b) \dtran(\max(a,b)) 
		&\leq 
		\sqrt{a b} \,\dtran(\sqrt{a b})
		\\&\leq 
		2\tran(\sqrt{a b})
		\eqcm
	\end{align*}
	which yields \ref{cor:rhsbound:tran}. The remaining parts \ref{cor:rhsbound:sum} and \ref{cor:rhsbound:transum}, can be obtained using \eqref{eq:xdtrantran} and Jensen's inequality:
	\begin{align*}
		\min(a,b) \dtran(\max(a,b)) 
		&\leq 
		\frac{a+b}2 \dtran\brOf{\frac{a+b}2}
		\\&\leq 
		2\tran\brOf{\frac{a+b}2}
		\\&\leq 
		\tran(a) + \tran(b)
		\eqfs
	\end{align*}
\end{proof}
\subsection{Corollaries for Special Cases}
We apply \cref{thm:main}, \cref{thm:ipsymm}, \cref{prp:powerinequ}, and \cref{cor:rhsbound} for a triple of points (a quadruple of points with two identical points), on the real line, and for parallelograms in inner product spaces to demonstrate the main results.
\begin{corollary}[For three points]
	Let $(\mc Q, d)$ be a metric space. Let $y, q, p \in \mc Q$.
	\begin{enumerate}[label=(\roman*)]
		\item Let $\tran\in\setcco$. Then
		\begin{equation}\label{cor:three:tran}
			\tran(\ol yq) - \tran(\ol yp) + \tran(\ol qp)
			\leq 2 \, \ol qp \, \dtran(\ol yq) 
			\eqfs
		\end{equation}
		\item Let $\alpha\in[1,2]$. Then
		\begin{equation}\label{cor:three:pow}
			\ol yq^\alpha - \ol yp^\alpha + \ol qp^\alpha
			\leq \alpha 2^{2-\alpha} \, \ol qp \, \ol yq^{\alpha-1}
			\eqfs
		\end{equation}
	\end{enumerate}
\end{corollary}
\begin{proof}
	By the triangle inequality, $\abs{\ol yq - \ol qp} \leq \ol yp$. After squaring this inequality, we obtain \eqref{eq:quad2} with $z=q$, which is
	\begin{equation}
		\ol yq^2 - \ol yp^2 + \ol qp^2 \leq 2 \,\ol qp \, \ol yq
		\eqfs
	\end{equation}
	Thus, \cref{thm:main} implies \eqref{cor:three:tran} and \cref{prp:powerinequ} implies \eqref{cor:three:pow}.
\end{proof}
\begin{corollary}[On the real line]
	Let $a,b,c\in\Rp$.
	\begin{enumerate}[label=(\roman*)]
		\item Let $\tran\in\setcco$. Then
		\begin{align}
			\label{eq:real:quad}
			\tran(a+b+c) - \tran(a+b) - \tran(b+c) + \tran(b) 
			&\leq 2 c \dtran(a) 
			\\
			\label{eq:real:symm}
			\tran(a+b+c) - \tran(a+b) - \tran(b+c) + \tran(b)  &\leq \tran(c) + \tran(a)
			\eqfs
		\end{align}
		\item Let $\alpha\in[1,2]$. Then
		\begin{align}\label{eq:real:pow}
			(a+b+c)^\alpha - (a+b)^\alpha - (b+c)^\alpha + b^\alpha
			&\leq \alpha 2^{2-\alpha} c a^{\alpha-1} 
			\\
			\label{eq:real:powsymm}
			(a+b+c)^\alpha - (a+b)^\alpha - (b+c)^\alpha + b^\alpha
			&\leq c^\alpha + a^\alpha
			\eqfs
		\end{align}
	\end{enumerate}
\end{corollary}
\begin{proof}
	In $(\R, \abs{\cdot - \cdot})$, \eqref{eq:quad2} is fulfilled. Choose $y = 0$, $z = a$, $p = a+b$, $q=a+b+c$ and apply \cref{thm:main} and \cref{prp:powerinequ} to obtain \eqref{eq:real:quad}. For \eqref{eq:real:symm}, apply \cref{thm:ipsymm}. To get \eqref{eq:real:pow}, use \cref{prp:powerinequ}. The equation \eqref{eq:real:powsymm} is \eqref{eq:real:symm} with $\tran = \tran_\alpha = (x\mapsto x^\alpha)$.
\end{proof}
\begin{corollary}[In inner product spaces]\label{cor:ip}
	Let $(V, \ip{\cdot}{\cdot})$ be an inner product space with induced norm $\normof{\cdot}$. Let $u, v \in V$.
	\begin{enumerate}[label=(\roman*)]
		\item\label{cor:ip:a:tran} Let $\tran\in\setcco$. Then
		\begin{equation}
			\tran(\normof{u}) - \tran(\normof{v})
			\leq \normof{u-v} \dtran(\normof{u+v}) 
			\eqfs
		\end{equation}
		\item\label{cor:ip:a:pow} Let $\alpha\in[1,2]$. Then
		\begin{equation}
			\normof{u}^\alpha - \normof{v}^\alpha
			\leq \alpha 2^{1-\alpha}  \normof{u-v} \normof{u+v}^{\alpha-1}
			\eqfs
		\end{equation}
		\item\label{cor:ip:b:tran} Let $\tran\in\setcco$. Then
		\begin{align}
			\tran(\normof{u+v}) + \tran(\normof{u-v}) &\leq  2 \tran(\normof{u}) + 2 \tran(\normof{v}) 
			\eqfs
		\end{align}
		\item\label{cor:ip:b:pow} Let $\alpha\in[1,2]$. Then
		\begin{equation}
			\normof{u+v}^\alpha + \normof{u-v}^\alpha
			\leq 2 \normof{u}^\alpha + 2 \normof{v}^\alpha
			\eqfs
		\end{equation}
	\end{enumerate}
\end{corollary}
\begin{proof}
	For the first two parts, set $y = 0$, $z = u+v$, $q = u$, $p = v$; for the last two parts, set $y = 0$, $z = v$, $q = u+v$, $p = u$. Then apply \cref{thm:main}, \cref{thm:ipsymm}, and \cref{prp:powerinequ}.
\end{proof}
\begin{remark}\label{rem:subopti:euclid}\mbox{ }
	Recall the \textit{parallelogram law}: Let $(V, \ip{\cdot}{\cdot})$ be an inner product space with induced norm $\normof{\cdot}$. Let $u, v \in V$. Then
	\begin{equation}
		\normof{u+v}^2 + \normof{u-v}^2 \leq 2 \normof{u}^2 + 2\normof{v}^2\eqcm
	\end{equation}
	which is also true with equality.
	Thus, we can say that \cref{cor:ip} generalizes the parallelogram law.
\end{remark}
\subsection{Quadruple Constants and Further Research}
We summarize our findings on the quadruple inequality and quadruple constant for symmetric and non-symmetric right-hand side, and for power functions as well as nondecreasing, convex functions with concave derivative.

Let $A \subset \Rpp^6$ be the set of points $x = (x_1, \dots, x_6)$ such that there is a metric space $(\{y_x,z_x,q_x,p_x\}, d)$ with
\begin{equation}
	x_1 = \ol {y_x}{q_x}
	\eqcm\quad
	x_2 = \ol {y_x}{p_x}
	\eqcm\quad
	x_3 = \ol {z_x}{q_x}
	\eqcm\quad
	x_4 = \ol {z_x}{p_x}
	\eqcm\quad
	x_5 = \ol {q_x}{p_x}
	\eqcm\quad
	x_6 = \ol {y_x}{z_x}
\end{equation} 
that fulfills \eqref{eq:quad2}.
For $\alpha\in\Rpp$, define
\begin{align*}
	L_\alpha &:= \sup_{x\in A}  \frac{\ol {y_x}{q_x}^\alpha - \ol {y_x}{p_x}^\alpha - \ol {z_x}{q_x}^\alpha + \ol {z_x}{p_x}^\alpha}{\ol {q_x}{p_x} \,\ol {y_x}{z_x}^{\alpha-1}}\eqcm\\
	K_\alpha &:= \sup_{x\in A}  \frac{\ol {y_x}{q_x}^\alpha - \ol {y_x}{p_x}^\alpha - \ol {z_x}{q_x}^\alpha + \ol {z_x}{p_x}^\alpha}{\ol {q_x}{p_x}^{\frac\alpha2} \,\ol {y_x}{z_x}^{\frac\alpha2}}\eqcm\\
	J_\alpha &:= \sup_{x\in A}  \frac{\ol {y_x}{q_x}^\alpha - \ol {y_x}{p_x}^\alpha - \ol {z_x}{q_x}^\alpha + \ol {z_x}{p_x}^\alpha}{\ol {q_x}{p_x}^\alpha + \ol {y_x}{z_x}^\alpha}\eqfs
\end{align*}
Furthermore, define
\begin{align*}
	L_{\setcco} &:= \sup_{\tran\in\setcco} \sup_{x\in A} \frac{\tran(\ol {y_x}{q_x}) - \tran(\ol {y_x}{p_x}) - \tran(\ol {z_x}{q_x}) + \tran(\ol {z_x}{p_x})}{\ol {q_x}{p_x} \,\dtran(\ol {y_x}{z_x})}\eqcm\\
	K_{\setcco} &:= \sup_{\tran\in\setcco} \sup_{x\in A} \frac{\tran(\ol {y_x}{q_x}) - \tran(\ol {y_x}{p_x}) - \tran(\ol {z_x}{q_x}) + \tran(\ol {z_x}{p_x})}{\tran(\sqrt{\ol {q_x}{p_x}\,\ol {y_x}{z_x}})}\eqcm\\
	J_{\setcco} &:= \sup_{\tran\in\setcco} \sup_{x\in A} \frac{\tran(\ol {y_x}{q_x}) - \tran(\ol {y_x}{p_x}) - \tran(\ol {z_x}{q_x}) + \tran(\ol {z_x}{p_x})}{\tran(\ol {q_x}{p_x}) + \tran(\ol {y_x}{z_x})}\eqfs
\end{align*}
\begin{proposition}\label{prp:summary}\mbox{ }
	\begin{enumerate}[label=(\roman*)]
		\item $L_\alpha = \alpha 2^{2-\alpha}$ for $\alpha\in[1,2]$ and $L_\alpha = \infty$ for $\alpha\in\Rpp\setminus[1,2]$.
		\item $K_\alpha = 2$ for $\alpha\in(0,1]$, $K_\alpha \in [2, \alpha 2^{2-\alpha}]$ for $\alpha\in[1,2]$, $K_\alpha = \infty$ for $\alpha\in(2,\infty)$. 
		\item $J_\alpha = 1$ for $\alpha\in(0,1]$, $J_\alpha \in [1, \alpha 2^{1-\alpha}]$ for $\alpha\in[1,2]$, $J_\alpha = \infty$ for $\alpha\in(2,\infty)$.
		\item $L_{\setcco} = 2$.
		\item $K_{\setcco} \in [2, 4]$.
		\item $J_{\setcco} \in [1, 2]$.
	\end{enumerate}
\end{proposition}
\begin{table}
	\begin{center}
		\begin{tabular}{c||c|c|c|c}
			$\bullet$ & $\alpha\in(0,1]$ & $\alpha\in[1,2]$ & $\alpha\in(2,\infty)$ & $\setcco$\\
			\hline
			\hline
			$L_\bullet$ & $\infty$ & $\alpha 2^{2-\alpha}$ & $\infty$ & $2$ \\
			$K_\bullet$ & $2$ & $[2, \alpha 2^{2-\alpha}]$ & $\infty$ & $[2, 4]$ \\
			$J_\bullet$ & $1$ & $[1, \alpha 2^{1-\alpha}]$ & $\infty$ & $[1, 2]$ \\
		\end{tabular}
	\end{center}
	\caption{Range of quadruple constants shown in \cref{prp:summary}.}\label{table:prp:summary}
\end{table}
\begin{figure}
	\begin{center}
		\begin{tikzpicture}[thick]
			\coordinate (y) at (0,0);
			\coordinate (q) at (4,1.5);
			\coordinate (z) at (0,1.5);
			\coordinate (p) at (4,0);
			
			\draw[colQuadSidePos] (y) -- node[near start,above,sloped] {$\sqrt{\epsilon^2 + 1}$} (q);
			\draw[colQuadSideNeg] (q) -- node[above] {$1$} (z);
			\draw[colQuadSidePos] (z) -- node[near end,above,sloped] {$\sqrt{\epsilon^2+1}$} (p);
			\draw[colQuadSideNeg] (p) -- node[below] {$1$} (y);
			\draw[colQuadDiag] (y) -- node[left] {$\epsilon$} (z);
			\draw[colQuadDiag] (q) -- node[right] {$\epsilon$} (p);
			
			\fill (y) circle[radius=2pt] node[below] {$y$};
			\fill (q) circle[radius=2pt] node[right] {$q$};
			\fill (z) circle[radius=2pt] node[left] {$z$};
			\fill (p) circle[radius=2pt] node[below] {$p$};
		\end{tikzpicture}
	\end{center}
	\caption{Construction in the proof of \cref{prp:summary}.}\label{fig:quadriRect}
\end{figure}
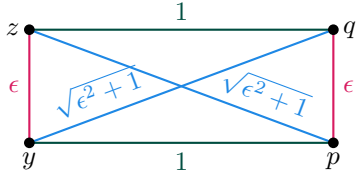
\begin{proof}
	From \cref{prp:powerinequ}, we know $L_\alpha = \alpha 2^{2-\alpha}$ for $\alpha\in[1,2]$ and $L_\alpha = \infty$ for $\alpha\in\Rpp\setminus[1,2]$. As direct consequences, we obtain $K_\alpha \leq \alpha 2^{2-\alpha}$ and $J_\alpha \leq \alpha 2^{1-\alpha}$ for $\alpha\in[1,2]$.
	
	If we set $y = p$ and $z = q$ and assume $\ol yq = a \in \Rpp$, we have 
	\begin{align*}
		\tran(\ol yq) - \tran(\ol yp) - \tran(\ol zq) + \tran(\ol zp) &= 2 \tran(a)\eqcm\\
		\tran(\sqrt{\ol {q}{p}\,\ol {y}{z}})  &=  \tran(a) \eqcm\\
		\tran(\ol {q}{p}) + \tran(\ol {y}{z})  &=  2\tran(a) \eqcm
	\end{align*}
	for any function $\tran\colon\Rpp\to\R$ with $\tran(0) = 0$.
	Thus, $K_\alpha, K_{\setcco} \geq 2$, $J_\alpha, J_{\setcco} \geq 1$.
	
	If $\tran$ is metric preserving, like $\tran_\alpha = (x\mapsto x^\alpha)$ with $\alpha\in(0,1]$, then 
	\begin{equation}
		\tran(\ol yq) - \tran(\ol yp) - \tran(\ol zq) + \tran(\ol zp) \leq 2\min\brOf{\tran(\ol qp), \tran(\ol yz)}
		\eqfs
	\end{equation}
	This shows $K_\alpha \leq 2$ and $J_\alpha \leq 1$ for $\alpha\in(0,1]$.
	
	Let $y = (0,0)$, $z = (0,\epsilon)$, $q = (1,\epsilon)$, $p = (1, 0)$ in the Euclidean plane $\R^2$, see \cref{fig:quadriRect}. Then 
	\begin{align*}
		\tran(\ol yq) - \tran(\ol yp) - \tran(\ol zq) + \tran(\ol zp) &= 2 \tran(\sqrt{\epsilon^2+1}) - 2 \tran(1)\eqcm\\
		\tran(\sqrt{\ol {q}{p}\,\ol {y}{z}})  &=  \tran(\epsilon) \eqcm\\
		\tran(\ol {q}{p}) + \tran(\ol {y}{z})  &=  2\tran(\epsilon) \eqcm
	\end{align*}
	for any function $\tran\colon\Rp\to\R$.
	Assume $\tran(0) = 0$, $\dtran(1) > 0$, and $\lim_{\epsilon\searrow0} \frac{\epsilon}{\dtran(\epsilon)} = \infty$. Then, by l'Hôpital's rule, 
	\begin{equation}
		\lim_{\epsilon\searrow0} \frac{\tran(\sqrt{\epsilon^2+1}) - \tran(1)}{\tran(\epsilon)}
		= 
		\lim_{\epsilon\searrow0} \frac{\epsilon \dtran(\sqrt{\epsilon^2+1})}{\sqrt{\epsilon^2+1}\,\dtran(\epsilon)}
		= \infty\eqfs
	\end{equation}
	In particular, $K_\alpha = \infty$ and $J_\alpha = \infty$ for $\alpha\in(2,\infty)$.

	By \cref{thm:main} and \cref{cor:rhsbound}, we have $L_{\setcco} \leq 2$, $K_{\setcco} \leq 4$, and $J_{\setcco} \leq 2$. As here we take the supremum over $\tran\in\setcco$, we also get $L_{\setcco} \geq 2$, e.g., for $\tran = \tran_1 = (x\mapsto x)$.
\end{proof}
\begin{remark}\mbox{ }
	We have $\alpha 2^{1-\alpha} \in [1 , \frac{2}{\euler\ln(2)}]$ for $\alpha\in[1,2]	$ and $\frac{2}{\euler\ln(2)} \approx 1.06$. The maximum is attained at $\alpha = \ln(2)^{-1}\approx 1.44$.
\end{remark} 
For future work it remains to find the precise values of $K_\alpha$ and $J_\alpha$ for $\alpha\in[1,2]$, and for $K_{\setcco}$ and $J_{\setcco}$.
Furthermore, it would be interesting to extend the result for an explicit form of the quadruple constant from the functions $\tran_\alpha = (x\mapsto x^\alpha)$, where we have $\ctranobest{\tran_\alpha} = 2^{2-\alpha}$ for $\alpha\in[1,2]$, to functions $\tran\in\setcc$, where we so far know $\ctranbest\in[1,2]$.

\begin{appendix}
	\section{Proof of \cref{thm:ipsymm}}\label{app:proofsymm}
The proof is inspired by the proofs of \cite[Proposition 4.1.1, Proposition 4.1.2]{enflo70}.

For any four points $y, z, q, p\in V$, we have 
\begin{equation}\label{eq:proofsymm:symm2}
	\normof{y-q}^2  - \normof{y-p}^2 - \normof{z-q}^2 + \normof{z-p}^2 
	= 
	2 \ip{y-z}{p-q} 
	\leq 
	\normof{q-p}^2 + \normof{y-z}^2
	\eqfs
\end{equation}
Let $u, v \in V$. Consider a parallelogram with vertices $(0, (u-v)/2, u, (u+v)/2)$. It has the diagonals $u$ and $v$ and the largest diagonal is not smaller than the largest side.
As $\tran\in\setcco$, $\trans$ is nonnegative, nondecreasing, and concave, see \cref{lmm:ccsqrtprop}. Thus, we can apply \cref{lmm:aux:six} to
\begin{equation}
	x_1 = x_2 = \normOf{\frac{u-v}2}^2\eqcm\quad
	x_3 = x_4 = \normOf{\frac{u+v}2}^2\eqcm\quad
	x_5 = \normof{u}^2\eqcm\quad
	x_6 = \normof{v}^2\eqcm
\end{equation}
where $x_1+x_2+x_3+x_4 \geq x_5+x_6$ is ensured by \eqref{eq:proofsymm:symm2}.
We obtain
\begin{equation}
	\tran(\normof{u})  + \tran(\normof{v}) \leq 2\tran\brOf{\normOf{\frac{u-v}2}} + 2\tran\brOf{\normOf{\frac{u+v}2}}
	\eqfs
\end{equation}
To extend the result from parallelograms to any quadrilateral, we note that $\tran$ is nondecreasing and convex, and apply \cref{lmm:aux:redistri}:
For every $x\in V$, 
\begin{align}
	2\tran\brOf{\normOf{\frac{u-v}2}} &\leq \tran\brOf{\normOf{x}} + \tran\brOf{\normOf{u-v-x}}\eqcm\\
	2\tran\brOf{\normOf{\frac{u+v}2}} &\leq \tran\brOf{\normOf{u-x}} + \tran\brOf{\normOf{v+x}}
	\eqfs
\end{align}%
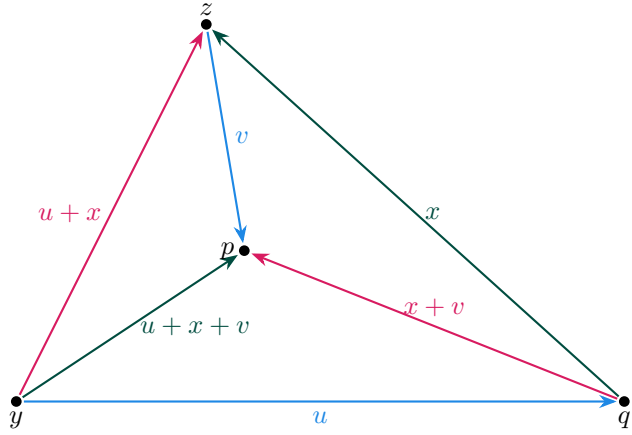
\begin{figure}%
	\begin{center}
		\begin{tikzpicture}[
			thick,
			myArrows/.style={shorten >=0.1cm,shorten <=0.1cm,-{Stealth}}]
			\coordinate (y) at (-1,0);
			\coordinate (q) at (7,0);
			\coordinate (z) at (1.5,5);
			\coordinate (p) at (2,2);
			
			\draw[colQuadSidePos,myArrows] (y) -- node[below] {$u$} (q);
			\draw[colQuadSideNeg,myArrows] (q) -- node[right] {$x$} (z);
			\draw[colQuadSidePos,myArrows] (z) -- node[right] {$v$} (p);
			\draw[colQuadSideNeg,myArrows] (y) -- node[right] {$u+x+v$} (p);
			\draw[colQuadDiag,myArrows] (y) -- node[left] {$u+x$} (z);
			\draw[colQuadDiag,myArrows] (q) -- node[above] {$x+v$} (p);
			
			\fill (y) circle[radius=2pt] node[below] {$y$};
			\fill (q) circle[radius=2pt] node[below] {$q$};
			\fill (z) circle[radius=2pt] node[above] {$z$};
			\fill (p) circle[radius=2pt] node[left] {$p$};
		\end{tikzpicture}
	\end{center}
	\caption{Four points in a vector space $V$. Their relative position is described by three vectors $u,v,x\in V$.}\label{fig:symm:uvx}
\end{figure}%
By appropriate choice of $u, v, x$ for a given quadrilateral with vertices $y,z,q, p$, see \cref{fig:symm:uvx}, we have shown
\begin{equation}
	\tran(\normof{y-q}) + \tran(\normof{z-p}) \leq \tran(\normof{q-p}) + \tran(\normof{y-z})  +  \tran(\normof{y-p}) + \tran(\normof{z-q})
\end{equation}
and finished the proof of \cref{thm:ipsymm}.

	\section{Proof of \cref{thm:main:param}}\label{app:mainproof}
If we want to show $f(x)\leq0$ for all $x\in A$, where $A\subset \R^n$, for a continuous function $f\colon A\to \R$, it is enough to prove the inequality on a dense subset of $A$. We use this fact in the following. If we write an expression with a quotient, we silently restrict the domains of the real parameters in all statements about this expression to a domain on which the denominator is not 0. The restricted domain, will always be dense in the unrestricted domain.
\subsection{Summary of Properties}
Recall $\trans(x) = \tran(\sqrt x)$. The next lemma shows some simple properties of $\trans$ and its derivatives.
\begin{lemma}\label{lmm:ccsqrtprop}
	Let $\tran\in\setcco$.
	\begin{itemize}
		\item Then $\trans$ is nonnegative, nondecreasing, and concave.
		\item If $\trans$ is differentiable, then $\trans\pr$ is nonnegative and nonincreasing.
		\item If $\trans$ is twice differentiable, then $\trans\prr$ is nonpositive.
	\end{itemize}
\end{lemma}
\begin{proof}
	As $\tran$ is nonnegative, so is $\trans$.
	We have 
	\begin{align*}
		\dtrans(x) &= \frac{\dtran(\sqrt{x})}{2 \sqrt{x}}\eqcm\\
		\ddtrans(x) &= \frac{\ddtran(\sqrt{x})-\frac{\dtran(\sqrt{x})}{\sqrt{x}}}{4x}\eqfs
	\end{align*}
	As $\dtran$ is nonnegative, so is $\dtrans$. Thus, $\trans$ is increasing.
	Furthermore, as $\ddtran$ is nonincreasing, 
	\begin{equation}
		\dtran(\sqrt{x}) - \dtran(0) = \int_0^{\sqrt{x}} \ddtran(u) \dl u \geq \sqrt{x} \ddtran(\sqrt{x}) \eqfs
	\end{equation}
	Thus, with $\dtran(0) \geq 0$, we obtain  $\dtran(\sqrt{x}) \geq \sqrt{x} \ddtran(\sqrt{x})$. Hence, $\ddtrans$ is nonpositive, $\dtrans$ is nonincreasing, and $\trans$ is concave.
\end{proof}
Properties of $\tran\in\setccod3$ (\cref{lmm:ccSimpleProps}) and the corresponding $\trans$ (\cref{lmm:ccsqrtprop}) are summarized in \cref{table:porps} for reference. There and in the proof below, we use the following shorthand notation for properties of functions $f \colon \Rp \to \R$:

\vspace{\topsep+\parskip}
\noindent\begin{minipage}{0.49\textwidth}%
\begin{itemize}%
	\item $f \posi$: $f$ is nonnegative.
	\item $f \incr$: $f$ is nondecreasing.
	\item  $f \conv$: $f$ is convex.
\end{itemize}%
\end{minipage}%
\begin{minipage}{0.49\textwidth}%
	\begin{itemize}%
		\item $f \nega$: $f$ is nonpositive.
		\item $f \decr$: $f$ is nonincreasing.
		\item $f \conc$: $f$ is concave.
	\end{itemize}%
\end{minipage}%
\vspace{\topsep+\parskip}
\begin{table}
	\begin{center}
		\begin{tabular}{c||c|c|c|c||c|c|c|c}
			$f$&$\tran$ & $\dtran$ & $\ddtran$ & $\dddtran$ & $\trans$ & $\dtrans$ & $\ddtrans$ & $\dddtrans$\\
			\hline
			\hline
			$f(0)$ & $0$ & $\geq0$ & $\geq0$ & $\leq0$ & $0$ & $\geq0$ & $\leq0$ & \\
			\hline
			$f(x)$ & $\geq0$ & $\geq0$ & $\geq0$ & $\leq0$ & $\geq0$ & $\geq0$ & $\leq0$ &  \\
			\hline
			$f$ monotone & $\nearrow$ & $\nearrow$ & $\searrow$ &  & $\nearrow$ & $\searrow$ &  &  \\
			\hline
			$f$ curvature & $\smallsmile$ & $\smallfrown$ & &  & $\smallfrown$ &  &  & 
		\end{tabular}
	\end{center}
	\caption{Properties of $\tran\in\setccod3$ and function $\trans(x) = \tran(\sqrt x)$.}\label{table:porps}
\end{table}
\subsection{\cref{lmm:param:firstreduction}: Elimination of $r$}
We will show that the following lemma implies \cref{thm:main:param}.
\begin{lemma}[Elimination of $r$]\label{lmm:param:firstreduction}
	Let $\tran\in\setccod3$.
	\begin{enumerate}[label=(\roman*)]
		\item \label{lmm:param:firstreduction:a}
		For all $a,b,c\in\Rp$, $s\in [-1, \min(1, 2\frac ac -1)]$, we have
		\begin{equation}
			\tran(a) - \tran(c) -\tran(\abs{a-b}) + \trans(c^2 - 2scb + b^2) \leq 2b \dtran(a-sc)\
			\eqfs
		\end{equation}
		\item \label{lmm:param:firstreduction:b}
		For all $a,b,c\in\Rp$ with $a\geq c$, we have
		\begin{equation}
			\tran(a) - \tran(c) -\trans((a-b)^2 + 4bc) + \tran(b + c) \leq 2b \dtran(a-c)
			\eqfs
		\end{equation}
	\end{enumerate}
\end{lemma}
\subsubsection{Proof that \cref{lmm:param:firstreduction} implies \cref{thm:main:param}}\label{sssec:firstReduImpliesMainParam}
For this proof, we first show some auxiliary lemmas.
We distinguish the cases $ra-sc \leq |a-c|$ and $ra-sc \geq |a-c|$ as well as $a\geq c$ and $c \geq a$. Some trivial implications of these cases are recorded in the following lemma. 
\begin{lemma}\label{lmm:srbound}
	Let $a,b,c\geq0$, $r,s\in[-1,1]$.
	Then
	\begin{align*}
		&ra - sc \geq a-c
		\quad\Leftrightarrow\quad 
		s 
		\leq
		(r-1) \frac ac+1
		\quad\Leftrightarrow\quad 
		r \geq 
		(s-1)\frac ca +1
		\eqcm
		\\
		&ra - sc \geq c-a
		\quad\Leftrightarrow\quad 
		s 
		\leq
		(r+1) \frac ac -1
		\quad\Leftrightarrow\quad 
		r \geq 
		(s+1)\frac ca -1
		\eqfs
	\end{align*}
\end{lemma}
Denote 
\begin{align*}
	\lhs(a,b,c,r,s) &:= \tran(a) - \tran(c) - \trans\brOf{a^2-2rab+b^2} + \trans\brOf{c^2-2scb+b^2}\eqcm\\
	\ifu(a,b,c,r,s) &:= \lhs(a,b,c,r,s) - 2 b \dtran\brOf{ra - sc}
	\eqfs
\end{align*}
For $ra - sc \geq \abs{a-c}$, we want to show $\ifu(a,b,c,r,s) \leq 0$. Because of the next lemma, we can reduce the number of values of $r$ for which we need to check this inequality.
\begin{lemma}[Convexity in $r$]\label{lmm:ddr}
	Let $\tran\in\setccod3$. Let $a,b,c\geq0$, $s,r\in[-1,1]$.
	Assume $ra - sc \geq 0$.
	Then
	\begin{equation*}
		\partial_r^2 \ifu(a,b,c,r,s) \geq 0
		\eqfs
	\end{equation*}
\end{lemma}
\begin{proof}
	As $\ddtrans,\dddtran\leq0$, we have
	\begin{align*}
		\partial_r^2 \trans\brOf{a^2-2rab+b^2} = 4a^2b^2\ddtrans(a^2-2rab+b^2) \leq 0\eqcm
		\partial_r^2 \dtran(ra - sc) = a^2 \dddtran(ra - sc) \leq 0
		\eqfs
	\end{align*}
	Thus, 
	\begin{equation*}
		\partial_r^2 \ifu(a,b,c,r,s) = -\partial_r^2 \trans\brOf{a^2-2rab+b^2} - 2 b \partial_r^2\dtran\brOf{ra - sc} \geq 0
		\eqfs
	\end{equation*}
\end{proof}
In the case $|a-c| \geq ra-sc$, the right-hand side of \eqref{eq:quadtran:param} does not depend on $r$ or $s$. Thus, we will only need to check the inequality with the left-hand side $\lhs$ maximized in $r$ and $s$.
\begin{lemma}[Maximizing the left-hand side for $|a-c| \geq ra-sc$]\label{lmm:maxlhs}
	Let $\tran\in\setccod3$.  Let $a,b,c\in\Rp$, $r,s\in[-1, 1]$. Assume $|a-c| \geq ra-sc$.
	\begin{enumerate}[label=(\roman*)]
		\item If $a\geq c$ and $a^2 \leq c^2+2ab-2cb$, then
		\begin{equation*}
			\lhs(a,b,c,r,s) \leq \tran(a)-\tran(c)-\tran(|a-b|)+\tran(|c-b|)
			\eqfs
		\end{equation*}
		\item If $a\geq c$ and $a^2 \geq c^2+2ab-2cb$, then
		\begin{equation*}
			\lhs(a,b,c,r,s) \leq  \tran(a)-\tran(c)-\trans((a-b)^2+4cb)+\tran(c+b)
			\eqfs
		\end{equation*}
		\item If $a\leq c$, then
		\begin{equation*}
			\lhs(a,b,c,r,s) \leq \tran(a)-\tran(c)-\tran(\abs{a-b})+\trans((c+b)^2-4ab)
			\eqfs
		\end{equation*}
	\end{enumerate}
\end{lemma}
\begin{proof}
	As $\trans\incr$, $s\mapsto\lhs(a,b,c,r,s) \decr$ and $r\mapsto\lhs(a,b,c,r,s) \incr$, i.e., for $s_0,r_0\in[-1,1]$, 
	\begin{equation*}
		\max_{s\geq s_0, r\leq r_0} \lhs(a,b,c,r,s) = \lhs(a,b,c,r_0,s_0)
		\eqfs
	\end{equation*}
	\underline{Case 1: $a\geq c$.} For $r\in[-1,1]$, set $s_{\ms{min}}(r) := (r-1) \frac ac+1$, cf.\ \cref{lmm:srbound}. Define
	\begin{align*}
		f(r) &:= \lhs(a,b,c,r,s_{\ms{min}}(r)) 
		\\&= \tran(a)-\tran(c)-\trans(a^2-2rab+b^2)+\trans(c^2-2rab+2ab-2cb+b^2)
		\eqfs
	\end{align*}
	Then
	\begin{equation*}
		\frac{f\pr(r)}{2ab} = \dtrans(a^2-2rab+b^2) - \dtrans(c^2-2rab+2ab-2cb+b^2)
		\eqfs
	\end{equation*}
	\underline{Case 1.1: $a^2 \leq c^2+2ab-2cb$.} As $\dtrans\decr$, we have
	\begin{align*}
		a^2-2rab+b^2 &\leq c^2-2rab+2ab-2cb+b^2\eqcm\\
		\dtrans(a^2-2rab+b^2) &\geq \dtrans(c^2-2rab+2ab-2cb+b^2)
		\eqfs
	\end{align*}
	Thus, $f\pr(r)\geq0$ and $f$ is maximal at $r = r_{\ms{max}} = 1$, with $s_{\ms{min}}(r) = 1$. Hence,
	\begin{equation*}
		\lhs(a,b,c,r,s) \leq f(1) = \tran(a)-\tran(c)-\tran(|a-b|)+\tran(|c-b|)
		\eqfs
	\end{equation*}
	\underline{Case 1.2: $a^2 \geq c^2+2ab-2cb$.} As $\dtrans\decr$, we have
	\begin{align*}
		a^2-2rab+b^2 &\geq c^2-2rab+2ab-2cb+b^2\eqcm\\
		\dtrans(a^2-2rab+b^2) &\leq \dtrans(c^2-2rab+2ab-2cb+b^2) \eqfs
	\end{align*}
	Thus, $f\pr(r)\leq0$ and $f$ is maximal at $r = r_{\ms{min}} = 1-2\frac ca$, with $s_{\ms{min}}(r) = -1$. Hence,
	\begin{equation*}
		\lhs(a,b,c,r,s) \leq  f(r_{\ms{min}}) = \tran(a)-\tran(c)-\trans((a-b)^2+4cb)+\tran(c+b)
		\eqfs
	\end{equation*}
	\underline{Case 2: $a\leq c$.}  For $r\in[-1,1]$, set $s_{\ms{min}}(r) := (r+1) \frac ac-1$, cf.\ \cref{lmm:srbound}. Define
	\begin{align*}
		f(r) &:= \lhs(a,b,c,r,s_{\ms{min}}(r)) 
		\\&= \tran(a)-\tran(c)-\trans(a^2-2rab+b^2)+\trans(c^2-2rab-2ab+2cb+b^2)
		\eqfs
	\end{align*}
	Then 
	\begin{equation*}
		\frac{f\pr(r)}{2ab} = 
		\dtrans(a^2-2rab+b^2) - \dtrans(c^2-2rab-2ab+2cb+b^2)
		\eqfs
	\end{equation*}
	\underline{Case 2.1: $a^2 \leq c^2-2ab+2cb$.} As $\dtrans\decr$, we have
	\begin{align*}
		a^2-2rab+b^2 &\leq c^2-2rab-2ab+2cb+b^2\eqcm\\
		\dtrans(a^2-2rab+b^2) &\geq \dtrans(c^2-2rab-2ab+2cb+b^2)\eqfs
	\end{align*}
	Thus, $f\pr(r)\geq0$ and $f$ is maximal at $r = r_{\ms{max}} = 1$, with $s_{\ms{min}}(r) = 2\frac ac-1$. Hence,
	\begin{equation*}
		\lhs(a,b,c,r,s) \leq f(1) = \tran(a)-\tran(c)-\tran(\abs{a-b})+\trans((c+b)^2-4ab)
		\eqfs
	\end{equation*}
	\underline{Case 2.2: $a^2 \geq c^2-2ab+2cb$.} This cannot happen for $a < c$. Hence, the proof is finished.
\end{proof}
\begin{proof}[that \cref{lmm:param:firstreduction} implies \cref{thm:main:param}]\mbox{ }\\
	\underline{Case 1: $|a-c| \geq ra-sc$.}\\
	\underline{Case 1.1: $a\geq c$.} 
		We can apply \cref{lmm:maxlhs} and it suffices to show 
		\begin{equation}
			\tran(a)-\tran(c)-\tran(|a-b|)+\tran(|c-b|) \leq 2 b \dtran(a - c)
		\end{equation}
		for $a^2 \leq c^2+2ab-2cb$, and for $a^2 \geq c^2+2ab-2cb$,
		\begin{equation}
			\tran(a)-\tran(c)-\trans((a-b)^2+4cb)+\tran(c+b) \leq 2 b \dtran(a - c)
			\eqfs
		\end{equation}
		The latter is exactly \cref{lmm:param:firstreduction} \ref{lmm:param:firstreduction:b}. The former follows from \cref{lmm:param:firstreduction} \ref{lmm:param:firstreduction:a} with $s\in\{-1, 1\}$.
	\\
	\underline{Case 1.2: $a\leq c$.} 
		We can apply \cref{lmm:maxlhs} and it suffices to show 
		\begin{equation}
			\tran(a)-\tran(c)-\tran(\abs{a-b})+\trans((c+b)^2-4ab) \leq 2 b \dtran(c - a)
			\eqfs
		\end{equation}
		This follows from \cref{lmm:param:firstreduction} \ref{lmm:param:firstreduction:a} with $s = 2\frac ac-1$.
	\\
	\underline{Case 2: $|a-c| \leq ra-sc$.}\\
	\underline{Case 2.1: $a\geq c$.} 
		In this case, $\ifu(a,b,c,r,s) \leq 0$ implies \eqref{eq:thm:main:param}.
		As $r\mapsto \ifu(a,b,c,r,s)$ is convex by \cref{lmm:ddr}, it suffices to show $\ifu(a,b,c,r,s) \leq 0$ for the extreme values of $r$ in order to establish this inequality for all $r$.
		By \cref{lmm:srbound}, 
		\begin{equation}
			r\in\abOf{(s-1)\frac ca+1,\, 1}\eqfs
		\end{equation}
		For maximal $r$, by \cref{lmm:param:firstreduction} \ref{lmm:param:firstreduction:a}, we have $\ifu(a,b,c,1,s) \leq 0$.
		For minimal $r$, we have $ra-sc = a-c$. Thus, we are in case 1.1. 
	\\
	\underline{Case 2.2: $a\leq c$.} 
		As in case 2.1, it suffices show $\ifu(a,b,c,r,s) \leq 0$ for extreme values of $r$. By \cref{lmm:srbound}, 
		\begin{equation}
			r\in\abOf{(s+1)\frac ca -1,\, 1}
			\eqfs
		\end{equation}
		For maximal $r$, by \cref{lmm:param:firstreduction} \ref{lmm:param:firstreduction:a}, we have $\ifu(a,b,c,1,s) \leq 0$.
		For minimal $r$, we have $ra-sc = c-a$.  Thus, we are in case 1.2.
\end{proof}
\subsection{Proof of \cref{lmm:param:firstreduction} \ref{lmm:param:firstreduction:b}}\label{ssec:reduii}
\begin{lemma}\label{lmm:param:reduii}
	Let $\tran\in\setccod3$.
	Let $a,b,c\in\Rp$. Assume $a\geq c$. Then
	\begin{equation}
		\tran(a) - \tran(c) -\trans((a-b)^2 + 4bc) + \tran(b + c) \leq 2b \dtran(a-c)
		\eqfs
	\end{equation}
\end{lemma}
\begin{proof}
	Define
	\begin{align*}
		f(a,b,c) &:= \tran(a) - \tran(c) - \trans((a-b)^2+4cb) + \tran(c+b) - 2 b \dtran(a - c)
		\eqfs
	\end{align*}
	By \cref{mech:1} and \cref{mech:2}, $\partial_b f(a,b,c) \leq 0$. Thus, as $b\geq0$, we have 
	\begin{equation*}
		f(a,b,c) \leq f(a, 0, c) = 0
		\eqfs
	\end{equation*}
\end{proof}
\subsection{Proof of \cref{lmm:param:firstreduction} \ref{lmm:param:firstreduction:a} for $c \geq a \geq b \geq sc$}\label{ssec:reduicGaGbGsc}
\begin{lemma}\label{lmm:param:reduicGaGbGsc}
	Let $\tran\in\setccod3$.
	Let $a,b,c\geq0$, $s\in[-1,1]$.
	Assume $(s+1)c\leq 2a$, $c \geq a \geq b \geq sc$.
	Then
	\begin{equation*}
		\tran(a)-\tran(c)-\tran(a-b)+\trans(c^2-2scb+b^2) 
		\leq
		2 b \dtran(a - sc) 
		\eqfs
	\end{equation*}
\end{lemma}
\begin{proof}
	Define
	\begin{align*}
		f(a,b,c,s) &:= \tran(a)-\tran(c)-\tran(a-b)+\trans(c^2-2scb+b^2) - 2 b \dtran(a - sc)\eqcm\\
		\partial_b f(a,b,c,s) &= \dtran(a-b) + 2(b-sc)\dtrans(c^2-2scb+b^2) - 2 \dtran(a - sc) 
		\eqfs
	\end{align*}
	By \cref{mech:3}, $\partial_b f(a,b,c,s) \leq 0$. Thus, as $b \geq b_{\ms{min}} := \max(0, sc)$, we have
	$f(a,b,c,s) \leq f(a,b_{\ms{min}},c,s)$.
	If $s \leq 0$, then $b_{\ms{min}}=0$ and
	\begin{equation*}
		f(a,0,c,s)  = \tran(a)-\tran(c)-\tran(a)+\trans(c^2) = 0
		\eqfs
	\end{equation*}
	For $s\geq 0$, we have $b_{\ms{min}}=sc$. Define
	\begin{align}
		g(a,c,s) &:= f(a,sc,c,s) = \tran(a)-\tran(c)-\tran(a-sc)+\trans((1-s^2)c^2) - 2 sc\dtran(a - sc)\eqcm\\
		\partial_a  g(a,c,s) &= \dtran(a)-\dtran(a-sc) - 2 sc\ddtran(a - sc)\eqfs
	\end{align}
	Set $d:= sc$ and define
	\begin{equation}
		h(a, d) := \dtran(a)-\dtran(a-d) - 2 d\ddtran(a - d) = \partial_a  g(a,c,s)\eqfs
	\end{equation}
	Then
	\begin{equation}
		\partial_{d} h(a, d) = -\ddtran(a - d) + 2 d\dddtran(a - d) \leq 0
		\eqfs
	\end{equation}
	As $d = sc$ and $s,c\geq 0$, we obtain
	\begin{equation*}
		h(a, d) \leq h(a, 0) = 0
		\eqfs
	\end{equation*}
	Thus, $\partial_a g(a,c,s) = h(a, d) \leq 0$. Therefore, as $a \geq a_{\ms{min}} := \frac{s+1}{2}c$,
	\begin{align*}
		g(a,c,s)  
		&\leq 
		g(a_{\ms{min}},c,s)
		\\&= 
		\tran(\frac{s+1}{2}c)-\tran(c)-\tran(\frac{1-s}{2}c)+\trans((1-s^2)c^2) - 2 sc\dtran(\frac{1-s}{2}c)
		\eqfs
	\end{align*}
	Set $u := \frac12(c+sc)$, $v := \frac12(c-sc)$ with $0\leq v \leq u$  due to $s \in[0,1]$, and define
	\begin{equation}
		\ell(u, v)  :=  \tran(u)-\tran(u+v)-\tran(v)+\trans(4uv) - 2 (u-v)\dtran(v) = g(a_{\ms{min}},c,s)
		\eqfs
	\end{equation}
	By \cref{mech:5}, $\partial_u \ell(u, v) \leq 0$. Thus, as $u \geq v$, we have 
	\begin{align*}
		\ell(u, v)  
		&\leq 
		\ell(v, v) 
		\\&= 
		\tran(v)-\tran(2v)-\tran(v)+\trans(4vv) - 2(v-v)\dtran(v) 
		\\&= 0
		\eqfs
	\end{align*}
	Therefore, we obtain $f(a,sc,c,s) \leq g(a,c,s) \leq \ell(u, v) \leq 0$ and we have finally shown $f(a,b,c,s) \leq 0$.
\end{proof}
\subsection{Proof of \cref{lmm:param:firstreduction} \ref{lmm:param:firstreduction:a} for $c \geq a \geq b$, $b \leq sc$}\label{ssec:reduicGaGbLsc}
\begin{lemma}\label{lmm:param:reduicGaGbLsc}
	Let $\tran\in\setccod3$.
	Let $a,b,c\geq0$, $s\in[-1,1]$.
	Assume $a\leq c$, $(s+1)c\leq 2a$, $a\geq b$, and $b\leq sc$.
	Then
	\begin{equation*}
		\tran(a)-\tran(c)-\tran(a-b)+\trans(c^2-2scb+b^2) 
		\leq
		2 b \dtran(a - sc) 
		\eqfs
	\end{equation*}
\end{lemma}
\begin{proof}
	Define
	\begin{equation*}
		f(a,b,c,s) := \tran(a)-\tran(c)-\tran(a-b)+\trans(c^2-2scb+b^2) - 2 b \dtran(a - sc)\eqfs
	\end{equation*}
	By \cref{mech:9}, $\partial_a f(a,b,c,s) \leq 0$. Thus, as $a \geq  a_{\ms{min}} := (1+s)c/2$, we have 
	\begin{align*}
		f(a,b,c,s) 
		&\leq 
		f(a_{\ms{min}},b,c,s) 
		\\&= 
		\tran\brOf{\frac12(1+s)c} - \tran(c) - \tran\brOf{\frac12(1+s)c - b} + \trans(c^2-2scb+b^2) - 2b\dtran\brOf{\frac12(1-s)c}
		\\&=:
		g(b,c,s)
		\eqfs
	\end{align*}
	By \cref{lmm:gbcs}, $\partial_b g(b,c,s)  + \partial_c g(b,c,s) \leq 0$.
	Set $u := c - b$. Define $h(b, u, s) := g(b,u+b,s)$. Then $\partial_b h(b,u,s) = \partial_b g(b,c,s)  + \partial_c g(b,c,s) \leq 0$. Thus, as $b \geq 0$, we have 
	\begin{equation*}
		h(b, u, s) \leq h(0, u, s) =  g(0,u,s) = \tran\brOf{\frac12(1+s)u} - \tran(u) - \tran\brOf{\frac12(1+s)u} + \trans(u^2) = 0
		\eqfs
	\end{equation*}
	Thus,
	\begin{equation*}
		f(a,b,c,s) \leq g(b,c,s) = h(b, u, s) \leq 0
		\eqfs
	\end{equation*}
\end{proof}
\begin{lemma}\label{lmm:gbcs}
	Let $\tran\in\setccod3$.
	Let $b,c\geq0$, $s\in[-1,1]$.
	Assume $b\leq sc$.
	Define
	\begin{equation*}
		g(b,c,s) := \tran\brOf{\frac12(1+s)c} - \tran(c) - \tran\brOf{\frac12(1+s)c - b} + \trans(c^2-2scb+b^2) - 2b\dtran\brOf{\frac12(1-s)c}\eqfs
	\end{equation*}
	Then
	\begin{equation*}
		\partial_b g(b,c,s)  + \partial_c g(b,c,s) \leq 0\eqfs
	\end{equation*}
\end{lemma}
\begin{proof}
	We have
	\begin{align*}
		\partial_b g(b,c,s) 
		&=
		\dtran\brOf{\frac12(1+s)c - b} - 2(sc-b)\dtrans(c^2-2scb+b^2) - 2\dtran\brOf{\frac12(1-s)c}\eqcm
		\\
		\partial_c g(b,c,s) 
		&=
		\frac12(1+s)\dtran\brOf{\frac12(1+s)c} - \dtran(c) - \frac12(1+s)\dtran\brOf{\frac12(1+s)c - b} + 
		\\&\hphantom{=}\ \ 2(c-sb)\dtrans(c^2-2scb+b^2) - (1-s)b\ddtran\brOf{\frac12(1-s)c}
		\eqfs
	\end{align*}
	Define 
	\begin{align*}
		f(b,c,s) &:= \frac{1-s}2\dtran( c-b)
		+ 2(1-s)(b+c)\dtrans( (c-b)^2)
		- 2\dtran\brOf{\frac{(1-s)c}2} 
		+ 
		\\&\hphantom{=}\ \ 
		-\frac{1-s}2\dtran( c)
		- b(1-s)\ddtran\brOf{\frac{c-b}2}
		\eqfs
	\end{align*}
	By \cref{mech:17}, 
	\begin{equation}
		\partial_b g(b,c,s) + \partial_c g(b,c,s) \leq f(b,c,s)\eqfs
	\end{equation}
	As $\dddtran\nega$, we have
	\begin{equation*}
		\partial_s^2 f(b,c,s) =
		- 2 (-c/2)^2\dddtran\brOf{\frac{(1-s)c}2}
		\geq 0		
		\eqcm
	\end{equation*}
	i.e., $f(b,c,s) $ is convex in $s$. Thus, to show $f(b,c,s) \leq 0$, we only need to check the extremes of $s$, $s_{\ms{min}} := \frac bc$ and $s_{\ms{max}} := 1$. For maximal $s$, we have
	\begin{equation*}
		f(b,c,s_{\ms{max}}) = - 2\dtran(0) \leq 0
		\eqfs
	\end{equation*}
	For minimal $s$, set $u := c-b\in[0,c]$. We define
	\begin{align*}
		h(u,c) &:=
		c f(c-u, c, s_{\ms{min}})
		\\&=
		\frac{4c-u}2\dtran(u)
		- 2c\dtran(u/2)
		- \frac u2 \dtran(c)
		- (c-u)u\ddtran(u/2)
		\eqfs
	\end{align*}
	By \cref{mech:15}, $\partial_c h(u,c) \leq 0$. Thus, as $c \geq u$, we have $h(u,c) \leq	h(u,u)$.
	By \cref{mech:16}, $h(u,u) \leq 0$.
	Thus, $c f(b, c, s_{\ms{min}}) = h(u,c) \leq h(u,u) \leq 0$.
	Because $ f(b,c,s)$ is convex in $s$ and negative for the extremes of $s$, we obtain
	\begin{equation*}
		\partial_b g(b,c,s)  + \partial_c g(b,c,s) \leq f(b,c,s)  \leq 0
		\eqfs
	\end{equation*}
\end{proof}
\subsection{Proof of \cref{lmm:param:firstreduction} \ref{lmm:param:firstreduction:a} for $c \geq a$, $a \leq b$}\label{ssec:reduicGaLb}
\begin{lemma}\label{lmm:param:reduicGaLb}
	Let $\tran\in\setccod3$.
	Let $a,b,c\geq0$, $s\in[-1,1]$.
	Assume $(s+1)c\leq 2a$, $c \geq a$, $a \leq b$.
	Then
	\begin{equation*}
		\tran(a)-\tran(c)-\tran(b-a)+\trans(c^2-2scb+b^2) 
		\leq
		2 b \dtran(a - sc) 
		\eqfs
	\end{equation*}
\end{lemma}
\begin{proof}
	Define
	\begin{equation*}
		f(a,b,c,s) := \tran(a)-\tran(c)-\tran(b-a)+\trans(c^2-2scb+b^2) - 2 b \dtran(a - sc)
		\eqfs
	\end{equation*}
	By \cref{mech:6}, $\partial_b f(a,b,c,s)\leq0$. Thus, as $b \geq a$, we have 
	\begin{align*}
		f(a,b,c,s) 
		&\leq 
		f(a,a,c,s) 
		\\&= 
		\tran(a)-\tran(c)+\trans(c^2-2sca+a^2) - 2 a \dtran(a - sc) 
		\\&=: 
		g(a,c,s)
		\eqfs
	\end{align*}
	By \cref{mech:7}, $\partial_a g(a,c,s) \leq 0$. Thus, as $a \geq a_{\ms{min}} := (s+1)c/2$, we have $g(a,c,s) \leq g(a_{\ms{min}},c,s)$.
	Set $u := (1+s)c/2$, $v := (1-s)c/2$. Then $c = u+v$, $s = (u-v)/(u+v)$, $sc = u-v$, $u-sc=v$, and
	\begin{align*}
		g(a_{\ms{min}},c,s) 
		&=
		g(u, u+v, (u-v)/(u+v)) 
		\\&=
		\tran(u)-\tran(u+v)+\trans(4uv + v^2) - 2 u \dtran(v) 
		\\&:=
		h(u, v)
		\eqfs
	\end{align*}
	By \cref{mech:8}, $\partial_u h(u, v) \leq 0$. Thus, as $u \geq 0$, we have 
	\begin{align*}
		h(u, v) \leq h(0, v) = \tran(0) - \tran(v) + \tran(v) = 0 
		\eqfs
	\end{align*}
	Thus, $f(a,b,c,s) \leq g(a,c,s) \leq h(u, v) \leq 0$.
\end{proof}
\subsection{Proof of \cref{lmm:param:firstreduction} \ref{lmm:param:firstreduction:a} for $a \geq c$, $b \geq 2sc$}\label{ssec:reduicLabGsc}
\begin{lemma}\label{lmm:param:reduicLabGsc}
    Let $\tran\in\setccod3$.
    Let $a,b,c\in\Rp$, $s\in[-1,1]$. Assume $b \geq 2sc$, $a\geq sc$.
    Then
    \begin{align*}
        \tran(a) - \tran(c) - \tran(\abs{a-b}) + \trans(c^2 - 2 s c b + b^2)
        \leq 
        2 b \dtran\brOf{\frac{a - s c}2}
        \eqfs
    \end{align*}
\end{lemma}
\begin{proof}
    The function $x \mapsto \trans(x+u) - \trans(x)$, $u\geq0$ is $\decr$ as $\dtrans\decr$.
    Thus, 
    \begin{align*}
        -\trans(c^2) + \trans(c^2 - 2 s c b+b^2)
        &\leq 
        -\trans((sc)^2) + \trans((sc)^2 - 2 s c b+b^2)
        \\&=
        -\tran|sc| + \tran|sc-b|
        \eqfs
    \end{align*}
    We use this and then apply \cref{lmm:merging:simple} to obtain,
    \begin{align*} 
        \tran(a) - \tran(c) - \tran|a-b| + \tran(c^2 - 2 s c b + b^2)
        &\leq 
        \tran(a) - \tran|sc| - \tran|a-b| + \tran|sc-b|
        \\&\leq 
        2 \br{\tran\brOf{\frac{a - s c + b}2} - \tran\abs{\frac{a - s c - b}2}}
        \eqfs
    \end{align*}	
    As $a - s c + b \geq \abs{a - s c - b}$, \cref{lmm:ccdiff} \ref{lmm:ccdiff:tightabs} yields
    \begin{align*}
        &\tran\brOf{\frac{a - s c + b}2} - \tran\abs{\frac{a - s c - b}2} 
        \\&\leq 
        \br{\frac{a - s c + b}2 - \frac{\abs{a - s c - b}}2} \dtran\brOf{\frac{a - s c + b}4 + \frac{\abs{a - s c - b}}4}
        \\&\leq 
        \min(a-sc, b) \dtran\brOf{\frac{\max(a-sc, b)}2}
        \eqfs
    \end{align*}
    By \cref{lmm:tranconcave} \ref{lmm:tranconcave:balance},
    \begin{equation}
        \min(a-sc, b) \dtran\brOf{\frac{\max(a-sc, b)}2}
        \leq
        b \dtran\brOf{\frac{a-sc}2}
        \eqfs
    \end{equation}
\end{proof}
\subsection{Proof of \cref{lmm:param:firstreduction} \ref{lmm:param:firstreduction:a} for $a \geq c$, $b \leq 2sc$, $sc \geq a-b$}\label{ssec:reduicLaLbscbLsc}
\begin{lemma}\label{lmm:param:reduicLaLbscbLsc}
	Let $\tran\in\setccod3$.
	Let $a,b,c\in\Rp$, $s\in[-1,1]$.
	Assume $\frac12 b \leq sc$, $sc \geq a-b$, and $a\geq c$.
	Then
	\begin{equation}\label{eq:reduicLaLbscbLsc:main}
		\tran(a) - \tran(c) - \tran(a-b) + \trans(c^2 - 2 s c b + b^2) 
		\leq 
		2 b \dtran\brOf{\frac{a-sc}{2}}
		\eqfs
	\end{equation}
\end{lemma}
\begin{proof}
	Define
	\begin{equation*}
		x_1^+ := a^2\eqcm\qquad
		x_2^+ := c^2 - 2 s c b + b^2\eqcm\qquad
		x_1^- := c^2\eqcm\qquad
		x_2^- := (a-b)^2
	\end{equation*} 
	to write the left-hand side of \eqref{eq:reduicLaLbscbLsc:main} as
	\begin{equation*}
		\trans(x_1^+) + \trans(x_2^+) - \trans(x_1^-) - \trans(x_2^-)
		\eqfs
	\end{equation*}
	As $a\geq c$,
	\begin{equation}\label{eq:reduicLaLbscbLsc:sqr}
		x_1^+ + x_2^+ -  x_1^- - x_2^-  = 2 b \br{a - s c} \geq 0
		\eqfs
	\end{equation}
	Because $a \geq c$ and $b \leq 2sc$, we have $a-b \geq a-2sc \geq a-2c \geq -c$. Together with $a-b \leq a$, we obtain $x_1^+ \geq \max(x_1^-, x_2^-)$.\\
	\underline{Case 1:  $x_2^+ \geq \min(x_1^-, x_2^-)$:}
	By first applying \cref{lmm:f1} and then using \eqref{eq:reduicLaLbscbLsc:sqr}, we obtain
	\begin{equation*}
		\trans(x_1^+) + \trans(x_2^+) - \trans(x_1^-) - \trans(x_2^-)
		\leq 
		2 \trans\brOf{b \br{a - s c}}
		\eqfs
	\end{equation*}
	\underline{Case 2:   $x_2^+ \leq \min(x_1^-, x_2^-)$:}
	By \eqref{eq:reduicLaLbscbLsc:sqr}, we have
	\begin{equation}
		x_1^+ + x_2^+ \geq  x_1^- + x_2^-
		\eqfs
	\end{equation}
	By first applying \cref{lmm:f2}, then using \eqref{eq:reduicLaLbscbLsc:sqr}, and finally $\trans\conc$, we obtain
	\begin{align*}
		&\trans(x_1^+) + \trans(x_2^+) - \trans(x_1^-) - \trans(x_2^-)
		\\&\leq 
		\trans\brOf{2 b \br{a - s c}}
		\\&\leq 
		2\trans\brOf{b \br{a - s c}}
		\eqfs
	\end{align*}
	\underline{Finally:}
	The condition $0 \leq a-sc \leq b$ together with \cref{lmm:cc:transdtran} implies 
	\begin{equation*}
		\trans\brOf{b \br{a - s c}}
		\leq 
		b \dtran\brOf{\frac{a-sc}{2}}
		\eqfs
		\qedhere
	\end{equation*}
\end{proof}
\subsection{Proof of \cref{lmm:param:firstreduction} \ref{lmm:param:firstreduction:a} for $a \geq c$, $b \leq 2sc$, $sc \leq a-b$}\label{ssec:reduicLaGbscbLsc}
\begin{lemma}\label{lmm:param:reduicLaGbscbLsc}
	Let $\tran\in\setccod3$.
	Let $a,b,c\in\Rp$, $s\in[-1,1]$.
	Assume $a\geq c$ and $\frac12 b \leq sc\leq a-b$.
	Then
	\begin{equation*}
		\tran(a) - \tran(c) - \tran(a-b) + \trans(c^2 - 2 s c b + b^2) \leq 2 b \dtran\brOf{\frac{a - s c}2}
		\eqfs
	\end{equation*}
\end{lemma}
\begin{proof}
	Set $x:=a-b$.
	Define
	\begin{align*}
		f(x,b,c,s) := 
		\tran(x+b)-\tran(c)-\tran(x)+\trans(c^2 - 2 s c b + b^2) -
		2 \br{
			\tran\brOf{\frac{x - s c}{2} + b}  -  \tran\brOf{\frac{x - s c}{2}}
		}
		\eqfs
	\end{align*}
	Then
	\begin{equation*}
		\partial_x f(x,b,c,s) =
		\dtran(x+b)-\dtran(x)
		- \dtran\brOf{\frac{x - s c}{2} + b}
		+ \dtran\brOf{\frac{x - s c}{2}}
		\eqfs
	\end{equation*}
	We have
	\begin{align*}
		x+b &\geq  x\eqcm&
		\frac{x - s c}{2} &\leq  \frac{x - s c}{2} + b\eqcm&
		(x+b) + \frac{x - s c}{2} &= x +\br{\frac{x - s c}{2} + b}\eqfs
	\end{align*}
	Thus, \cref{lmm:aux:extreme} implies
	\begin{equation*}
		\dtran(x+b)
		+ \dtran\brOf{\frac{x - s c}{2}}
		\leq
		\dtran(x)
		+\dtran\brOf{\frac{x - s c}{2} + b}
		\eqfs
	\end{equation*}
	Thus, $\partial_x f(x,b,c,s) \leq 0$. Thus, as $x \geq sc$, we have $f(x,b,c,s) \leq f(sc, b,c,s)$.
    In the case $x = x_0 := sc$, we have
    \begin{equation}
        f(sc, b,c,s) = g(x_0,b,c) :=  \tran(x_0+b)-\tran(c)-\tran(x_0)+\trans(c^2 - 2 x_0 b + b^2) - 2\tran\brOf{b}
        \eqcm
    \end{equation}
    with $c,b,x_0 \in\Rp$ and
    $c \leq x_0+b$,
    $b \leq 2x_0$,
    $x_0 \leq c$.    
    By \cref{lmm:xabc}, $g(x_0,b,c) \leq 0$. Thus, $f(x,b,c,s) \leq 0$ for all valid $x$.
	Then we obtain, using  \cref{lmm:ccdiff} \ref{lmm:ccdiff:tight} and $\dtran\incr$,
	\begin{equation*}
		\tran(a) - \tran(c) - \tran(a-b) + \trans(c^2 - 2 s c b + b^2) \leq 2 b \dtran\brOf{\frac{a - s c}2}
		\eqfs
	\end{equation*}
\end{proof}
\begin{lemma}\label{lmm:xabc}
	Let $\tran\in\setccod3$.
	Let $x,b,c\in\Rp$.
	Assume $b \leq 2x$, $x+b \geq c$, $x\leq c$.
	Then
	\begin{equation*}
		\tran(x+b) + \trans(c^2 - 2 x b + b^2) \leq \tran(c) + \tran(x) + 2 \tran(b) 
		\eqfs
	\end{equation*}	
\end{lemma}
\begin{proof}
	Define
	\begin{align*}
		f(x,b,c) &:= \tran(x+b) + \trans(c^2 - 2 x b + b^2) - \tran(c) - \tran(x) - 2 \tran(b) \eqcm
		\\
		g(x,b,c) := \partial_x f(x,b,c)
		&=
		\dtran(x+b) - \dtran(x) - 2b\dtrans(c^2 - 2 x b + b^2)
		\eqfs
	\end{align*}
	By \cref{mech:14}, $\partial_x g(x,b,c) \leq 0$. 
	Thus, as $x \geq x_{\ms{min}} := \max(\frac b2, c-b)$, we have $g(x,b,c) \leq g(x_{\ms{min}},b,c)$. 
	\\
	\underline{Case 1, $x_{\ms{min}} = c-b$:}\\
	In this case $c-b \geq b/2$, i.e., $2c\geq3b$. By \cref{mech:10}, $g(c-b, b, c) \leq 0$. 
	Thus, as $x\geq x_{\ms{min}} = c-b$, we have $f(x,b,c) \leq f(c-b, b, c)$. By \cref{mech:11}, $f(c-b, b, c) \leq 0$.
	Thus, $f(x,b,c) \leq 0$.
	\\
	\underline{Case 2, $x_{\ms{min}} = \frac b2$:}\\
	In this case $c-b\leq b/2$, i.e., $2c \leq 3b$. By \cref{mech:12}, $g(b/2, b, c) \leq 0$. 
	Thus, as $x\geq x_{\ms{min}} = b/2$, we have $f(x,b,c) \leq f(b/2, b, c)$. By \cref{mech:13}, $f(b/2, b, c) \leq 0$.
	Thus, $f(x,b,c) \leq 0$.
\end{proof}

	\section{Auxiliary Results}\label{app:aux}
To make the proofs presented in appendix \ref{app:proofsymm} and \ref{app:mainproof} more readable, we have extracted some calculations to this section.
\subsection{Nondecreasing, Convex Functions with Concave Derivative}\label{app:aux:cc}
\begin{proof}[of \cref{lmm:ccSimpleProps}]
	\begin{enumerate}[label=(\roman*)]
		\item 
		As $\tran$ and $\dtran$ are convex and concave, respectively, we immediately obtain continuity on $\Rpp$. As $\tran$ is nondecreasing and convex, it must also be continuous at $0$.
		As $\tran$ is nondecreasing, $\dtran$ is nonnegative. As $\tran$ is convex, $\dtran$ is nondecreasing.
		Furthermore, $\lim_{x\searrow0} \dtran(x)$ exists as $\dtran$ is nonnegative and nondecreasing.
		Hence, by definition of $\dtran(0)$, $\dtran$ is continuous.
		As $\tran$ is convex and nondecreasing on $\Rp$,
		\begin{equation}
			h\mapsto \frac{\tran(h)-\tran(0)}{h}
		\end{equation}
		is nondecreasing and nonnegative. Thus, $\partial_+ \tran(0)$ exists. 
		It remains to show that $\partial_+ \tran(0) = \dtran(0)$.
		
		Fix $\epsilon >0$. As $\dtran$ is continuous, it is uniformly continuous on the compact interval $[0,2]$. Thus, we can find $\delta\in(0,1]$ such that $\abs{\dtran(x)-\dtran(y)} \leq \epsilon$ for all $x,y\in[0,2]$ with $0 < \abs{x-y} \leq \delta$. Thus, for any $h\in (0, \delta]$ and $x \in (0, 1]$,
		\begin{equation}
			\abs{\frac{\tran(x+h)-\tran(x)}{h} - \dtran(x)}
			= 
			\abs{\frac1h \int_0^h \dtran(x+z) - \dtran(x) \dl z}
			\leq
			\epsilon
			\eqfs
		\end{equation}
		Choose $h\in (0, \delta]$ small enough such that,
		\begin{equation}
			\abs{\frac{\tran(h)-\tran(0)}{h} - \partial_+\tran(0)}
			\leq
			\epsilon
			\eqfs
		\end{equation}
		As $\tran$ is continuous and $h$ is fixed, we can find $x\in(0,1]$ small enough such that
		\begin{align*}
			\abs{\tran(x + h) - \tran(h)} &\leq h\epsilon\eqcm\\
			\abs{\tran(x) - \tran(0)} &\leq h\epsilon\eqfs
		\end{align*}
		Thus, we have
		\begin{align*}
			\abs{\dtran(0) - \partial_+\tran(0)}
			&\leq
			\abs{\dtran(0) - \dtran(x)} +
			\abs{\dtran(x) - \partial_+\tran(0)}
			\\&\leq
            \epsilon +
			\abs{\dtran(x) - \frac{\tran(x+h)-\tran(x)}{h}} +
			\abs{\frac{\tran(x+h)-\tran(x)}{h} - \partial_+\tran(0)}
			\\&\leq
            2\epsilon +
			\abs{\frac{\tran(x+h)-\tran(x)}{h} - \frac{\tran(h)-\tran(0)}{h}}+
			\abs{\frac{\tran(h)-\tran(0)}{h} - \partial_+\tran(0)}
			\\&\leq
            3\epsilon +
			\frac1h\br{\abs{\tran(x+h)-\tran(h)}+\abs{\tran(x)-\tran(0)}}
			\\&\leq 5\epsilon
			\eqfs
		\end{align*}
		As $\epsilon > 0$ can be chosen arbitrarily small, $\dtran(0) = \partial_+\tran(0)$.
		\item As $\tran$ is convex, $\ddtran$ is nonnegative. As $\dtran$ is concave, $\ddtran$ is nonincreasing. 
		\item As $\dtran$ is concave, $\dddtran$ is nonpositive.
	\end{enumerate}
\end{proof}
\begin{proof}[of \cref{lmm:ccpoly}]
	\begin{enumerate}[label=(\roman*)]
		\item Apply a second and a third order Taylor expansion to $y\mapsto \tran(x + y)$ at $0$ and use \cref{lmm:ccSimpleProps}.
		\item Apply a first and a second order Taylor expansion to $y\mapsto \dtran(x + y)$ at $0$ and use \cref{lmm:ccSimpleProps}.
	\end{enumerate}
\end{proof}
\begin{proof}[of \cref{lmm:ccdiff}]
	\begin{enumerate}[label=(\roman*)]
		\item[\ref{lmm:ccdiff:tight}]
		For the lower bound, as $\dtran$ is concave,
		\begin{align*}
			\tran(x) - \tran(y) 
			&= 
			\int_y^x \dtran(u) \dl u
			\\&\geq
			(x-y)\int_0^1 (1-t)\dtran(y) + t \dtran(x)\, \dl t
			\\&=
			\frac {x-y}2 \br{\dtran(x)+\dtran(y)}
			\eqfs
		\end{align*}
		For the upper bound, concavity of $\dtran$ implies the existence of an affine linear function $h$ with $h(u) \geq \dtran(u)$ for all $u \in\Rp$ and 
		\begin{equation}
			h\brOf{\frac{x+y}{2}} = \dtran\brOf{\frac{x+y}{2}}
			\eqfs
		\end{equation}
		Thus,
		\begin{align*}
			\tran(x) - \tran(y) 
			&\leq
			\int_y^x h(u) \dl u
			\\&=
			\frac {x-y}2\br{h(x)+h(y)}
			\\&=
			(x-y) h\brOf{\frac{x+y}{2}}
			\eqfs
		\end{align*}
		\item[\ref{lmm:ccdiff:tightabs}]
		Follows directly from the upper bound in \ref{lmm:ccdiff:tight}.
	\end{enumerate}
\end{proof}
\begin{lemma}[Concave derivative]\label{lmm:tranconcave}
	Let $\tran\in\setcc$.
	\begin{enumerate}[label=(\roman*)]
		\item \label{lmm:tranconcave:add}
		Let $x,y\in\Rp$. Then
		\begin{equation*}
			\dtran(x + y) \leq \dtran(x) +  \dtran(y) \leq 2 \dtran\brOf{\frac{x+y}2}
			\eqfs
		\end{equation*}
		\item\label{lmm:tranconcave:factor} Let $a,x\in\Rp$. Then
		\begin{align*}
			\dtran(ax) &\geq a\dtran(x) \text{ for } a\leq 1\eqcm\\
			\dtran(ax) &\leq a\dtran(x) \text{ for } a\geq 1
			\eqfs
		\end{align*}
		\item\label{lmm:tranconcave:balance} Let $x,y\in\Rp$. Assume $y\geq x$. Then
		\begin{equation*}
			x \dtran(y) \leq y \dtran(x)
			\eqfs
		\end{equation*}
	\end{enumerate}
\end{lemma}
\begin{proof}
	These are all well-known properties of nonnegative, concave functions.
	\begin{enumerate}[label=(\roman*)]
		\item Use \cref{lmm:aux:redistri} and Jensen's inequality.
		\item Use $(1-t)\dtran(x_0) + t \dtran(x_1) \leq \dtran((1-t) x_0 + t x_1)$ on points $x_0=0$, $x_1=x$, $t = a$ and on $x_0=0$, $x_1=ax$, $t = 1/a$, respectively, and note that $\dtran(0)\geq0$.
		\item Apply \ref{lmm:tranconcave:factor} with $a = y/x$.
	\end{enumerate}
\end{proof}
\begin{lemma}[Square root and derivative]\label{lmm:cc:transdtran}
	Let $\tran\in\setccod2$.
	Let $x,y\in\Rp$. Assume $x \geq y$. Then
	\begin{equation*}
		\tran\brOf{\sqrt{xy}} \leq x \dtran\brOf{\frac12 y}\eqfs
	\end{equation*} 
\end{lemma}
\begin{proof}
	Define $f(x, y) := x \dtran(\frac12 y) - \tran(\sqrt{xy})$. Its partial derivative with respect to $x$ is
	\begin{align*}
		\partial_x f(x, y)
		&= 
		\dtran\brOf{\frac12 y} - \frac{\sqrt{y}}{2\sqrt{x}}\dtran\brOf{\sqrt{xy}}
		\\&\geq  
		\dtran\brOf{\frac12 y} - \dtran\brOf{\frac12 y} = 0
		\eqcm
	\end{align*}
	by \cref{lmm:tranconcave} \ref{lmm:tranconcave:factor} with $\frac{\sqrt{y}}{2\sqrt{x}} \in [0,1]$. Thus, 
	\begin{align*}
		f(x, y) 
		&\geq 
		f(y, y) 
		\\&= 
		y \dtran\brOf{\frac12 y} - \tran(y)
		\\&\geq 
		0\eqcm
	\end{align*}
	where the last inequality is due to \cref{lmm:ccdiff} \ref{lmm:ccdiff:tight} with $\tran(0) = 0$.
\end{proof}
\subsection{Merging Terms}
\begin{lemma}\label{lmm:aux:redistri}\mbox{ }
	\begin{enumerate}[label=(\roman*)]
		\item Let $f\colon\Rp \to \R$. Assume $f$ is concave. Let $a,b\in\Rp$ with $a\geq b$. Then
		$x \mapsto f(a+x) + f(b-x)$ is nonincreasing. If additionally $f(0)\geq 0$, then $f$ is subadditive.
		\item 
		Let $f\colon\Rp \to \R$. Assume $f$ is convex. Let $a,b\in\Rp$ with $a\geq b$. Then
		$x \mapsto f(a+x) + f(b-x)$ is nondecreasing.
	\end{enumerate}
\end{lemma}
\begin{proof}
	We prove the first part; the second part is analogous.
	As $f$ is concave, we have
	\begin{align*}
		f(a) &\geq \frac{a-b+x}{a-b+2x} f(a+x) + \frac{x}{a-b+2x} f(b-x)\eqcm\\
		f(b) &\geq \frac{x}{a-b+2x} f(a+x) + \frac{a-b+x}{a-b+2x} f(b-x)
	\end{align*}
	for $x\in[0, b]$. Adding the two inequalities yields
	\begin{equation}
		f(a) + f(b) \geq f(a+x)  + f(b-x)
		\eqfs
	\end{equation}
	As this inequality also applies to $\tilde a = a+x$, $\tilde b = b-x$, we have that $x \mapsto f(a+x) + f(b-x)$ is nonincreasing. 
	Subadditivity follows by setting $x = b$.
\end{proof}
\begin{lemma}\label{lmm:aux:six}
	Let $f\colon\Rp \to \R$. Assume $f(0)\geq 0$, $f$ is nondecreasing, and $f$ is concave. Let $x_1, \dots, x_6 \in \Rp$.
	Assume $\max(x_1, x_2, x_3, x_4) \leq \max(x_5, x_6)$ and $x_1+x_2+x_3+x_4 \geq x_5+x_6$. Then
	\begin{equation}\label{eq:aux:six}
		f(x_1) + f(x_2) + f(x_3) + f(x_4) \geq f(x_5) + f(x_6)
		\eqfs
	\end{equation}
\end{lemma}
\begin{proof}
	Without loss of generality assume $x_1 \geq x_2 \geq x_3 \geq x_4$ and $x_5\geq x_6$.
	
	First consider the case $x_1 \geq x_6$. We decrease $x_5$ and increase $x_6$ while holding $x_5+x_6$ constant until one $x_\bullet$ on the right-hand side coincides with one $x_\bullet$ one the left-hand side. By \cref{lmm:aux:redistri}, this can only increase the right-hand side of \eqref{eq:aux:six}.
	If $\{x_1, x_2, x_3, x_4\} \cap \{x_5, x_6\} \neq \emptyset$, we can subtract the term with the value in the intersection from \eqref{eq:aux:six}.
	The inequality of the form $f(x_1) + f(x_2) + f(x_3) \geq f(x_1 +x_2+x_3) \geq f(x_5)$ for $x_5 \leq x_1+x_2+x_3$ is obtained using subadditivity of $f$, see \cref{lmm:aux:redistri}, and the assumption that $f$ is nondecreasing.
	
	Now consider the case $x_1 < x_6$. Set $s := (x_5+x_6)/2$. Using \cref{lmm:aux:redistri}, we obtain $f(x_5) + f(x_6) \leq 2f(s)$. Furthermore $x_1 \leq s$ and $x_1 + x_2 + x_3 + x_4 \geq 2s$. Thus, again using  \cref{lmm:aux:redistri} and the assumption that $f$ is nondecreasing, we can increase $x_1$ and $x_2$ while decreasing $x_3$ and $x_4$ to $0$ to get
	\begin{equation}
		f(x_1) + f(x_2) + f(x_3) + f(x_4) \geq 2f(s) + 2f(0)
		\eqfs
	\end{equation}
	As $f(0)\geq 0$, we arrive at the desired result.
\end{proof}
\begin{lemma}\label{lmm:aux:extreme}
	Let $\tran\in\setcc$.
	Let $a,b,c,d\in\Rp$.
	Assume $a \geq b \geq c \geq d$ and $a+d \leq b+c$. Then $\dtran(a) + \dtran(d) \leq \dtran(b) + \dtran(c)$.
\end{lemma}
\begin{proof}
	As $\dtran$ is concave, \cref{lmm:aux:redistri} applies.
\end{proof}
\begin{lemma}\label{lmm:f1}
	Let $\tran\in\setccod{3}$.
	Let $a\geq b\geq 0 $, $d \geq c \geq 0$.
	Then
	\begin{equation*}
		\trans(a) - \trans(b) - \trans(c) + \trans(d) \leq 
		2 \trans\brOf{\frac12 (a-b+d-c)}
		\eqfs
	\end{equation*}
\end{lemma}
\begin{proof}
	As $a\geq b$, $d\geq c$, subadditivity of $\trans$
	\begin{equation*}
		\trans(a) - \trans(b) - \trans(c) + \trans(d) \leq \trans(a-b) + \trans(d-c) \eqfs
	\end{equation*}
	Furthermore, by concavity of $\trans$,
	\begin{equation*}
		\frac12 \trans(a-b) + \frac12 \trans(d-c)
		\leq
		\trans\brOf{\frac12 (a-b+d-c)}
		\eqfs
		\qedhere
	\end{equation*}
\end{proof}
\begin{lemma}\label{lmm:f2}
	Let $\tran\in\setccod{3}$. Let $a\geq b\geq c \geq d\geq 0$, $a+d\geq b+c$.
	Then
	\begin{equation*}
		\trans(a) - \trans(b) - \trans(c) + \trans(d) \leq \trans(a-b-c+d)
		\eqfs
	\end{equation*}
\end{lemma}
\begin{proof}
	Define $f(x,y) = \trans(x) + \trans(y) - \trans(x + y)$ for $x,y\geq0$.
	Then $\partial_xf(x,y) = \dtrans(x) - \dtrans(x+y) \geq 0$ and similarly $\partial_yf(x,y) \geq 0$.
	Set $\delta := a-b$ and $\epsilon := c-d$. The assumptions ensure $\delta\geq\epsilon\geq0$.
	Then, 
	\begin{equation*}
		f(b,\delta) \geq f(b,\epsilon) \geq f(d,\epsilon)
		\eqfs
	\end{equation*}
	Thus,
	\begin{align*}
		0 
		&\geq 
		f(d,\epsilon) - f(b,\delta) 
		\\&= 
		\trans(d) + \trans(\epsilon) - \trans(d+\epsilon)
		- \trans(b) - \trans(\delta) + \trans(b+\delta)
		\\&=
		\trans(d) + \trans(\epsilon) - \trans(c)
		- \trans(b) - \trans(\delta) + \trans(a)
		\eqfs
	\end{align*}
	With this we get
	\begin{align*}
		\trans(d) - \trans(c) - \trans(b) + \trans(a)
		&\leq
		\trans(\delta) - \trans(\epsilon)
		\\&\leq 
		\trans(\delta-\epsilon)
		\\&=
		\trans(a-b-c+d)
		\eqfs
		\qedhere
	\end{align*}
\end{proof}%
\begin{lemma}[Simple Merging Lemma]\label{lmm:merging:simple}
	Let $\tran\in\setccod{3}$.
	Let $b\geq0$, $a,c\in\R$.
	Then
	\begin{equation*}
		\tran(|a|) - \tran(|c|) - \tran(|a-b|) + \tran(|c-b|) 
		\leq  
		2 \br{ \tran\brOf{\frac{a - c + b}2} - \tran\brOf{\frac{|a-c - b|}2} } \indOf{a>c}\eqfs
	\end{equation*}
\end{lemma}
\begin{proof}
	Define
	$f(x) := \tran(\abs{x}) - \tran(\abs{x - y})$. Then
	$f\pr(x) = \sgn(x)\dtran(\abs{x}) - \sgn(x)\dtran(\abs{x - y})$.
	If $y \geq 0$ then: if $x \geq 0$ then $\abs{x} \geq \abs{x - y}$, if $x \leq 0$ then $\abs{x} \leq \abs{x - y}$.
	Thus, $f\pr(x) \geq 0$, as $\dtran$ is increasing. Hence, $f(x)$ is increasing.
	Thus, if $a\leq c$, then
	\begin{equation*}
		\tran(|a|)-\tran(|a-b|) \leq \tran(|c|) - \tran(|c-b|) 
		\eqfs
	\end{equation*}
	This shows the inequality for the case $a\leq c$.
	
	Now assume $a>c$.
	Set $q := a-b$ and define
	\begin{equation*}
		g(b) :=  \tran(|q+b|) - \tran(|c|) - \tran(|q|) + \tran(|c-b|) - 2\br{\tran\brOf{\frac{q-c}{2}+b}-\tran\brOf{\frac{q-c}{2}}}
		\eqfs
	\end{equation*}
	We have
	\begin{equation*}
		g\pr(b) = \sgn(q+b)\dtran(|q+b|) - \sgn(c-b)\dtran(|c-b|) - 2\dtran\brOf{\frac{q-c}{2}+b}
		\eqfs
	\end{equation*}
	\underline{Case 1: $\sgn(q+b) = +1$, $\sgn(c-b)=+1$}:
	\begin{equation*}
		g\pr(b) = \dtran(q+b) - \dtran(c-b) - 2\dtran\brOf{\frac{q-c}{2}+b}
		\eqcm
	\end{equation*}
	As $\dtran$ is concave, it is subadditive and $\dtran(2x)\leq2\dtran(x)$. Furthermore, $q+b  = a > c \geq  c-b$. Thus, 
	\begin{equation*}
		\dtran(q+b) - \dtran(c-b) \leq \dtran(q+b-c+b) \leq 2\dtran\brOf{\frac{q-c}{2}+b}
		\eqfs
	\end{equation*}
	\underline{Case 2: $\sgn(q+b) = -1$, $\sgn(c-b)=-1$}:
	\begin{equation*}
		g\pr(b) = -\dtran(-q-b) + \dtran(b - c) - 2\dtran\brOf{\frac{q-c}{2}+b}
		\eqcm
	\end{equation*}
	Similarly to the first case, we have $b - c \geq -c > - a = - q - b$ and
	\begin{equation*}
		\dtran(b - c)-\dtran(-q-b) \leq  \dtran(b - c + q + b) \leq 2\dtran\brOf{\frac{q-c}{2}+b}
		\eqfs
	\end{equation*}
	\underline{Case 3: $\sgn(q+b) = +1$, $\sgn(c-b)=-1$}:
	\begin{equation*}
		\dtran(q+b) + \dtran(b-c) - 2\dtran\brOf{\frac{q-c}{2}+b}
		\eqcm
	\end{equation*}
	$\dtran$ is concave, thus
	\begin{equation*}
		\frac 12 \dtran(q+b) + \frac 12 \dtran(b-c) \leq \dtran\brOf{\frac{q-c}{2}+b}
		\eqfs
	\end{equation*}
	\underline{Case 4: $\sgn(q+b) = -1$, $\sgn(c-b)=+1$}:
	\begin{equation*}
		-\dtran(-q-b) - \dtran(c-b) - 2\dtran\brOf{\frac{q-c}{2}+b}
		\eqcm
	\end{equation*}
	\begin{equation*}
		-\dtran(-q-b)-\dtran(c-b) \leq 0
		\eqfs
	\end{equation*}
	\underline{Together:}
	In every case, we have $g\pr(b) \leq 0$ and $g(0)=0$. Thus, 
	\begin{equation*}
		g(b) \leq 0
		\eqfs
		\qedhere
	\end{equation*}
\end{proof}
\subsection{Mechanical Proofs}\label{sec:mecha}
The following auxiliary results consist of simple term transformations. Their proofs are not commented further.
\begin{lemma}\label{mech:1}
	Let $\tran\in\setccod3$.
	Let $a,b,c\in\Rp$. Assume $a\geq c$, $a-b-2c \geq 0$. Then
	\begin{equation}
		2(a-b-2c)\dtrans((a-b)^2+4cb) + \dtran(c+b) - 2 \dtran(a - c) \leq 0
		\eqfs
	\end{equation}
\end{lemma}
\begin{proof}
	\begin{align*}
		&  
		2(a-b-2c)\dtrans((a-b)^2+4cb) + \dtran(c+b) - 2 \dtran(a - c) 
		\\
		a-c \geq c+b, \dtran\incr
		\qquad&\leq
		2(a-b-2c)\dtrans((a-b)^2+4cb) - \dtran(a - c) 
		\\
		\qquad&=
		2(a-b)\dtrans((a-b)^2+4cb) -4c\dtrans((a-b)^2+4cb) - \dtran(a - c) 
		\\
		a\geq b,\dtrans\decr
		\qquad&\leq
		2(a-b)\dtrans((a-b)^2) -4c\dtrans((a-b)^2+4cb) - \dtran(a - c) 
		\\
		\qquad&=
		\dtran(a-b) - 4c\dtrans(a^2 - b((a-2c)+(a-b-2c))) - \dtran(a - c) 
		\\
		a-2c\geq b\geq0,\dtrans\decr
		\qquad&\leq
		\dtran(a-b) - 4c\dtrans(a^2) - \dtran(a - c) 
		\\
		\qquad&=
		\dtran(a-b) - \frac{2c}a \dtran(a) - \dtran(a - c) 
		\\
		\dtran\incr
		\qquad&\leq
		\dtran(a) - \frac{2c}a \dtran(a) - \dtran(a - c) 
		\\
		\qquad&=
		\br{1 - \frac{2c}a} \dtran(a) - \dtran(a - c) 
		\\
		a\geq2c, \dtran\conc
		\qquad&\leq
		\dtran\brOf{\br{1 - \frac{2c}a}a} - \dtran(a - c) 
		\\
		\qquad&=
		\dtran\brOf{a-2c} - \dtran(a - c) 
		\\
		\dtran\incr
		\qquad&\leq
		0
		\eqfs
	\end{align*}
\end{proof}
\begin{lemma}\label{mech:2}
	Let $\tran\in\setccod3$.
	Let $a,b,c\in\Rp$. Assume $a\geq c$, $a-b-2c \leq 0$. Then
	\begin{equation}
		-2(b+2c-a)\dtrans((a-b)^2+4cb) + \dtran(c+b) - 2 \dtran(a - c) \leq 0
		\eqfs
	\end{equation}
\end{lemma}
\begin{proof}
	\begin{align*}
		&
		-2(b+2c-a)\dtrans((a-b)^2+4cb) + \dtran(c+b) - 2 \dtran(a - c)
		\\
		\qquad&=  
		-2(b+2c-a)\dtrans(a^2 + b(-2a + b + 4c)) + \dtran(c+b) - 2 \dtran(a - c)
		\\
		a \geq c, b+2c\geq a, \dtrans\decr
		\qquad&\leq
		-2(b+2c-a)\dtrans(a^2 + b(-2a + b + 4a)) + \dtran(c+b) - 2 \dtran(a - c)
		\\
		\qquad&=
		-2(b+2c-a)\dtrans((a+b)^2) + \dtran(c+b) - 2 \dtran(a - c)
		\\
		\qquad&=
		-\frac{b+2c-a}{a+b}\dtran(a+b) + \dtran(c+b) - 2 \dtran(a - c)
		\\
		a \geq c, \dtran\incr
		\qquad&\leq
		-\frac{b+2c-a}{a+b}\dtran(a+b) + \dtran(a+b) - 2 \dtran(a - c)
		\\
		\qquad&=
		\frac{2(a-c)}{a+b}\dtran(a+b) - 2 \dtran(a - c)
		\\
		0 \leq a-c \leq a+b, \dtran\conc
		\qquad&\leq
		2\dtran(a-c) - 2 \dtran(a - c)
		\\
		\qquad&=
		0
		\eqfs
	\end{align*}
\end{proof}
\begin{lemma}\label{mech:3}
	Let $\tran\in\setccod3$.
	Let $a,b,c\geq0$, $s\in[-1,1]$.
	Assume $a \geq b \geq sc$.
	Then
	\begin{equation*}
		\dtran(a-b) + 2(b-sc)\dtrans(c^2-2scb+b^2) - 2 \dtran(a - sc) \leq 0
		\eqfs
	\end{equation*}
\end{lemma}
\begin{proof}
	\begin{align*}
		&
		\dtran(a-b) + 2(b-sc)\dtrans(c^2-2scb+b^2) - 2 \dtran(a - sc)
		\\
		c\geq sc, \dtrans\decr
		\qquad&\leq
		\dtran(a-b) + 2(b-sc)\dtrans(s^2c^2-2scb+b^2) - 2 \dtran(a - sc)
		\\
		\qquad&=
		\dtran(a-b) + \dtran(b-sc) - 2 \dtran(a - sc)
		\\
		a \geq b \geq sc, \dtran\incr
		\qquad&\leq
		\dtran(a-sc) + \dtran(a-sc) - 2 \dtran(a - sc)
		\\
		\qquad&=
		0
		\eqfs
	\end{align*}
\end{proof}
\begin{lemma}\label{mech:5}
	Let $\tran\in\setccod3$.
	Let $u,v\in\Rp$.
	Assume $u \geq v$.
	Then
	\begin{equation*}
		\dtran(u)-\dtran(u+v)+4v\dtrans(4uv) - 2\dtran(v) \leq 0
		\eqfs
	\end{equation*}
\end{lemma}
\begin{proof}
	\begin{align*}
		&
		\dtran(u)-\dtran(u+v)+4v\dtrans(4uv) - 2\dtran(v)
		\\
		\dtran\incr
		\qquad&\leq
		4v\dtrans(4uv) - 2\dtran(v)
		\\
		u\geq v, \dtrans\decr
		\qquad&\leq
		4v\dtrans(4vv) - 2\dtran(v)
		\\
		\qquad&=
		\dtran(2v) - 2\dtran(v)
		\\
		\dtran\conc
		\qquad&\leq
		0
		\eqfs
	\end{align*} 
\end{proof}
\begin{lemma}\label{mech:6}
	Let $\tran\in\setccod3$.
	Let $a,b,c\geq0$, $s\in[-1,1]$.
	Assume $(s+1)c\leq 2a$, $a \leq b$.
	Then
	\begin{equation*}
		-\dtran(b-a)+2(b-sc)\dtrans(c^2-2scb+b^2) - 2\dtran(a - sc)
		\leq
		0
		\eqfs
	\end{equation*}
\end{lemma}
\begin{proof}
	\begin{align*}
		&
		-\dtran(b-a)+2(b-sc)\dtrans(c^2-2scb+b^2) - 2\dtran(a - sc)
		\\
		\dtran\conc \text{(subadditive)}
		\qquad&\leq
		-\dtran(b-a + 2(a - sc))+2(b-sc)\dtrans(c^2-2scb+b^2)
		\\
		\qquad&=
		-\dtran(b+a-2sc) + 2(b-sc)\dtrans(c^2-2scb+b^2)
		\\
		a \geq \frac{1+s}{2}c,\dtran\incr
		\qquad&\leq
		-\dtran(b+\frac{1+s}{2}c-2sc) + 2(b-sc)\dtrans(c^2-2scb+b^2)
		\\
		\qquad&=
		-\dtran(b+\frac{1-s}{2}c-sc) + 2(b-sc)\dtrans(c^2-2scb+b^2)
		\\
		1-s \geq 0,\dtran\incr
		\qquad&\leq
		-\dtran(b-sc) + 2(b-sc)\dtrans(c^2-2scb+b^2)
		\\
		sc \leq c,\dtrans\decr
		\qquad&\leq
		-\dtran(b-sc) + 2(b-sc)\dtrans((sc)^2-2scb+b^2)
		\\
		\qquad&=
		-\dtran(b-sc) + \dtran(b-sc)
		\\
		\qquad&=
		0
		\eqfs
	\end{align*}
\end{proof}
\begin{lemma}\label{mech:7}
	Let $\tran\in\setccod3$.
	Let $a,c\in\Rp$, $s\in[-1,1]$.
	Assume $a \geq sc$.
	Then
	\begin{equation*}
		\dtran(a) + 2(a-sc)\dtrans(c^2-2sca+a^2) - 2 \dtran(a - sc) - 2 a \ddtran(a - sc) 
		\leq
		0
		\eqfs
	\end{equation*}
\end{lemma}
\begin{proof}
	\begin{align*}
		&
		\dtran(a) + 2(a-sc)\dtrans(c^2-2sca+a^2) - 2 \dtran(a - sc) - 2 a \ddtran(a - sc) 
		\\
		a \geq sc, sc \leq c, \dtrans\decr
		\qquad&\leq
		\dtran(a) + 2(a-sc)\dtrans((sc)^2-2sca+a^2) - 2 \dtran(a - sc) - 2 a \ddtran(a - sc) 
		\\
		\qquad&=
		\dtran(a) + \dtran(a - sc) - 2 \dtran(a - sc) - 2 a \ddtran(a - sc) 
		\\
		\qquad&=
		\dtran(a) - \dtran(a - sc) - 2 a \ddtran(a - sc) 
		\\
		sc \leq 2 a, \ddtran\posi
		\qquad&\leq
		\dtran(a) - \dtran(a - sc) - sc \ddtran(a - sc) 
		\\
		\text{\cref{lmm:ccpoly} \ref{lmm:ccpoly:dtaylor}}
		\qquad&\leq
		0
		\eqfs
	\end{align*}
\end{proof}
\begin{lemma}\label{mech:8}
	Let $\tran\in\setccod3$.
	Let $u,v\in\Rp$.
	Then
	\begin{equation*}
		\dtran(u) - \dtran(u+v) + 4v\dtrans(4uv + v^2) - 2\dtran(v)
		\leq
		0
		\eqfs
	\end{equation*}
\end{lemma}
\begin{proof}
	\begin{align*}
		&\dtran(u) - \dtran(u+v) + 4v\dtrans(4uv + v^2) - 2\dtran(v) 
		\\
		\qquad&=
		\dtran(u) - \dtran(u+v) + \frac{1}{\sqrt{u/v + 1/4}}\dtran(2v\sqrt{u/v + 1/4}) - 2\dtran(v) 
		\\
		2\sqrt{u/v + 1/4}\geq 1, \dtran\conc
		\qquad&\leq
		\dtran(u) - \dtran(u+v) + \frac{2\sqrt{u/v + 1/4}}{\sqrt{u/v + 1/4}}\dtran(v) - 2\dtran(v) 
		\\
		\qquad&=
		\dtran(u)-\dtran(u+v)
		\\
		v\geq0, \dtran\incr
		\qquad&\leq
		0 
		\eqfs
	\end{align*}
\end{proof}
\begin{lemma}\label{mech:9}
	Let $\tran\in\setccod3$.
	Let $a,b,c\in\Rp$ and $s\in[-1,1]$.
	Assume $a \geq sc \geq b$.
	Then
	\begin{equation*}
		\dtran(a) - \dtran(a-b) - 2b \ddtran(a - sc)
		\leq
		0
		\eqfs
	\end{equation*}
\end{lemma}
\begin{proof}
	\begin{align*}
		&  
		\dtran(a) - \dtran(a-b) - 2b \ddtran(a - sc)
		\\
		b\leq sc, \ddtran\decr
		\qquad&\leq
		\dtran(a) - \dtran(a-b) - 2b \ddtran(a - b)
		\\
		\text{\cref{lmm:ccpoly} \ref{lmm:ccpoly:dtaylor}}
		\qquad&\leq
		b \ddtran(a - b) - 2b \ddtran(a - b)
		\\
		\qquad&=
		-b \ddtran(a - b) 
		\\
		\ddtran\posi
		\qquad&\leq
		0
		\eqfs
	\end{align*}
\end{proof}
\begin{lemma}\label{mech:10}
	Let $\tran\in\setccod3$.
	Let $b,c\in\Rp$.
	Assume $2c\geq3b$.
	Then
	\begin{equation*}
		\dtran(c) - \dtran(c-b) - 2b \dtrans(c^2 - 2 (c-b) b + b^2) 
		\leq
		0
		\eqfs
	\end{equation*}
\end{lemma}
\begin{proof}
	\begin{align*}
		&\dtran(c) - \dtran(c-b) - 2b \dtrans(c^2 - 2 (c-b) b + b^2)
		\\
		\qquad&=
		\dtran(c)- \dtran(c-b) - 2b \dtrans(c^2 - b (2c-3b))
		\\
		2c\geq3b, \dtrans\decr
		\qquad&\leq
		\dtran(c)- \dtran(c-b) - 2b \dtrans(c^2)
		\\
		\qquad&=
		\dtran(c)- \dtran(c-b) - \frac bc \dtran(c)
		\\
		\qquad&=
		\frac{c-b}{c}\dtran(c) - \dtran(c-b)
		\\
		\frac{c-b}{c}\in[0,1], \dtran\conc
		\qquad&\leq
		\dtran(c-b) - \dtran(c-b)
		\\
		\qquad&=
		0
		\eqfs
	\end{align*}
\end{proof}
\begin{lemma}\label{mech:11}
	Let $\tran\in\setccod3$.
	Let $b,c\in\Rp$.
	Assume $c\geq b$.
	Then
	\begin{equation*}
		\trans((c-b)^2 + 2 b^2) - \tran(c-b) - 2\tran(b) 
		\leq
		0
		\eqfs
	\end{equation*}
\end{lemma}
\begin{proof}
	\begin{align*}
		&\trans((c-b)^2 + 2 b^2) - \tran(c-b) - 2\tran(b) 
		\\
		\trans\conc\text{ (subadditive)}
		\qquad&\leq
		\trans((c-b)^2) + \trans(2 b^2) - \tran(c-b) - 2\tran(b) 
		\\
		\qquad&=
		\trans(2b^2) - 2\tran(b) 
		\\
		\trans\conc\text{ (subadditive)}
		\qquad&\leq
		\trans(b^2) + \trans(b^2) - 2\tran(b) 
		\\
		\qquad&=
		0
		\eqfs
	\end{align*}
\end{proof}
\begin{lemma}\label{mech:12}
	Let $\tran\in\setccod3$.
	Let $b,c\in\Rp$.
	Assume $2c \leq 3b$.
	Then
	\begin{equation*}
		\dtran\brOf{\frac32 b} - \dtran\brOf{\frac12 b} - 2b \dtrans(c^2)
		\leq
		0
		\eqfs
	\end{equation*}
\end{lemma}
\begin{proof}
	\begin{align*}
		&\dtran\brOf{\frac32 b} - \dtran\brOf{\frac12 b} - 2b \dtrans(c^2)
		\\
		c \leq \frac32b, \dtrans\decr
		\qquad&\leq
		\dtran\brOf{\frac32 b} - \dtran\brOf{\frac12 b} - 2b \dtrans\brOf{\br{\frac32 b}^2}
		\\
		\qquad&=
		\dtran\brOf{\frac32 b} - \dtran\brOf{\frac12 b} - \frac23 \dtran\brOf{\frac32 b}
		\\
		\qquad&=
		\frac13\dtran\brOf{\frac32 b} - \dtran\brOf{\frac12 b}
		\\
		\frac13\leq 1,\dtran\conc
		\qquad&\leq
		\dtran\brOf{\frac12 b} - \dtran\brOf{\frac12 b}
		\\
		\qquad&=
		0\eqfs
	\end{align*}
\end{proof}
\begin{lemma}\label{mech:13}
	Let $\tran\in\setccod3$.
	Let $b\in\Rp$.
	Then
	\begin{equation*}
		\tran(\frac32 b) - \tran(\frac 12 b) - 2 \tran(b)
		\leq
		0
		\eqfs
	\end{equation*}
\end{lemma}
\begin{proof}
	\begin{align*}
		&\tran(\frac32 b) - \tran(\frac 12 b) - 2 \tran(b)
		\\
		\qquad&=
		\trans(\frac94 b^2) - \trans(\frac 14 b^2) - 2 \tran(b)
		\\
		\trans\conc\text{ (subadditive)}
		\qquad&\leq
		\trans(\frac84 b^2) - 2 \tran(b)
		\\
		\qquad&=
		\trans(2 b^2) - 2 \tran(b)
		\\
		\trans\conc\text{ (subadditive)}
		\qquad&\leq
		\trans(b^2) + \trans(b^2) - 2 \tran(b)
		\\
		\qquad&=
		0
		\eqfs
	\end{align*}
\end{proof}
\begin{lemma}\label{mech:14}
	Let $\tran\in\setccod3$.
	Let $x,b,c\in\Rp$.
	Assume $x \leq c$.
	Then
	\begin{equation*}
		\ddtran(x+b) - \ddtran(x) + 4 b^2\ddtrans(c^2 - 2 x b + b^2)
		\leq
		0
		\eqfs
	\end{equation*}
\end{lemma}
\begin{proof}
	\begin{align*}
		&
		\ddtran(x+b) - \ddtran(x) + 4 b^2\ddtrans(c^2 - 2 x b + b^2)
		\\
		\ddtrans\nega
		\qquad&\leq
		\ddtran(x+b) - \ddtran(x)
		\\
		\ddtran\decr
		\qquad&\leq
		0
		\eqfs
	\end{align*}
\end{proof}
\begin{lemma}\label{mech:15}
	Let $\tran\in\setccod3$.
	Let $u,c\in\Rp$.
	Then
	\begin{equation*}
		2\dtran(u)
		- 2\dtran(u/2)
		- \frac u2 \ddtran(c)
		- u\ddtran(u/2)
		\leq
		0
		\eqfs
	\end{equation*}
\end{lemma}
\begin{proof}
	\begin{align*}
		&
		2\dtran(u)
		- 2\dtran(u/2)
		- \frac u2\ddtran(c)
		- u\ddtran(u/2)
		\\
		\ddtran\posi
		\qquad&\leq
		2\dtran(u)
		- 2\dtran(u/2)
		- u\ddtran(u/2)
		\\
		\text{\cref{lmm:ccpoly} \ref{lmm:ccpoly:dtaylor}}
		\qquad&\leq
		u \ddtran(u/2)
		-
		u \ddtran(u/2)
		\\
		\qquad&=
		0
		\eqfs
	\end{align*}
\end{proof}
\begin{lemma}\label{mech:16}
	Let $\tran\in\setccod3$.
	Let $u\in\Rp$.
	Then
	\begin{equation*}
		u\dtran(u)
		- 2u\dtran(u/2)
		\leq
		0
		\eqfs
	\end{equation*}
\end{lemma}
\begin{proof}
	\begin{align*}
		&
		u\dtran(u)
		- 2u\dtran(u/2)
		\\
		\dtran \smallfrown
		\qquad&\leq
		2u\dtran(u/2)
		- 2u\dtran(u/2)
		\\
		\qquad&=		
		0
		\eqfs
	\end{align*}
\end{proof}
\begin{lemma}\label{mech:17}
	Let $\tran\in\setccod3$.
	Let $b,c\in\Rp$, $s\in[-1,1]$.
	Assume $b\leq sc$.
	Then
	\begin{align*}
		&
		\frac{1-s}2\dtran\brOf{\frac{(1+s)c}2 - b} 
		+ 2(1-s)(b+c)\dtrans( c^2 - 2scb + b^2)
		- 2\dtran\brOf{\frac{(1-s)c}2} 
		+ 
		\\&\hphantom{=}\ \ 
		\frac{1+s}2 \dtran\brOf{\frac{(1+s)c}2} 
		- \dtran(c)
		- b(1-s)\ddtran\brOf{\frac{(1-s)c}2} 
		\\&\leq
		\frac{1-s}2\dtran( c-b)
		+ 2(1-s)(b+c)\dtrans( (c-b)^2)
		- 2\dtran\brOf{\frac{(1-s)c}2} 
		+ 
		\\&\hphantom{=}\ \ 
		-\frac{1-s}2\dtran( c)
		- b(1-s)\ddtran\brOf{\frac{c-b}2}
		\eqfs
	\end{align*}
\end{lemma}
\begin{proof}
	Apply
	\begin{align*}
		\dtran( (s+1)c/2-b) &\leq \dtran( c-b)\eqcm && s\leq 1,\dtran\nearrow\eqcm\\
		\dtrans( c^2 - 2scb + b^2) &\leq \dtrans( (c-b)^2)\eqcm&& s\leq 1,\dtrans\searrow\eqcm\\
		\dtran( (s+1)c/2) &\leq \dtran(c)\eqcm&& s\leq 1,\dtran\nearrow\eqcm\\
		-\ddtran((1-s)c/2) &\leq -\ddtran((c-b)/2)\eqcm&& b\leq sc,\ddtran\searrow
		\eqfs
	\end{align*}
\end{proof}

\end{appendix}
\section*{Acknowledgments}
I want to thank Christophe Leuridan for his very quick answer on math overflow\footnote{\url{https://mathoverflow.net/questions/447718/smooth-approximation-of-nonnegative-nondecreasing-concave-functions/447722}} that sped up the creation of \cref{lmm:cc:molli} significantly.

\printbibliography
\end{document}